\documentclass[10pt]{amsart}
\usepackage{amssymb, amsfonts, latexsym, amsthm, amsmath, eucal, graphics, hyperref}

\theoremstyle{plain}
\newtheorem{theorem}{Theorem}

\newtheorem{lemma}{Lemma}
\newtheorem{cor}[theorem]{Corollary}
\newtheorem{prop}[theorem]{Proposition}

\newtheorem{conjecture}{Conjecture}

\theoremstyle{definition}
\newtheorem{defn}[theorem]{Definition}

\newtheorem{remark}{Remark}

\newtheorem{question}{Question}

\newcommand{\ba}{\backslash}
\newcommand{\Q}{\mathbb{Q}}
\newcommand{\R}{\mathbb{R}}
\newcommand{\N}{\mathbb{N}}
\newcommand{\Z}{\mathbb{Z}}

\newcommand{\ov}{\overline}
\newcommand{\ds}{\displaystyle}
\newcommand{\codim}{\operatorname{codim}}
\newcommand{\length}{\operatorname{length}}
\newcommand{\val}{\operatorname{val}}

\begin{document}
\pagestyle{headings}

\title[Codimensions of Newton Strata for $SL_3(F)$]{Codimensions of Newton Strata for $SL_3(F)$ in the Iwahori Case}
\author{E. T. Beazley}
\address{University of Chicago, Department of Mathematics, 5734 S. University Ave., Chicago, IL 60637}
\email{townsend@math.uchicago.edu}

\begin{abstract}
We study the Newton stratification on $SL_3(F)$, where $F$ is a Laurent power series field.  We provide a formula for the codimensions of the Newton strata inside each component of the affine Bruhat decomposition on $SL_3(F)$.  These calculations are related to the study of certain affine Deligne-Lusztig varieties.  In particular, we describe a method for determining which of these varieties is non-empty in the case of $SL_3(F)$.
\end{abstract}

\maketitle

\begin{section}{Introduction}\label{S:intro}

\renewcommand{\thefootnote}{}

The study of abelian varieties in positive characteristic dates back to Andr\'{e} Weil in the 1940s \cite{Weil},
and was further developed in the 1960s by Barsotti \cite{Bar}.  In \cite{Gro}, Grothendieck describes the theory
of $F$-crystals and Barsotti-Tate groups, or $p$-divisible groups, which are fundamental to the study of
algebraic geometry in characteristic $p >0$.  Isogeny classes of $F$-crystals are indexed by combinatorial
objects called Newton polygons, and these polygons thus naturally provide a stratification on the space of
$F$-crystals.  In the late 1960s, Grothendieck proved his famous specialization theorem, which asserts that Newton polygons ``go down'' under specialization, according to the conventions of this paper.  In particular, the set of points for which the associated Newton polygons lie below a given Newton polygon is Zariski closed.  Grothendieck conjectured the converse to his theorem in 1970, which says that given a $p$-divisible group $G_0$ having Newton polygon $\gamma$, then for any $\beta$ lying above $\gamma$, there exists a deformation of $G_0$ whose generic fiber has Newton polygon equal to $\beta$. 
This conjecture was proved by Oort a quarter of a century later (see \cite{dJO}, \cite{Onp&formalgps}, \cite{Onpinmoduli}).
In the late 1970s, Katz extended these ideas of Grothendieck to Newton polygons associated to families of
$F$-crystals \cite{Kat}.

By the early 1980s, work on the moduli problem for abelian varieties of fixed dimension in
positive characteristic was well underway, having been outlined by Mumford and applied by Norman and Oort
\cite{NO}.  Analogs of Grothendieck's specialization theorem and its converse were formulated and proved in the context of (moduli spaces of) abelian varieties (see \cite{Man}, \cite{Ta}, \cite{Kob}, \cite{Omoduli&NP}, \cite{Onp&formalgps}, \cite{Onpinmoduli}, \cite{Onp&pdiv}).  Oort has also formulated several conjectures and results about irreducibility and the dimensions and number of
the components, which generalize his work with Li on the supersingular case (see \cite{Omoduliposchar}, \cite{Onpinmoduli}, \cite{LiO}).  In \cite{dJO}, de Jong and Oort prove a purity result that gives an estimate for the codimensions of
the Newton polygon strata in moduli spaces of $p$-divisible groups.  Viehmann computes the dimensions and the number of connected
and irreducible components of moduli spaces of $p$-divisible groups in \cite{VieGlobal} and \cite{VieModuli}.  More
information is known about the structure of the Newton strata and the poset of slopes of Newton polygons in the
special case of the Siegel moduli space (see \cite{Onpinmoduli}, \cite{WedCong}).  Recent work by Harashita extends some of
Oort's work determining when certain Newton strata are non-empty \cite{Har}.  There has also been recent
interest in the foliation structure on the Newton polygon strata (see \cite{Ofoliations}, \cite{AG}).  For a survey of
both classical and newer results on the properties of the Newton stratification for moduli spaces of abelian
varieties in positive characteristic, see \cite{Omoduliposchar}, \cite{vdGO}, and the recent survey article by Rapoport \cite{RapNewton}.

In the mid-1990s, Rapoport and Richartz generalized Grothendieck's specialization theorem and the notion of the
Newton stratification to $F$-isocrystals with $G$-structure, where $G$ is a reductive group over a discretely
valued field \cite{RR}.  These generalized Newton strata are indexed by $\sigma$-conjugacy classes, where
$\sigma$ is the Frobenius automorphism.  There is a natural bijection between the set of $\sigma$-conjugacy
classes and a suitably generalized notion of the set of Newton polygons, which was described by Kottwitz in \cite{KotIsoI} and \cite{KotIsoII}.  The poset of Newton
polygons in the context of reductive group theory has interesting combinatorial and Lie-theoretic
interpretations, which were described by Chai in \cite{Ch}.  In particular, Chai generalizes the work of Li and
Oort to the case of $F$-isocrystals with $G$-structure, where $G$ is any quasisplit reductive group over a
non-Archimedean local field $F$, proving that the poset of Newton slope sequences is catenary; \textit{i.e.},
any two maximal chains have the same length.  Chai also provides a root-theoretic formula for the expected
dimension of the Newton strata.  Rapoport lists several other conjectures in the context of these generalized
Newton strata in \cite{RapNewton}.

Chai was primarily interested in applications to the reduction modulo $p$ of a Shimura variety, although his
results are more general.  Interest in the special case of Shimura varieties has roots in implications toward
the local Langlands conjecture, as demonstrated by Harris and Taylor \cite{HT}.  Similar topological and
geometric questions have been answered about the Newton stratification on Shimura varieties.  Wedhorn has
demonstrated that these strata are locally closed, and he provides a codimension formula in \cite{WedDim}. In
addition, he showed that the ordinary locus is open and dense \cite{WedOrdinariness}.  B\"{u}ltel and Wedhorn have studied
the relationship among the Newton polygon stratification, the Ekedahl-Oort stratification, and the final
stratification in \cite{BW}.  Recent work by Yu generalizes considerations of Oort's Siegel case to type $C$
families of Shimura varieties \cite{Yu}. Again, for a comprehensive overview of known and conjectured results on
Shimura varieties, see Rapoport's survey \cite{RapShimura}. Haines has also written an article on Shimura varieties
that provides an introduction to the field through the theory of local models \cite{Hai}.

The notion of the Newton stratification has since arisen in many other contexts.  Goren and Oort discuss the
relationship between the Newton polygon strata and the Ekedahl-Oort strata for Hilbert modular varieties in
\cite{GO}.  Vasiu studies latticed $F$-isocrystals over the field of Witt vectors in \cite{Vas}, in which he proves a
purity property on the Newton polygon stratification. Blache and F\'{e}rard explicitly describe the Newton stratification for polynomials over finite
fields in their recent work \cite{BF}. In \cite{KotNewtStrata}, Kottwitz describes the Newton stratification in the adjoint
quotient of a reductive group. Goresky, Kottwitz, and MacPherson discuss the root valuation strata in Lie
algebras over Laurent power series fields \cite{GKM}, and while this stratification is not a Newton
stratification per se, the techniques employed are reminiscent of those that arise in the situation of a Newton
stratification.

In all of the aforementioned contexts, there are several common topological, geometric, and combinatorial
themes.  The goal of this paper is to address these common themes in the specific context of the Newton
stratification on the algebraic group $G=SL_3(F)$ in the so-called Iwahori case. Here, $F$ is the field of
Laurent power series over an algebraic closure of a finite field.  We begin by reviewing the theory of
isocrystals over $F$ and the associated Newton stratification on $SL_3(F)$.  In order to develop a topology and
notions of irreducibility and codimension, in Section \ref{S:admis} we define admissible subsets of Iwahori (double) cosets of $SL_3(F)$,
which are sets that satisfy a property analogous to Vasiu's Crystalline boundedness principle \cite{Vas}.  We
then compute explicit equations that characterize the isocrystals having Newton polygons lying below
a given polygon.  These equations turn out to be polynomial in finitely many coefficients of the entries of a given $g \in G$, which guarantees admissibility.   Our main theorem is a sort of purity result.  Namely, we
show that the codimension between adjacent strata jumps by one.  We provide two versions of a formula for the
codimensions of the Newton strata inside each component of the affine Bruhat decomposition on $G =
I\widetilde{W}I$, where $I$ is the Iwahori subgroup and $\widetilde{W}$ the affine Weyl group.  Following Chai,
we present both root-theoretic and combinatorial versions of this codimension formula, and we provide the combinatorial version here.  This theorem appears as Theorem \ref{T:main} in Section \ref{S:thm}.

\vspace{5pt}
\noindent \textbf{Theorem}
\textit{Let $G=SL_3(F)$ and fix $x \in \widetilde{W}$.  For a Newton slope sequence $\lambda \in \mathcal{N}(G)_x$, the
subset $(IxI)_{\leq \lambda}$ of $IxI$ is admissible, and} \begin{equation*} \codim\left(
(IxI)_{\leq \lambda} \subseteq  IxI \right) =\length_{\mathcal{N}(G)_x}[\lambda, \nu_x].\end{equation*} \textit{Moreover, the closure of a given Newton stratum $(IxI)_{\lambda}$ in $IxI$ is precisely $(IxI)_{\leq \lambda}$.  Therefore, for two Newton polygons $\lambda_1< \lambda_2$ which are adjacent in the poset $\mathcal{N}(G)_x$,} \begin{equation*} \codim\left( (IxI)_{\leq \lambda_1} \subset (IxI)_{\leq \lambda_2} \right) = 1. \end{equation*} For a definition of the poset $\mathcal{N}(G)_x$, see Section \ref{S:gpthy}, although we also discuss $\mathcal{N}(G)_x$ later in this introduction.  We also introduce the Newton strata $(IxI)_{\lambda}$ and $(IxI)_{\leq \lambda}$ in Section \ref{S:gpthy}.  We denote by $\nu_x$ the generic slope sequence in $\mathcal{N}(G)_x$, which both appears later in the introduction and is formally defined in Section \ref{S:generic}.  The length of the segment $[\lambda,\nu_x]$ is defined in Section \ref{S:length}.

The calculations performed in Sections \ref{S:valcalc} and \ref{S:codimcalc} provide a very concrete description of the closed Newton strata $(IxI)_{\leq \lambda}$.  These descriptions vary greatly with $x$, but there are some things that can be said regarding their geometric structure.  First, the sets $(IxI)_{\leq \lambda}$ are often, but not always, irreducible.  In the cases in which the closures of the Newton strata are irreducible, they are fiber bundles over an irreducible affine scheme, having irreducible fibers (see Section \ref{S:thmproof}).  For certain $x$ the structure of $(IxI)_{\leq \lambda}$ is a bit more complicated.  When $(IxI)_{\leq \lambda}$ is reducible, each irreducible component is the closure of a fiber bundle over an irreducible affine scheme whose fibers are themselves irreducible affine schemes.  We see in Section \ref{S:thmproof}, however, that the space $(IxI)_{\leq \lambda}$ is equicodimensional; \textit{i.e.}, all of the irreducible components have the same codimension inside $IxI$.  Since the geometric structure of the Newton strata varies so greatly with $x$, we do not make formal statements regarding irreducibility or equicodimensionality, although we provide informal discussion when we are able.  Independent of whether or not the space $(IxI)_{\leq \lambda}$ is irreducible, we shall see that it is precisely the closure of the Newton stratum $(IxI)_{\lambda}$, which is a locally closed subset of $IxI$.  This fact demonstrates, in particular, that these Newton strata satisfy the strong stratification
property; \textit{i.e.} the closure of a given stratum is the union of locally closed subsets defined by all
Newton polygons which lie below that of the given stratum and are comparable in the poset $\mathcal{N}(G)_x$.  Theorem \ref{T:main} then says that for two Newton polygons that are adjacent in the poset $\mathcal{N}(G)_x$, the codimension of the smaller stratum in the closure of the larger is always equal to 1.  

The Newton strata associated to Shimura varieties are related to the
study of certain affine Deligne-Lusztig varieties.  Rapoport's dimension
formula for affine Deligne-Lusztig varieties inside the affine Grassmannian was proved in two steps by G\"{o}rtz, Haines, Kottwitz, and Reuman \cite{GHKR} and
Viehmann \cite{VieDim}.  Viehmann has also described the set of connected components of these affine Deligne-Lusztig varieties \cite{VieConncpt}.  Our study of the Newton strata in the Iwahori case relates to certain affine
Deligne-Lusztig varieties inside the affine flag manifold, about which less is known. To every $\sigma$-conjugacy class in $G$ and every
element of the affine Weyl group, we can associate the affine Deligne-Lusztig variety \begin{equation*} X_x(b) :=
\{ g \in G(F)/I : g^{-1}b\sigma(g) \in IxI\}. \end{equation*}  For a fixed $x \in \widetilde{W}$, we explicitly describe in Section \ref{S:slopes}
the poset $\mathcal{N}(G)_x$ of Newton polygons that arise for elements in the double coset $IxI$, when $G = SL_3(F)$. As a direct consequence, we can answer the question of which of the associated affine
Deligne-Lusztig varieties are non-empty.  This result was proved by Reuman for the
case $b=1$ \cite{Reu}, but our results may be applied to any $b\in G$.  We believe that it might be possible to apply arguments similar to those we employ in this paper to
obtain a description for the poset $\mathcal{N}(G)_x$, at least in the case of $G=SL_n(F)$ (see Question
\ref{T:nonemptyQ}, Section \ref{S:adlv}).

In addition to giving information about non-emptiness of certain affine Deligne-Lusztig varieties, our work
gives rise to a conjectural relationship between the dimension of these varieties and the codimensions of the
associated Newton strata.  \begin{conjecture}\label{T:dimconj} Let $x\in \widetilde{W}$ be an element of the
affine Weyl group and $b \in G$.  The relationship between the dimension of the affine Deligne-Lusztig variety
$X_x(b)$ and the codimension of the associated set of Newton strata is given by \begin{equation}\label{E:dimconj} \dim
X_x(b) + \codim((IxI)_{\leq \ov{\nu}(b)} \subseteq IxI) = \ell(x) - \langle 2\rho, \ov{\nu}(b)\rangle.
\end{equation}  Here, $\ell(x)$ is the length of $x$, $\rho$ is the half-sum of the positive roots, and
$\ov{\nu}(b)$ is the Newton slope sequence associated to $b$.  \end{conjecture} \noindent An earlier version of
this conjecture, due to Kottwitz, assumed that the codimensions of the Newton strata in the Iwahori case were
given by the same root-theoretic expression appearing in \cite{Ch}, and combined this with the right-hand side of \eqref{E:dimconj} to provide a
conjecture for the dimension of $X_x(b)$.  We demonstrate in Corollary \ref{T:maincor} that for certain affine
Weyl group elements, we must correct this initial guess for $\codim((IxI)_{\leq \ov{\nu}(b)} \subseteq IxI)$ by -1.  In fact, Theorem
\ref{T:main} suggests that the codimensions of the Newton strata are most conveniently expressed combinatorially
in terms of lengths in the poset of Newton slopes, rather than root-theoretically.  Conjecture
\ref{T:dimconj} accounts for this fact and thus includes the codimension as a separate term.  In the case $G=SL_3(F)$ and $b=1$, Conjecture
\ref{T:dimconj} is true, as \cite{Reu} and our work demonstrate.  We refer the reader to \cite{GHKR} for
conjectural formulas for the dimensions of the varieties $X_x(b)$ for more general $G$. Although neither these
dimensions nor the codimensions of the Newton strata are known except in a few cases, the right-hand side of the equality in Conjecture \ref{T:dimconj} is simple and can be easily computed in general.

The study of the poset of Newton slope sequences for $SL_3(F)$ has led to the several additional combinatorial
questions.  In the Iwahori case, the problem of determining the Newton slope sequence associated
to the open stratum remains unsolved.  Even in the case of $SL_3(F)$, in which we explicitly compute this
generic slope sequence $\nu_x$ for every $x \in \widetilde{W}$, we cannot yet do better than providing a list.  We
expect there to be a closed root-theoretic formula that would provide an analog of Mazur's inequality in the
Iwahori case (see Question \ref{T:nuxform}, Section \ref{S:generic}).  When $G=SL_n(F)$ and $x$ corresponds to
an alcove lying in the dominant Weyl chamber, we can formulate a precise conjecture determining the generic
slope.
\begin{conjecture}\label{T:domconj}
For $G = SL_n(F)$, let $x=\pi^{\mu}w \in \widetilde{W}$ be such that $\mu$ is dominant; \textit{i.e.}, $\langle \alpha_i,\mu\rangle \geq 0$ for all simple roots $\alpha_i$ in $\text{Lie}(G)$, or equivalently that $\mu_1\geq \cdots \geq \mu_n$.
Then $\nu_x = -\mu_{\text{dom}}$.
\end{conjecture}
\noindent We point out that the negative sign appears in front of $\mu_{\text{dom}}$ as a result of making several less traditional conventions in our definitions (see Section \ref{S:gpthy}) which make our main calculations easier.  In general, the generic slope sequence $\nu_x$ does not necessarily coincide with the unique dominant element in the Weyl orbit of $-\mu$.  Rather, there is frequently a correction term needed, which comes in the form of a sum of simple coroots, at least for $x$ lying inside the so-called ``shrunken'' Weyl chambers (see \cite{Reu}).  Experimentation with groups of higher rank confirms that this observation should be able to be made precise for general $G$, although a closed formula is not yet obvious.

  Already in the case where $G=SL_3(F)$, the poset of Newton slope sequences has both interesting and surprising
properties.  We prove, for example, that the poset $\mathcal{N}(G)_x$ for $SL_3(F)$ is a ranked lattice, and we
expect that this might be the case for general $G$ (see Question \ref{T:rklattice}, Section \ref{S:length}). One
might guess that an analog of Manin's converse to the specialization theorem holds in the Iwahori case as well;
\textit{i.e.}, that $\mathcal{N}(G)_x$ consists of all possible Newton slope sequences lying below the generic
one.  \begin{figure}[h] \centering
\includegraphics{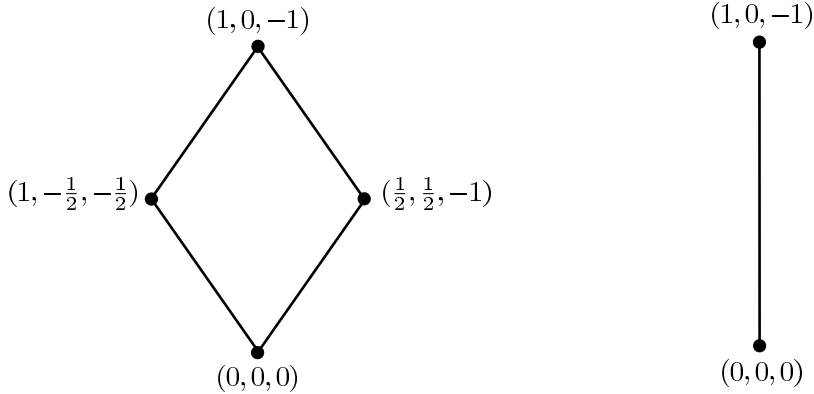}
\caption{The posets $\{\lambda \in \mathcal{N}(G) \mid \lambda \leq (1,0,-1)\}$ and $\mathcal{N}(G)_x$ for
$x=\pi^{(-2,0,2)}s_{121}$}\label{fig:poset}
\end{figure} We shall demonstrate that this suspicion is actually false. Consider the example in which $x \in
\widetilde{W}$ has finite Weyl part equal to the longest element in the Weyl group $s_{121}$, and translation
part equal to $(-2,0,2)$.  We show that the generic slope sequence is given by $(1,0,-1)$, in which case the
lattice consisting of all possible slope sequences in $\mathcal{N}(G)$ lying below $(1,0,-1)$ is given by the
poset on the left in Figure \ref{fig:poset}.  The actual description of the poset $\mathcal{N}(G)_x$ in this
example consists, however, of only the two elements $(1,0,-1)$ and $(0,0,0)$; see the picture on the right in
Figure \ref{fig:poset}.  Therefore, the length of the segment $[(0,0,0), (1,0,-1)]$ inside $\mathcal{N}(G)_x$
equals one, even though the length inside $\mathcal{N}(G)$ is two.  Moreover, the codimension of the stratum
associated to $(0,0,0)$ also equals one, since the slope sequences $(0,0,0)$ and $(1,0,-1)$ are adjacent in the poset $\mathcal{N}(G)_x$.  The author believes that it is possible to characterize the affine
Weyl group elements that produce these strange examples using the language of parabolic and Levi subgroups,
although we are not yet prepared to formulate a precise conjecture.

\vspace{5pt}
\noindent \textbf{Acknowledgments.}
The author would like to thank Robert Kottwitz for suggesting this problem, for numerous helpful comments on earlier versions of this paper, and for his unparalleled dedication as an advisor.  Eva Viehmann also provided several useful suggestions for improving the introduction.  The author also thanks the anonymous referee for facilitating important structural and technical improvements.

\subsection{Isocrystals over the discretely valued field $F$}\label{S:cyclic}
Let $k$ be a finite field with $q$ elements, and let $\overline{k}$ be an algebraic closure of $k$.  Denote by $\pi$ the uniformizing element of the discrete valuation ring $\mathcal{O}:= \ov{k}[[\pi]]$, having fraction field $F:=\overline{k}((\pi))$ and maximal ideal $P:= \pi \mathcal{O}$.  Normalize the valuation homomorphism $\val: F^{\times} \rightarrow \Z$ so that $\val(\pi) = 1$.  We can extend the usual Frobenius automorphism $x \mapsto x^q$ on $\ov{k}$ to a map $\sigma: F \rightarrow F$ given by $\sum a_i\pi^i \mapsto \sum a_i^q\pi^i$.

Recall that an isocrystal $(V,\Phi)$ is a finite-dimensional vector space $V$ over $F$ together with a
$\sigma$-linear bijection $\Phi: V\rightarrow V$; \textit{i.e.}, $\Phi(av)=\sigma(a)\Phi(v)$ for $a \in
F$ and $v \in V$. We now define a ring $R = F[\sigma]$, where any element $y\in R$ is of the form $y=\sum\limits a_i\sigma^i$, for $a_i \in F$.  Note that $R$ is not a polynomial ring in the usual sense, since for $a \in F$, we have that $\sigma a = \sigma(a) \sigma$.  Given an isocrystal $(V,\Phi)$ over $F$, defining $\sigma^iv:=\Phi^i(v)$ makes $V$ into an $R$-module.  We then have the following well-known
proposition describing isocrystals as cyclic modules over the ring $R$.

\begin{prop}\label{T:Cyclic}
Let $(V,\Phi)$ be an isocrystal and $R = F[\sigma]$, as above. Then $V$ is a cyclic $R$-module; \textit{i.e.},
$Rv = V$ for some $v$ in $V$.
\end{prop}

In the context of Proposition \ref{T:Cyclic}, we call the generator $v$ a \textit{cyclic vector}.  The ring $R$
is non-commutative, but there exist both a right and left division algorithm, whence we may conclude that $R$ is
a principal ideal domain.  Upon choosing a cyclic vector, we may thus write $V \cong R/Rf$ for some $f =
\sigma^n + \cdots + a_{n-1}\sigma + a_n \in R$, where $n=\mbox{dim}_F(V)$.   We shall call $f$ the
\emph{characteristic polynomial} associated to the isocrystal $(V,\Phi)$. Note, however, that $f$ depends on the
choice of a cyclic vector.  Consequently, we shall be interested in the Newton polygon associated to $f$, an
isocrystal invariant that is independent of the choice of cyclic vector.

The Newton polygon of $f$, or equivalently the Newton polygon of $(V, \Phi)$, is defined to be the convex hull
of the set of points $\{(0,0),\ (i, -\val(a_i))\mid i = 1, 2, \dots, n\}$, where the $a_i$ are the
coefficients of $f$. More specifically, the Newton polygon of $f$ is the tightest-fitting convex polygon joining the points
$(0,0)$ and $(n,-\val(a_n))$ that passes either through or above all of the points in the set
$\{(0,0),(i, -\val(a_i))\}$.  The reader should observe that our definition
of the Newton polygon differs from the usual one, in which the polygon is formed from the set of points
$\{(0,0), (i,\val(a_i))\}$.  We adopt the less conventional construction in order that our definitions for the Newton stratification in Section \ref{S:strata} agree with those in other related contexts.  The associated Newton slope sequence is the $n$-tuple $\lambda = (\lambda_1, \dots, \lambda_n) \in \Q^n$, where
the $\lambda_i$ are the slopes of the edges of the Newton polygon, repeated with multiplicity and ordered such
that $\lambda_1 \geq \cdots \geq \lambda_n$.  Occasionally we will wish to move freely between a slope sequence $\lambda$ and the Newton polygon having $\lambda$ as its slope sequence, which we shall denote by $N_{\lambda}$.

\subsection{The characteristic polynomial for $GL_3(F)$}\label{S:charpoly}

If we fix a basis for the $n$-dimensional vector space $V$, the isocrystal $(V, \Phi)$ is isomorphic to one the form $(F^n, A\sigma)$ for some $A \in GL_n(F)$.  In this context, $\sigma(v)$ means that we apply $\sigma$ to each component of the vector $v$.

In this section, we specialize to 3-dimensional isocrystals.  In order that we may work with matrices in our calculations, we fix a basis.  Let $e_1, e_2, e_3$ denote the standard basis vectors for $F^3$.

\begin{prop}\label{T:Dneq0}
Let $\Phi =A\sigma$, where $A=\begin{pmatrix} a& b & c\\ d& e & f\\g& h& i\end{pmatrix} \in GL_3(F)$.  Then
$e_1$ is a cyclic vector for $(F^3, \Phi)$ if and only if
\begin{equation*}D := \sigma(d)\begin{vmatrix}d & e\\g&h\end{vmatrix}+ \sigma(g)\begin{vmatrix}d&f\\g&i\end{vmatrix}\neq 0.\end{equation*}
\end{prop}

\begin{proof}
Compute that $e_1 \wedge \Phi(e_1) \wedge \Phi^2(e_1) = D(e_1 \wedge e_2
\wedge e_3)$.
\end{proof}

We shall now assume that the hypotheses of Proposition \ref{T:Dneq0} are met so that $e_1$ is a cyclic vector for $(F^3, \Phi)$. The characteristic polynomial of $(F^3, \Phi)$ is of the form $f:= \sigma^3 +
\alpha\sigma^2 + \beta\sigma + \gamma = 0$, for some $\alpha, \beta, \gamma \in F$.  Using linear algebra, we
calculate that $\alpha$ and $\beta$ are determined by $\Phi$ as follows:
\begin{equation}\label{E:mateqn}
\begin{pmatrix}\alpha\\ \beta\end{pmatrix} = -\begin{pmatrix} \Phi^2(e_1)e_2 & \Phi(e_1)e_2\\ \Phi^2(e_1)e_3 & \Phi(e_1)e_3 \end{pmatrix}^{-1} \begin{pmatrix} \Phi^3(e_1)e_2 \\ \Phi^3(e_1)e_3 \end{pmatrix}.
\end{equation}
Here we denote by $\Phi^k(e_i)e_j$ the coefficient of $e_j$ in the vector $\Phi^k(e_i)$.  We can then solve matrix equation \eqref{E:mateqn} explicitly for $\alpha$ and $\beta$
to obtain the following:\begin{equation*}
\displaystyle \alpha = -\sigma^2(a) - \frac{1}{D}\left(
(\sigma^2(d)\sigma(e)+\sigma^2(g)\sigma(f))\begin{vmatrix} d&e\\g&h\end{vmatrix} +
(\sigma^2(d)\sigma(h)+\sigma^2(g)\sigma(i))\begin{vmatrix}d&f\\g&i\end{vmatrix}\right),
\end{equation*}
\begin{equation*}
\displaystyle \beta = -(\sigma(a)\alpha + \sigma^2(a)\sigma(a)+\sigma^2(d)\sigma(b)+\sigma^2(g)\sigma(c)) +
\frac{\sigma(D)}{D}\begin{vmatrix}e&f\\h&i\end{vmatrix}.
\end{equation*}
Here we observe directly the need for the hypothesis $D \neq 0$.  

So that formulas like the previous ones for $\alpha$ and $\beta$ appear less complicated, we define
$\overline{x} := \sigma(x)$.  In this notation, our formulae for $\alpha$ and $\beta$ then become \begin{equation}\label{E:alpha}
\displaystyle \alpha = -\overline{\overline{a}} - \frac{1}{D}\left(
(\ov{\ov{d}}\ov{e}+\ov{\ov{g}}\ov{f})\begin{vmatrix} d&e\\g&h\end{vmatrix} +
(\ov{\ov{d}}\ov{h}+\ov{\ov{g}}\ov{i})\begin{vmatrix}d&f\\g&i\end{vmatrix}\right),
\end{equation}
\begin{equation*}\label{E:beta1}
\displaystyle \beta = -(\ov{a}\alpha + \ov{\ov{a}}\ov{a}+\ov{\ov{d}}\ov{b}+\ov{\ov{g}}\ov{c}) +
\frac{\ov{D}}{D}\begin{vmatrix}e&f\\h&i\end{vmatrix}.
\end{equation*}
To make our calculations in Section
\ref{S:valcalc} less cumbersome, let us also introduce the following notation:
\begin{equation*}B_1:= \frac{\ov{a}}{D} \left( (\ov{\ov{d}}\ov{e}+\ov{\ov{g}}\ov{f})\begin{vmatrix} d&e\\g&h\end{vmatrix}\right) \hspace{20pt} B_2:= \frac{\ov{a}}{D}\left((\ov{\ov{d}}\ov{h}+\ov{\ov{g}}\ov{i})\begin{vmatrix}d&f\\g&i\end{vmatrix}\right)
\end{equation*}
so that
\begin{equation}\label{E:beta}
\beta = B_1 +B_2-\ov{\ov{d}}\ov{b}-\ov{\ov{g}}\ov{c} + \frac{\ov{D}}{D}\begin{vmatrix}e&f\\h&i\end{vmatrix}.
\end{equation}

Finally, recall from the proof of Proposition \ref{T:Dneq0} that $\Phi(e_1 \wedge \Phi(e_1) \wedge \Phi^2(e_1))
=\linebreak \ov{D}\Phi(e_1\wedge e_2 \wedge e_3) = \ov{D}\det A(e_1\wedge e_2\wedge e_3)$, where $\Phi=A\sigma$.  On the
other hand, using that $(\Phi^3+\alpha \Phi^2 + \beta \Phi + \gamma)(e_1) = 0$, we compute that $\Phi(e_1
\wedge \Phi(e_1) \wedge \Phi^2(e_1)) = -\gamma D(e_1 \wedge e_2 \wedge e_3)$. Equating these two expressions
yields
\begin{equation*}\gamma = -\frac{\ov{D}}{D}\det A. \end{equation*}  One should note that the method used to calculate $\gamma$ generalizes from $GL_3(F)$ to $GL_n(F)$.  We shall use
these explicit formulae for the coefficients of the characteristic polynomial to make calculations in Section
\ref{S:valcalc}.

\subsection{The Newton stratification}\label{S:strata}

Let $G = SL_3(F)$.  If $A \in G$, we have that $\val(\det A) = 0$ and thus $\gamma \in
\mathcal{O}^{\times}$. As discussed in Section \ref{S:cyclic}, the Newton polygon of $f$ is formed from the set
$\{ (0,0), (1, -\val(\alpha)), (2, -\val(\beta)), (3,0)\}$.  We define $\overline{\nu}(A)$ to be
the slope sequence of the Newton polygon associated to the isocrystal $(F^3, A\sigma)$.  Again, recall that our definition for $\ov{\nu}$ differs from the conventional one.  For example, if $\val(x) \leq \val(y) \leq \val(z)$, we have $\ov{\nu}(\text{diag}(x,y,z)) = (-\val(x), -\val(y), -\val(z))$.  The map
$$\overline{\nu}: G \longrightarrow \mathcal{N}(G)$$ induces a bijection  $B(G) \longleftrightarrow
\mathcal{N}(G)$, see \cite{KotIsoI}.  Here, $B(G)$ is the set of $\sigma$-conjugacy classes of $G$; \textit{i.e.},
$B(G) = G(F)/\sim$ where $x \sim y \iff x = gy\sigma(g)^{-1}$ for some $g \in G(F)$. We denote by $\mathcal{N}(G)$
the set of possible slope sequences arising from Newton polygons for isocrystals of the form $(F^3, A\sigma)$
with $A \in G$.

The map $\ov{\nu}$ induces a natural stratification on $G$ indexed by the elements of $\mathcal{N}(G)$.  We define the Newton strata referred to in the title of this paper to be \begin{equation*} G_{\lambda} := \{ g \in G \mid \ov{\nu}(g) = \lambda\}.\end{equation*} The group $G$ then breaks up into a disjoint union of these strata as follows: \begin{equation*} G = \coprod_{\lambda \in \mathcal{N}(G)} G_{\lambda}. \end{equation*}

The set $\mathcal{N}(G)$ is actually a partially ordered set.  We define $\lambda' \leq \lambda$ if $N_{\lambda'}$ and $N_{\lambda}$ have the same endpoints and all edges of $N_{\lambda'}$ lie on or below the corresponding edges of $N_{\lambda}$.
Given a particular $\lambda \in \mathcal{N}(G)$, we are interested in studying all strata $G_{\lambda '}$ such that $\lambda' \leq \lambda$. The closed subset of $G$ determined by a slope sequence $\lambda$ is \begin{equation*} G_{\leq \lambda}:= \coprod_{\lambda' \leq \lambda} G_{\lambda '}.\end{equation*}

\subsection{Lie-theoretic interpretation}\label{S:gpthy}

Let $B \subset G=SL_3(F)$ denote the Borel subgroup consisting of the upper triangular matrices, and $T$ the maximal torus consisting of all diagonal matrices. Let $W$ denote the Weyl group of $T$ in $G$, which is isomorphic to the
symmetric group $S_3$ in this case. Let $\mathfrak{a} := X_*(T)\otimes_{\Z}\R$, and denote its $\Q$-subspace by
$\mathfrak{a}_{\Q}:=X_*(T)\otimes_{\Z}\Q$.  Denote by $\alpha_i$ the simple roots in $\text{Lie}(G)$, and let $C
:= \{\lambda \in \mathfrak{a}\ \mid \langle \alpha_i, \lambda \rangle > 0, \ \forall i\}$ denote the dominant Weyl
chamber. Analogously, denote by $C^0:=\{\lambda \in \mathfrak{a}\ \mid \langle \alpha_i, \lambda \rangle < 0, \
\forall i\}$ the antidominant Weyl chamber.  Our convention will be to call the unique alcove in $C^0$ whose
closure contains the origin the base alcove $\mathbf{a}_1$.  (This convention makes our calculations in Section
\ref{S:valcalc} less cumbersome, and our less traditional definition of the Newton polygon in Section \ref{S:cyclic} aligns more naturally with this convention.) Let $I$ be the associated Iwahori subgroup of $G(F)$. According to our conventions, $I$ is the standard Iwahori subgroup $$I = \begin{pmatrix} \mathcal{O}^{\times} & \mathcal{O} & \mathcal{O} \\ P & \mathcal{O}^{\times} & \mathcal{O}\\ P & P & \mathcal{O}^{\times} \end{pmatrix}.$$  One can also
consider $\mathbf{a}_1$ to be the basepoint of the affine flag manifold $G/I$.

Denote by $\widetilde{W} = X_*(T)\rtimes W$ the affine Weyl group. We shall express an element $x\in
\widetilde{W}$ as $x = \pi^{\mu}w$, for $\mu \in X_*(T)$ and $w \in W$. For $G = GL_3(F)$, we may identify $X_*(T)$ with $\Z^3$.  We
then write $\pi^{\mu} = \text{diag}(\pi^{\mu_1},\pi^{\mu_2},\pi^{\mu_3})$ for $\mu = (\mu_1, \mu_2,
\mu_3) \in \Z^3$.  In this group-theoretic context, we can interpret a Newton slope sequence $\lambda =
(\lambda_1,\lambda_2, \lambda_3)$ as an element $\lambda \in \mathfrak{a}_{\Q,\text{dom}}$, where
$\mathfrak{a}_{\Q,\text{dom}}$ denotes the dominant elements in $\mathfrak{a}_{\Q}$.  Specifically, $\mathfrak{a}_{\Q,\text{dom}} = \{(\nu_1, \nu_2, \nu_3) \in \Q^3 \mid \nu_1\geq \nu_2 \geq \nu_3\}$.  For $G = SL_3(F)$, our description of the group of cocharacters is $X_*(T) \cong \{\mu \in \Z^3 \mid \sum \mu_i = 0\}$.  In this case, $\mathfrak{a}_{\Q,\text{dom}} = \{(\nu_1, \nu_2, \nu_3)
\in \Q^3 \mid \nu_1\geq \nu_2 \geq \nu_3\ \text{and}\ \nu_1+\nu_2+\nu_3=0\}$.  Under these identifications, the partial order on $\mathcal{N}(G)$ then becomes $\lambda' \leq \lambda \iff \lambda - \lambda'$ is a non-negative linear combination of positive coroots.

Recall the affine Bruhat decomposition for $G=SL_3(F)$: \begin{equation*} G =
\coprod_{x \in \widetilde{W}} IxI.\end{equation*}  We study the sets $G_{\leq
\lambda}$ when intersected with these double cosets $IxI$ in order that we may define a notion of codimension.  For a fixed $x \in \widetilde{W}$, we thus introduce
the following analog of the Newton strata discussed in Section \ref{S:strata}:
\begin{equation*} (IxI)_{\lambda}:= G_{\lambda} \cap IxI \end{equation*}
\begin{equation*}
(IxI)_{\leq \lambda}:= \coprod\limits_{\lambda' \leq \lambda} (IxI)_{\lambda'}. \end{equation*} The subset $(IxI)_{\leq \lambda}$ consists of all $g \in
IxI$ such that the Newton polygon associated to $(F^3, g\sigma)$ has the same endpoints as $N_{\lambda}$ and lies on or below $N_{\lambda}$.

The
stratum $(IxI)_{\lambda}$ is non-empty for only finitely many $\lambda \in \mathcal{N}(G)$.  It will be useful to introduce notation for the finitely many Newton slope sequences that actually arise for elements inside a particular Iwahori double coset: \begin{equation*} \mathcal{N}(G)_x := \{ \lambda \in \mathcal{N}(G) \mid (IxI)_{\lambda} \neq \emptyset \}. \end{equation*} The subset $\mathcal{N}(G)_x$ inherits the partial ordering $\leq$ on $\mathcal{N}(G)$.

\subsection{Admissibility of $(IxI)_{\lambda}$}\label{S:admis}

The double cosets $IxI$ are not finite dimensional; however, we can develop an adequate notion of
codimension by working with finite dimensional quotients of $IxI$ such as $IxI/I^N$, where $I^N:= \{ g \in I \mid g \equiv \text{id}\mod(P^N)\}$.  We obtain another finite dimensional quotient of $IxI$ by considering its image under the map on $3 \times 3$ matrices induced by $F \rightarrow F/P^N$.  By abuse of notation, we denote this image by $IxI/P^N$.

Following \cite{GKM}, denote by $p_N$ and $\rho_N$ the surjections $p_N:IxI \twoheadrightarrow IxI/I^N$ and $\rho_N:IxI \twoheadrightarrow IxI/P^N$.  Observe that the quotients $IxI/I^N$ and $IxI/P^N$ are finite dimensional affine schemes.  We say that a subset $Y$ of $IxI$ is \emph{admissible} if there exists an
integer $N$ such that $Y = p_N^{-1}p_NY$.  Note that $Y$ is admissible if and only if there exists an integer $M$ such that $Y = \rho_M^{-1}\rho_MY$.  Since these two notions of admissibility are equivalent, we will use whichever is most convenient for us in the given context.

If a subset $Y$ of $IxI$ is admissible, we can treat $Y$ as though it is
finite-dimensional. In particular, we define the \emph{codimension} of $Y$ in $IxI$ to
be the codimension of $p_NY$ in $IxI/I^N$ for any $N$ such that $Y =
p_N^{-1}p_NY$.  Similarly, we say that $Y$ is \emph{irreducible} (resp. open, closed, locally closed) in $IxI$ if $p_NY$ is irreducible (resp. open, closed, locally closed) for some $N$ such that $Y = p_N^{-1}p_NY$.  Note that we may replace $p_N$ by $\rho_N$ and $I^N$ by $P^N$ to obtain equivalent formulations of these topological notions using the image of $IxI$ under the map $F \rightarrow F/P^N$.  We will see that $(IxI)_{\lambda}$ is an admissible subset of $IxI$ for any $\lambda \in \mathcal{N}(G)_x$.  We remark that Vasiu has demonstrated the admissibility of latticed $F$-isocrystals, which are isocrystals over the field of fractions of the Witt vectors satisfying an additional property \cite{Vas}.  In addition, we will see that the sets $(IxI)_{\lambda}$ are locally closed in $IxI$ and that $(IxI)_{\leq \lambda}$ are precisely the closures of the $(IxI)_{\lambda}$ inside $IxI$.  It is not necessarily true in general that $(IxI)_{\leq \lambda}$ is an irreducible subset of $IxI$.  We show, however, that all of the irreducible components have the same codimension inside $IxI$.  We make more detailed remarks of this nature in Section \ref{S:thmproof}.

\subsection{The generic Newton slope sequence}\label{S:generic}

By observing that $IxI/I^N$ is irreducible for any positive integer $N$, we see that the double coset $IxI$ is irreducible for fixed $x \in \widetilde{W}$.  Furthermore, $IxI$ is the finite union of subsets of the form $(IxI)_{\lambda}$, any two of which are disjoint.  If $\lambda \in \mathcal{N}(G)_x$ is maximal, then $(IxI)_{\lambda}$ is actually an open subset of $IxI$.  Since $IxI$ is irreducible, there must exist a unique maximal element in $\mathcal{N}(G)_x$.

\begin{defn}
Given $x \in \widetilde{W}$, we define the generic Newton slope sequence $\nu_x \in \mathfrak{a}_{\Q, \text{dom}}$
to be the unique maximal element in $\mathcal{N}(G)_x$; \textit{i.e.},  $\nu_x$ is defined such that for all $\lambda \in \mathcal{N}(G)_x$, we have $\lambda \leq \nu_x$.
\end{defn}

For a given $x=\pi^{\mu}w$, note that $\nu_x$ may not coincide with the unique dominant element in the $W$-orbit
of $-\mu$, which we denote by $-\mu_{\text{dom}}$.  In general, $\nu_x \leq -\mu_{\text{dom}}$, with strict
inequality occurring for some $x$ such that $\mathbf{a}_x$ lies outside the dominant Weyl chamber.  In Section \ref{S:slopes} we provide explicit descriptions for the maximal and minimal elements in $\mathcal{N}(G)_x$ for all $x\in \widetilde{W}$. It would be nice to have a closed formula providing both the maximal element $\nu_x$ and the minimal element in $\mathcal{N}(G)_x$, even in the case of $G=SL_3(F)$. \begin{question}\label{T:nuxform} Is there a closed, root theoretic formula for the maximal and minimal elements in $\mathcal{N}(G)_x$ for $G = SL_3(F)$? For all $x \in \widetilde{W}$ and any $G$?\end{question}

\subsection{Length of a segment $[\mu, \lambda]$}\label{S:length}

The codimensions of the Newton strata inside $IxI$ are more conveniently expressed in terms of the length of a segment in the poset $\mathcal{N}(G)_x$, which we now introduce.  For $G=SL_3(F)$, the poset $\mathcal{N}(G)$ consists of a single connected component, which is a lattice.

Given $x \in \widetilde{W}$ and two slope sequences $\mu, \lambda \in
\mathcal{N}(G)_x$ such that $\mu \leq \lambda$, we may consider the segment $[\mu, \lambda]$ defined as follows:
\begin{equation*} [\mu,\lambda] := \{ \nu \in \mathcal{N}(G)_x \mid\ \mu \leq \nu \leq \lambda \}.
\end{equation*}  We define the length of the segment $[\mu, \lambda]$ inside $\mathcal{N}(G)_x$, denoted $\length_{\mathcal{N}(G)_x}[\mu,\lambda]$, to be the supremum of all natural numbers $n$ such that there exists a chain $\mu = \nu_0 < \nu_1 < \cdots < \nu_n = \lambda$ in the poset $\mathcal{N}(G)_x$.  Our definition of length is the same as Chai's notion of length on subsets of Newton points expected to appear in the
reduction modulo $p$ of a Shimura variety, see \cite{Ch}.

Similar to the situation in \cite{Ch}, it turns out that the poset $\mathcal{N}(G)_x$ is ranked or catenary; \textit{i.e.}, any two maximal chains have the same length.   We shall also see in Section \ref{S:slopes} that $\mathcal{N}(G)_x$ is a lattice.  We might reasonably expect that $\mathcal{N}(G)_x$ is always a ranked lattice. \begin{question}\label{T:rklattice} Is the poset $\mathcal{N}(G)_x$ ranked for all $G$?  Is $\mathcal{N}(G)_x$ a lattice for all $G$? \end{question}

\subsection{Problem statement}\label{S:thm}

We are now prepared to formally state the main theorem, which provides a formula for the codimension of the subset $(IxI)_{\leq \lambda}$ inside $IxI$.
\pagebreak
\begin{theorem}\label{T:main}
Let $G=SL_3(F)$ and fix $x \in \widetilde{W}$.  For a Newton slope sequence $\lambda \in \mathcal{N}(G)_x$, the
subset $(IxI)_{\leq \lambda}$ of $IxI$ is admissible, and \begin{equation*}\label{E:codim} \codim\left(
(IxI)_{\leq \lambda} \subseteq  IxI \right) =\length_{\mathcal{N}(G)_x}[\lambda, \nu_x].\end{equation*} Moreover, the closure of a given Newton stratum $(IxI)_{\lambda}$ in $IxI$ is precisely $(IxI)_{\leq \lambda}$.  Therefore, for two Newton polygons $\lambda_1< \lambda_2$ which are adjacent in the poset $\mathcal{N}(G)_x$, \begin{equation*} \codim\left( (IxI)_{\leq \lambda_1} \subset (IxI)_{\leq \lambda_2} \right) = 1. \end{equation*}
\end{theorem}

If we interpret Theorem \ref{T:main} root-theoretically, we can produce an alternative formula for the codimensions of the
Newton strata inside $IxI$.  Order the simple roots $\alpha_1, \alpha_2 \in X_*(T)$ in the usual way so that $\alpha_i = e_i-e_{i+1}$. Let $\omega_1=(1,0,0)$ and $\omega_2=(1,1,0)$.  For $s\in W$ denote by $s(C^0)$ the Weyl chamber
corresponding to $s$.  As a corollary to Theorem \ref{T:main} we have the following explicit formulae.

\begin{cor}\label{T:maincor} Let $G = SL_3(F)$, and fix $x \in \widetilde{W}$ and $\lambda \in \mathcal{N}(G)_x$.
\begin{enumerate}
\item  If $x = \pi^{(\mu_1,\mu_2,\mu_3)}s_1s_2s_1$ where $\mu_1+2 <\mu_2 +1 < \mu_3$, or if $x=\pi^{(-2n,n,n)}s_1s_2,\linebreak  \pi^{-(n,n,-2n)}s_2s_1,\  \pi^{(-2n+1,n-1,n)}s_2,$ or $\pi^{-(n,n-1,-2n+1)}s_1$ for some $n \in \N$, then
  \begin{equation}\label{E:rootform1}
     \codim\left( (IxI)_{\leq \lambda} \subseteq  IxI\right)= \left(\sum\limits_{i=1}^2 \lceil \langle \omega_i, \nu_x-\lambda\rangle\rceil\right)-1.
  \end{equation}
\item For all $x \in \widetilde{W}$ not of the form $x',\ \varphi(x'),$ or $\varphi^2(x')$ for $x'$ one of the values listed above, where $\varphi(x')$ rotates the alcove $\mathbf{a}_{x'}$ 120 degrees counterclockwise about the center of the base alcove, we have
  \begin{equation}\label{E:rootform2}
  \codim\left( (IxI)_{\leq \lambda} \subseteq  IxI\right)=
 \sum\limits_{i=1}^2 \lceil \langle \omega_i,\nu_x- \lambda\rangle\rceil.
  \end{equation}
\end{enumerate}
Here, $\lceil \ell \rceil$ denotes the ceiling function, which rounds up to the nearest integer.  Note that the formula in \eqref{E:rootform1} only makes sense for $\lambda \neq \nu_x$.
\end{cor}

Equation \eqref{E:rootform2} is the naive analog of Chai's root-theoretic formula for the length of posets of
Newton slope sequences associated to $G(F)$ for $F$ a $p$-adic field, which we now recall for comparison:
\begin{theorem}[Chai]\label{T:chai} Let $F$ be non-Archimedean local field, and let $C^{\nu}_{F,R^{\vee}}$
denote the poset of Newton slope sequences that lie below $\nu$ which occur for $G(F)$, where $G$ is connected,
reductive, quasisplit over $F$.  Denote by $\omega_{F,i}$ the fundamental $F$-weights, and let $\lambda$ be a
slope sequence lying below $\nu$. Then \begin{equation*}\label{E:chailength}
\length_{C^{\nu}_{F,R^{\vee}}}[\lambda,\nu]=\sum^n_{i=1}\lceil \langle \omega_{F,i}, \nu-\lambda\rangle \rceil
\end{equation*} \end{theorem} \noindent The statement of this theorem in \cite{Ch} is for $F$ a $p$-adic field,
since Chai is primarily interested in applications to Shimura varieties, although he remarks that the theorem is
true even when $F$ has positive characteristic.  A similar expression also arises as the formula for the
codimension of the Newton strata in the adjoint quotient of a reductive group, $\mathbb{A}(F)_{\leq \lambda}$ in
$\mathbb{A}(F)_{\leq \nu_x}$, appearing in \cite{KotNewtStrata}.

As indicated by Equation \eqref{E:rootform1}, for certain values of $x$, we require a correction term of -1 to the initial guess for the codimensions of the Newton strata, which incorrectly assumes that the length of the segment $[\lambda,\nu]$ inside the poset associated to a particular affine Weyl group element, $\mathcal{N}(G)_x$, coincides with the length in the larger poset $\mathcal{N}(G)$.  The affine Weyl group elements for which this correction term appears correspond either to ones whose poset $\mathcal{N}(G)_x$ is missing expected elements, as in our example from Figure \ref{fig:poset}, or to alcoves having half-integral generic Newton slopes, in which case rounding up to the nearest integer yields an overestimate for the codimension.  There are advantages to both the combinatorial and root-theoretic presentations for the codimension formula.  When expressed in terms of the length of the segment $[\lambda, \nu_x]$, the formula is independent of the affine Weyl group element $x$ in consideration.  On the other hand, the explicit root-theoretic version has a natural graphical interpretation, since it depicts, in some sense, the distance between the Newton polygons $N_{\lambda}$ and $N_{\nu_x}$.

\subsection{Affine Deligne-Lusztig varieties for $A_2$}\label{S:adlv}
Theorem \ref{T:main} is related to the study of certain affine Deligne-Lusztig varieties.  Let $G= SL_3(F)$ and $b \in G(F)$.  Recall the definition of the affine Deligne-Lusztig variety $X_x(b)$ inside the affine flag manifold: \begin{equation*}
 X_x(b) := \{ g \in G(F)/I : g^{-1}b\sigma(g) \in IxI\}.
\end{equation*} Little is known about the varieties $X_x(b)$, including, in most cases, whether or not they are empty as sets.  In \cite{Reu}, Reuman provides a simple criterion for determining non-emptiness of the affine Deligne-Lusztig
varieties inside the affine flag manifold for $G=SL_3(F)$ and $b=1$, and alternative methods are discussed in
\cite{GHKR}.  It is worth noting that the methods used in this
paper will provide another means by which we can determine for which $x$ the variety $X_x(1)$ is non-empty.  More
specifically, 0 is the minimal element in $\mathcal{N}(G)_x$ if and only if $X_x(1) \neq \emptyset$.

Denote by $\lambda$ the element $\ov{\nu}(b) \in \mathcal{N}(G)$.  Recall from Section \ref{S:gpthy} that
$(IxI)_{\lambda} = \linebreak IxI \cap \{ gb\sigma(g)^{-1} \mid g\in G\}$, so that $X_x(b) \neq \emptyset$ if and only if $(IxI)_{\lambda}
\neq \emptyset$.  One application of the calculations in Section \ref{S:slopes}, in which we explicitly describe the poset $\mathcal{N}(G)_x = \{\lambda \in \mathcal{N}(G) \mid (IxI)_{\lambda} \neq \emptyset \}$, is to determine for which $b \in G$ we have $X_x(b) \neq \emptyset$.  Our approach differs from Reuman's method and answers the non-emptiness question for any $b \in G$, rather than only $b=1$, in the case of $A_2$.  

Although our treatment of $SL_3(F)$ in Section \ref{S:valcalc} suggests that the number of cases becomes unmanageable as the rank of $G$ increases, the author believes that it might be possible to employ arguments similar in flavor to provide a complete answer to the question of non-emptiness in the case of $A_n$.  The reader will observe in Section \ref{S:slopes} that the examples constructed to prove non-emptiness all lie in $k((\pi))$, rather than $F=\ov{k}((\pi))$.  In this case, the characteristic polynomial is much simpler since the Frobenius $\sigma$ fixes all of the matrix entries.  Producing matrices over $F^{\sigma}$, together with defining the various cases in a much more combinatorial manner, might provide a strategy for answering questions about $\mathcal{N}(G)_x$.

\begin{question}\label{T:nonemptyQ}
Can arguments similar to those in appearing in Sections \ref{S:valcalc} and \ref{S:slopes} yield complete descriptions of $\mathcal{N}(G)_x$ for all $x$ and $G = SL_n(F)$?
\end{question}

\end{section}

\begin{section}{Reduction Steps}\label{S:reduction}

\subsection{Geometry of the Newton strata}\label{S:geom} The group $W$ is generated by $s_1$ and $s_2$, the simple reflections through the walls of the chamber $C$.  In
coordinates, if we write $x=\pi^{\mu}w$ for $\mu=(\mu_1, \mu_2, \mu_3)$, then $s_1: \mathbf{a}_x \mapsto
\mathbf{a}_{x'}$, where $x' = \pi^{(\mu_2, \mu_1, \mu_3)}s_1w$, and $s_2: \mathbf{a}_x \mapsto
\mathbf{a}_{x''}$, where $x'' = \pi^{(\mu_1, \mu_3, \mu_2)}s_2w$.  We shall use this description of $W$, together with some basic geometry of the root system for $A_2$, to make several key reduction steps.

\begin{prop}\label{T:thetacodim}
Let $\theta \in \text{Aut}_F(G)$ be such that $\theta(I) = I$.  The automorphism $\theta$ then induces
bijections on $\widetilde{W}$ and $\mathcal{N}(G)$, and a bijection $\mathcal{N}(G)_x \xrightarrow{\sim}
\mathcal{N}(G)_{\theta(x)}$.  Moreover, \begin{equation*}\label{E:thetacodim} \codim((IxI)_{\leq \lambda}
\subseteq IxI) = \codim((I\theta(x)I)_{\leq \theta(\lambda)} \subseteq I\theta(x)I). \end{equation*}
\end{prop}

\begin{proof}
Recall that there is a bijective correspondence between $\widetilde{W}$ and double cosets $IxI$.  The map $\theta: IxI \rightarrow \theta(IxI)=I\theta(x)I$ therefore induces a bijection on double cosets $\theta:I\ba G/I \xrightarrow{\sim} I\ba G/I$ and hence on $\theta: \widetilde{W} \xrightarrow{\sim} \widetilde{W}$.

In addition, since $\theta \in \text{Aut}_F(G)$, if two elements $g_1$ and $g_2$ are $\sigma$-conjugate in $G$, then $\theta(g_1)$ and $\theta(g_2)$ are also $\sigma$-conjugate.  We therefore also obtain a bijection on the level of $\sigma$-conjugacy classes $\theta: B(G) \xrightarrow{\sim} B(G)$, which gives rise to a bijection on the two sets of Newton slope sequences \linebreak $\theta:\mathcal{N}(G)\xrightarrow{\sim} \mathcal{N}(G)$ and $\theta: \mathcal{N}(G)_x  \xrightarrow{\sim} \mathcal{N}(G)_{\theta(x)}$.  Consequently, we obtain $\theta : (IxI)_{\lambda} \xrightarrow{\sim} (I \theta(x)I)_{\theta(\lambda)},$ which is an isomorphism of schemes.
\end{proof}

\begin{remark}
Let $\theta \in \text{Aut}_F(G)$ be such that $\theta (I)=I$, and assume in addition that $\theta(T) = T$.  Then $\theta(N_G(T)) = N_G(T)$, and so the bijection $\theta: \widetilde{W} \rightarrow \widetilde{W}$ will also be a group homomorphism.
\end{remark}

\begin{lemma}\label{T:reduction}
Let $x = \pi^{\mu}w \in \widetilde{W}$, where $\mu = (\mu_1, \mu_2, \mu_3)$. It suffices to calculate the
codimensions of the Newton strata in $IxI$ for the following cases:
\begin{enumerate}
\item[(A)] $\mathbf{a}_x \subset C^0$, where $\mu_2 \geq 0$,
\item[(B)] $\mathbf{a}_x \subset s_1(C^0)$, where $\mu_1 \geq 0$ and $\mu \neq (\mu_1, \mu_2, \mu_1)$.
\end{enumerate}
\end{lemma}

\begin{proof} Once we compute the codimensions of the Newton strata inside $IxI$ for all $x$ such that $\mathbf{a}_x$ lies in one of two fixed and adjacent Weyl chambers, then we can obtain the codimensions for the remaining $x$ by applying Proposition \ref{T:thetacodim} to the automorphism of $I$ which changes the coordinates that determine the origin for the base alcove.  Similarly, once we compute the codimensions of the Newton strata inside $IxI$ where $\mu$ has two non-negative coordinates, we obtain the remaining ones by applying Proposition \ref{T:thetacodim} to the automorphism which exchanges the two simple roots.

First we consider the symmetries of the base alcove, which change the coordinates that determine which vertex of $\mathbf{a}_1$ is the
origin.  Let us take a representative for the rotation by 120 degrees about the center of $\mathbf{a}_1$ to be
\begin{equation*}\tau := \displaystyle \begin{pmatrix}0 & 0 & \pi^{-1} \\ 1 & 0 & 0 \\ 0 & 1 & 0 \end{pmatrix} \in GL_3(F). \end{equation*} Define $\varphi(g):= \tau g \tau^{-1}\in \text{Aut}_F(G)$, and note that $\varphi$ fixes $I$.  Since $\sigma(\tau) = \tau$, the induced map $\varphi: \mathcal{N}(G) \rightarrow \mathcal{N}(G)$ is
the identity.  In addition, one can check that $\varphi(x) =\pi^ys_1s_2w(s_1s_2)^{-1},$ where $y = (-1,0,0)+s_1s_2(\mu_1,\mu_2,\mu_3) + s_1s_2w(0,0,1)$.  Proposition \ref{T:thetacodim} says that we can extend the calculations for $x \in \widetilde{W}$ such that $\mathbf{a}_x$ lie in two
adjacent Weyl chambers, to $I\varphi(x)I$ and $I\varphi^2(x)I$.

Another element of $\text{Aut}_F(G)$ comes from the automorphism of the Dynkin diagram associated to $\text{Lie}(G)$, which interchanges the two simple roots.  Denote by $\eta$ the longest element in the Weyl group \begin{equation*}\eta:= s_1s_2s_1 = \begin{pmatrix} 0&0&1\\0&1&0\\1&0&0\end{pmatrix},\end{equation*}  and define $\psi(g):= \eta (g^t)^{-1} \eta^{-1}$ for $g \in G$.  Then $\psi(I)=I$, and $\psi$ induces a map on $\mathcal{N}(G)$ given by $\psi (\mu_1, \mu_2,\mu_3) = (-\mu_3, -\mu_2, -\mu_1)$.  In addition, $\psi$ induces a bijection $\psi(x) = \pi^{(-\mu_3, -\mu_2,
-\mu_1)}w'$, where the reduced expression for $w'$ is obtained from $w$ by interchanging the subscripts 1 and 2.  Applying Proposition \ref{T:thetacodim} enables us to make the positivity restrictions on $\mu_1$ and $\mu_2$ in the two Weyl chambers.
\end{proof}

\subsection{Newton strata for single cosets}\label{S:coset}

As described in Section \ref{S:gpthy}, the affine Bruhat decomposition provides a natural decomposition of $G=SL_3(F)$ into Newton strata of the form $(IxI)_{\lambda}$.  In practice, however, it is easier to work with single cosets of the Iwahori subgroup.  For the purpose of computing the codimensions of the Newton strata, Lemma \ref{T:coset} below justifies passing to single cosets of the form $xI$.  We thus introduce two natural variants of the definitions of the Newton strata given in Section
\ref{S:gpthy}.  For a fixed $x \in \widetilde{W}$ and $\lambda \in \mathcal{N}(G)$, define \begin{equation*}
(xI)_{\lambda} := xI \cap (IxI)_{\lambda}, \end{equation*} \begin{equation*} (xI)_{\leq \lambda} :=
\coprod_{\lambda ' \leq \lambda} (xI)_{\lambda '}. \end{equation*}  Here again, $(xI)_{\lambda'}$ is non-empty
for only the finitely many $\lambda' \in \mathcal{N}(G)_x$, so that $(xI)_{\leq \lambda}$ is a union of finitely
many Newton strata $(IxI)_{\lambda'}$ intersected with the infinite-dimensional space $xI$.

By applying our notion of admissibility to the single coset stratum $(xI)_{\lambda}$, we can define
the codimension of $(xI)_{\leq \lambda}$ in $xI$. In addition, this codimension agrees with the desired codimension of $(IxI)_{\leq \lambda}$ in $IxI$.

\begin{lemma}\label{T:coset}
Fix $x \in \widetilde{W}$, and let $\lambda \in \mathcal{N}(G)_x$.  Then,
\begin{equation*}\label{E:coset}\codim\left( (IxI)_{\leq \lambda} \subseteq  IxI\right) =\codim\left( (xI)_{\leq
\lambda} \subseteq  xI\right). \end{equation*}
\end{lemma}

\begin{proof}
First consider the case in which $x=\pi^{\mu}w \in \widetilde{W}$ satisfies
$\mathbf{a}_x \subset C^0$.  In this case, $I \cap xIx^{-1}$ is of the form \begin{equation*}I\cap xIx^{-1}=
\begin{pmatrix} \mathcal{O}^{\times} & \mathcal{O} & \mathcal{O} \\ P^r & \mathcal{O}^{\times} & \mathcal{O} \\ P^s  & P^t & \mathcal{O}^{\times} \end{pmatrix},\end{equation*} where $r,s,t$ are positive integers that depend on $\mu$ and satisfy $s \geq r + t$. Consider
\begin{equation*}H:=
\begin{pmatrix} 1 & 0 & 0 \\ \ov{k}[\pi]_1^{r-1} & 1 & 0 \\ \ov{k}[\pi]_1^{s-1} & \ov{k}[\pi]_1^{t-1} & 1 \end{pmatrix}\subseteq
I,\end{equation*} where $\ov{k}[\pi]_1^n$ is the vector space over $\ov{k}$ generated by $\pi^i$ for $i = 1, \dots, n$. If $n=0$, we define
$\ov{k}[\pi]_1^n := 0$. The relationship $s \geq r+t$ implies that $H$ is a subgroup of $I$.  Observe that the Iwahori subgroup decomposes into a product $I=H\cdot (I\cap
xIx^{-1})$, in which $H \cap (I \cap xIx^{-1}) = 1$. One can verify that the map $H \times xI \rightarrow IxI$
given by $(h, xi) \mapsto hxi\sigma(h)^{-1}$ is an isomorphism of schemes.  Under this isomorphism, $H \times (xI)_{\lambda}$ is mapped to $(IxI)_{\lambda}$ and  $H \times (xI)_{\leq \lambda}$ is mapped to $(IxI)_{\leq \lambda}$, and so the result holds in the case where $\mathbf{a}_x \subset C^0$.  The other cases are handled similarly.
\end{proof}

\end{section}

\begin{section}{Conditions determining the Newton strata}\label{S:valcalc}

The proof of Theorem \ref{T:main} proceeds in two steps.  The focus of this section will be to calculate the
explicit form of the subscheme $(xI)_{\leq \lambda}$ inside $xI$ for all $x \in \widetilde{W}$ in the two cases determined by Lemma \ref{T:reduction}. We will see
that the conditions that define $(xI)_{\leq \lambda}$ are polynomial in finitely many of the
coefficients of the entries of a given $g \in xI$.  Having explicit formulae determining all Newton strata in
$xI$ allows us to compute $\mathcal{N}(G)_x$ for all $x$, from which we may obtain a concrete formula for $\length_{\mathcal{N}(G)_x}[\lambda, \nu_x]$.  We provide descriptions of $\mathcal{N}(G)_x$ in the next section, and the proof of Theorem \ref{T:main} then appears in Section \ref{S:codimcalc}.

For fixed $x \in \widetilde{W}$ and $\lambda \in \mathcal{N}(G)_x$, we use the
characteristic polynomial from Section \ref{S:charpoly} to find explicit conditions on the
entries of a particular $g \in xI$ that yield $\ov{\nu}(g) \leq \lambda$.  Since
$\lambda \in \mathfrak{a}_{\Q}$, we will encounter conditions on the valuations of the matrix entries that involve rational numbers.  We thus adopt the convention that $P^{\ell} := P^{\lceil \ell \rceil}$, for $\ell
\in \Q$.  In addition, we will occasionally abuse notation and write $\pi^{\ell}:= \pi^{\lceil \ell \rceil}$, for $\ell \in \Q$.

\subsection{Two technical lemmas}\label{S:twolemmas}

We open with two technical, but useful lemmas.  The first lemma reformulates the definition of the partial ordering on $\mathcal{N}(G)$ in terms of conditions on the valuations of the coefficients of the characteristic polynomial.

\begin{lemma}\label{T:charpolypo}
Fix $\lambda \in \mathcal{N}(G)$, and suppose that $\nu$ is the Newton slope sequence associated to the isocrystal $(F^3,g\sigma)$ for $g\in G$, having characteristic polynomial of the form $f = \sigma^3 + \alpha\sigma^2 + \beta\sigma + \gamma$.  Then, \begin{equation*}\label{E:charpolypo}\nu \leq \lambda \iff \alpha \in P^{-\lambda_1}\ \text{and}\ \beta \in P^{\lambda_3}.\end{equation*}
\end{lemma}

\begin{proof}
Write $\nu = (\nu_1,\nu_2,\nu_3)$.  Since $\val(\gamma) = 0$ in $SL_3(F)$, the endpoints for $N_{\nu}$ and $N_{\lambda}$ coincide.  We thus have that $\nu \leq \lambda$ precisely when $\lambda_1 \geq \nu_1$ and $\lambda_3 \leq \nu_3$.
\end{proof}

At present the second lemma is unmotivated, but the result will be useful in certain
natural subcases within the proofs of almost every subsequent proposition.

\begin{lemma}\label{T:irrelterms}
Let $\lambda, \mu \in X_*(T) \otimes_{\Z} \Q$, and assume that $\lambda$ is dominant.  If, in addition,
$\mu_1+\mu_3 \leq \lambda_3$, then $P^{\mu_2 - \lambda_1} \subseteq P^{\lambda_3}$.
\end{lemma}

\begin{proof}
It suffices to show that $\mu_2 - \lambda_1 \geq \lambda_3$.  Now, $\mu_1 + \mu_3 \leq \lambda_3
\iff -\mu_1 - \mu_3 \geq -\lambda_3$.  But $\mu_1+\mu_2+\mu_3 = 0$ in $SL_3(F)$, so that $\mu_2 \geq -\lambda_3
\geq -\lambda_2$, where we have also used that $\lambda$ is dominant.  Hence, $\mu_2
\geq \lambda_1+\lambda_3$, and so $\mu_2 - \lambda_1 \geq \lambda_3$, as desired.
\end{proof}

\subsection{Conditions on valuations determining the Newton strata: Case A}\label{S:valpolysA}

It suffices to compute $(xI)_{\leq \lambda}$ for $x \in \widetilde{W}$ such that the alcoves $\mathbf{a}_x$ satisfy either condition A or B as specified in Lemma \ref{T:reduction}. We begin by systematically analyzing case A.  Recall that in this case, we consider $x \in \widetilde{W}$ such that the translation component is antidominant and has two non-negative coordinates. In addition, if $x = \pi^{\mu}w \in \widetilde{W}$, then there are six possible values for $w \in W$:
\begin{align*}
(\text{I})\ &w = s_1s_2 & (\text{IV})\ &w = s_1\\
(\text{II})\ &w = s_2s_1 & (\text{V})\ &w = s_2\\
(\text{III})\ &w = s_1s_2s_1 & (\text{VI})\ &w = 1.
\end{align*}
In this subsection, we compute $(xI)_{\leq \lambda}$ for $x=\pi^{\mu}w$ such that $\mathbf{a}_x \subset C^0$ and $\mu_2\geq 0$, where $w\in W$ falls into one of the above six cases.

The reader should note that in order to rigorously verify the arguments for case B, which are only indicated in an abbreviated form in Section \ref{S:valpolysB}, he should also perform the following calculations not only for $xI$, but also for $xI'$, where $I'$ is the non-standard Iwahori subgroup defined in Equation \eqref{E:I'}.  We justify this claim in Section \ref{S:valpolysB}, although the calculations are more easily performed simultaneously with those for case A.

\begin{prop}[\textbf{Case IA}]\label{T:123vals} Let $x=\pi^{\mu}s_1s_2$ satisfy $\mathbf{a}_x \subset C^0$ and $\mu_2 \geq 0$.  Then $\mu_1< \mu_2 \leq \mu_3$.  In addition, $\mu_1 < 0$ and $\mu_3 > 0$.

Now fix $\lambda=(\lambda_1, \lambda_2, \lambda_3) \in \mathcal{N}(G)_x$.   We then have \begin{equation}\label{E:123gen} \lambda_1 \leq -\mu_1 - 1 \quad \text{and} \quad \lambda_3 \geq -\mu_3 + \frac{1}{2}.\end{equation} Further, the only possibility in which $\lambda_3 = -\mu_3 + \frac{1}{2}$ is for $\mu_2=\mu_3$.  Otherwise, $\lambda_3 \geq -\mu_3 + 1$.  In describing $(xI)_{\leq \lambda}$, we have the following two subcases:
\begin{enumerate}
\item[($i$)] If $-\mu_3 + \frac{1}{2}\leq \lambda_3\leq -\mu_2+1$, then \begin{equation}\label{E:123i} (xI)_{\leq \lambda} = \left\{ \begin{pmatrix} a& b& c\\d&e&f\\g&h&i \end{pmatrix} \in xI \biggm| a \in P^{-\lambda_1} \ \text{and}\ \begin{vmatrix} a & b\\ d& e\end{vmatrix} \in P^{\lambda_3}\right\}. \end{equation}
\item[($ii$)] If $-\mu_2+1 < \lambda_3$, then \begin{equation}\label{E:123ii} (xI)_{\leq \lambda} = \left\{ \begin{pmatrix} a& b& c\\d&e&f\\g&h&i \end{pmatrix} \in xI \biggm| a \in P^{-\lambda_1}\ \text{and}\ \ov{d}b+\ov{g}c \in P^{\lambda_3}\right\}. \end{equation} Note that subcase ($ii$) only arises when $\mu_2 >1$, since $\lambda_3$ is always non-positive.
\end{enumerate}
\end{prop}

\begin{proof}
We first claim that if $A \in xI$, then $e_1$ is a cyclic vector for $(F^3,A\sigma)$.  Since $w = s_1s_2$, we have that \begin{equation*}\label{E:123coset}
A:= \begin{pmatrix} a& b& c\\d&e&f\\g&h&i \end{pmatrix} \in \begin{pmatrix}P^{\mu_1+1} & P^{\mu_1+1} &
P^{\mu_1}_{\times}\\ P^{\mu_2}_{\times} & P^{\mu_2} & P^{\mu_2} \\ P^{\mu_3+1} & P^{\mu_3}_{\times} & P^{\mu_3}
\end{pmatrix} = xI.
\end{equation*}  Here, by $y \in P^k_{\times}$ we mean that $\val(y) = k$.  Recall from Section \ref{S:charpoly} that $D= \displaystyle \ov{d}\begin{vmatrix} d& e\\g&h \end{vmatrix} +
\ov{g}\begin{vmatrix} d&f\\g&i\end{vmatrix}$.  We may directly compute that $\val(D) = 2\mu_2+\mu_3$, since in this case we have $\mu_2 <
\mu_3 + 1$.  Hence, by Proposition \ref{T:Dneq0}, $e_1$
is a cyclic vector for $(F^3, A\sigma)$, and therefore the characteristic polynomial is given by the equations
provided in Section \ref{S:charpoly}.

Denote by $\nu_A$ the Newton slope sequence associated to $(F^3, A\sigma)$.  Recall from Lemma \ref{T:charpolypo} that $\nu_A \leq \lambda \iff \alpha \in P^{-\lambda_1}$ and $\beta \in P^{\lambda_3}$.  It thus suffices to compute the conditions under which $\alpha \in P^{-\lambda_1}$ and $\beta \in P^{\lambda_3}$.

We begin by examining the conditions under which $\alpha \in P^{-\lambda_1}$.  Observe that
$$\frac{1}{D}\left( (\ov{\ov{d}}\ov{e}+\ov{\ov{g}}\ov{f})\begin{vmatrix} d&e\\g&h\end{vmatrix} +
(\ov{\ov{d}}\ov{h}+\ov{\ov{g}}\ov{i})\begin{vmatrix}d&f\\g&i\end{vmatrix}\right) \in \mathcal{O},$$ since $\mu_2 \geq 0$.  Thus we see
by Equation \eqref{E:alpha} that $\alpha \in P^{-\lambda_1} \iff \ov{\ov{a}} \in P^{-\lambda_1} \iff a \in P^{-\lambda_1}$, since $\sigma(P)=P$. In addition, since in this case $a \in P^{\mu_1+1}$, we see that $\lambda_1\leq -\mu_1-1$.

Now we consider the condition $\beta \in P^{\lambda_3}$.  Compute that
\begin{align*}
B_1 &\in  P^{\mu_1+\mu_2+1} & B_2 &\in P^{\mu_1+\mu_3+1}\\
- \ov{\ov{d}}\ov{b}&\in P^{\mu_1+\mu_2+1} &  -\ov{\ov{g}}\ov{c} &\in P^{\mu_1+\mu_3+1}\\
\frac{\ov{D}}{D}\begin{vmatrix}e&f\\h&i\end{vmatrix} &\in P^{\mu_2+\mu_3} \subset \mathcal{O} & \phantom{}
\end{align*}  In particular, $\beta \in P^{\mu_1+\mu_2+1}$ and so if $\mu_2<\mu_3$, then $\lambda_3 \geq -\mu_3 + 1.$ In the special case in which $\mu_2 = \mu_3$, we instead have $\lambda_3 \geq \frac{\mu_1+1}{2} =  -\mu_3 + \frac{1}{2},$ and we have verified Equation \eqref{E:123gen}.  We will use our estimates for $\lambda_1$ and $\lambda_3$ to explicitly describe $\mathcal{N}(G)_x$ in Section \ref{S:slopes}.  Similar estimates will appear in all subsequent propositions without further comment.

Comparing the valuations of the summands of $\beta$, we see that we should consider two subcases, which are identical to those provided in the statement of the proposition:
\begin{enumerate}
\item[($i$)] $\mu_1+\mu_2+\frac{1}{2} \leq \lambda_3\leq \mu_1+\mu_3+1$,
\item[($ii$)] $\mu_1+\mu_3+1 < \lambda_3$.
\end{enumerate}
We now consider each subcase individually.

\vskip 10 pt Subcase ($i$): $\mu_1+\mu_2+\frac{1}{2} \leq \lambda_3\leq \mu_1+\mu_3+1$
\vskip 10 pt

First assume that $\lambda_3 = -\mu_3 + \frac{1}{2}$, which only arises if $\mu_2=\mu_3$.  In this special case, $\beta \in P^{\lambda_3}$ automatically.  Therefore, if $\mu_2=\mu_3$, we have \begin{equation}\label{E:123i*} (xI)_{\leq \lambda} = \left\{ \begin{pmatrix} a& b& c\\d&e&f\\g&h&i \end{pmatrix} \in xI \biggm| a \in P^{-\lambda_1} \right\}. \end{equation} The expression in \eqref{E:123i*} is equivalent to Equation \eqref{E:123i}, since in this special case we automatically have $\ov{\begin{vmatrix}a&b\\d&e\end{vmatrix}} \in P^{\lambda_3} \iff \begin{vmatrix}a&b\\d&e\end{vmatrix} \in P^{\lambda_3}$.

Now assume that $\mu_2 < \mu_3$ so that $\lambda_3 \geq -\mu_3 +1$. Then $\beta \in P^{\lambda_3} \iff B_1 -
\ov{\ov{d}}\ov{b} \in P^{\lambda_3}$.  We may write
\begin{align} B_1 - \ov{\ov{d}}\ov{b} &=
\frac{1}{D}\left(\ov{\ov{d}}\ov{a}\ov{e}\begin{vmatrix}d&e\\g&h\end{vmatrix} +
\ov{\ov{g}}\ov{a}\ov{f}\begin{vmatrix} d&e\\g&h\end{vmatrix}- \ov{\ov{d}}\ov{b}D\right)\notag \\ \phantom{B_1 -
\ov{\ov{d}}\ov{b}} &= \frac{1}{D}\left(\ov{\ov{d}}\ov{a}\ov{e}\begin{vmatrix}d&e\\g&h\end{vmatrix} +
\ov{\ov{g}}\ov{a}\ov{f}\begin{vmatrix} d&e\\g&h\end{vmatrix}-
\ov{\ov{d}}\ov{b}\ov{d}\begin{vmatrix}d&e\\g&h\end{vmatrix} -
\ov{\ov{d}}\ov{b}\ov{g}\begin{vmatrix}d&f\\g&i\end{vmatrix}\right) \notag \\\phantom{B_1 - \ov{\ov{d}}\ov{b}} &=
\frac{1}{D}\left(\ov{\ov{d}}\ov{\begin{vmatrix}a&b\\d&e\end{vmatrix}}\begin{vmatrix}d&e\\g&h\end{vmatrix} +
\ov{\ov{g}}\ov{a}\ov{f}\begin{vmatrix} d&e\\g&h\end{vmatrix}-
\ov{\ov{d}}\ov{b}\ov{g}\begin{vmatrix}d&f\\g&i\end{vmatrix}\right) .\label{E:betaarrange}\end{align}   First,
observe that $\displaystyle \frac{\ov{\ov{g}}\ov{a}\ov{f}}{D}\begin{vmatrix}d&e\\g&h\end{vmatrix} \in
P^{\mu_1+\mu_3+2}\subset P^{\lambda_3}$ in subcase ($i$).  Similarly, $\displaystyle
-\frac{\ov{\ov{d}}\ov{b}\ov{g}}{D}\begin{vmatrix}d&f\\g&i\end{vmatrix} \in P^{\lambda_3}$.  Thus, the second and
third terms in our final expression for $B_1 - \ov{\ov{d}}\ov{b}$ are automatically in $P^{\lambda_3}$.
Consequently, $\beta \in P^{\lambda_3} \iff \displaystyle
\frac{\ov{\ov{d}}}{D}\ov{\begin{vmatrix}a&b\\d&e\end{vmatrix}}\begin{vmatrix}d&e\\g&h\end{vmatrix} \in
P^{\lambda_3} \iff \ov{\begin{vmatrix}a&b\\d&e\end{vmatrix}} \in P^{\lambda_3}$, since $\displaystyle
\frac{\ov{\ov{d}}}{D}\begin{vmatrix}d&e\\g&h\end{vmatrix} \in \mathcal{O}^{\times}$.  Lemma \ref{T:charpolypo} and the fact that $\sigma(P)=P$ now imply the result for subcase ($i$).

\vskip 10 pt Subcase ($ii$): $\mu_1+\mu_3+1 < \lambda_3$
\vskip 10 pt

A priori, all four terms $B_1+B_2 - \ov{\ov{d}}\ov{b} - \ov{\ov{g}}\ov{c}$ affect $\val(\beta)$ if $\lambda_3 > \mu_1+\mu_3+1$.  Note, however, that both $B_1$ and $B_2$ contain a factor of $a$.  We used that $a \in
P^{\mu_1+1}$ to obtain our original estimates for the valuations of $B_1$ and $B_2$.  In this subcase, we
actually have a better estimate for $\val(a)$.  Namely, $a \in P^{-\lambda_1}$ by our analysis of $\alpha$.
In particular, we then see that $B_1 \in
P^{\mu_2-\lambda_1}.$ Since $\mu_1+\mu_3+1 < \lambda_3$, the hypotheses of Lemma \ref{T:irrelterms} are satisfied, and $B_1$ is automatically in $P^{\lambda_3}$. Similarly, we can compute that $B_2 \in
P^{\mu_3-\lambda_1} \subseteq P^{\mu_2-\lambda_1}$ since $\mu_2 \leq \mu_3$ so that $B_2$ is also automatically in
$P^{\lambda_3}$.  For the range of $\lambda_3$ specified in subcase ($ii$), we thus see that $\beta \in P^{\lambda_3} \iff \ov{\ov{d}}\ov{b} +
\ov{\ov{g}}\ov{c}\in P^{\lambda_3} \iff \ov{d}b +
\ov{g}c\in P^{\lambda_3}$, and Lemma \ref{T:charpolypo} implies Equation \eqref{E:123ii}.
\end{proof}

\begin{prop}[\textbf{Case IIA}]\label{T:132vals}
Let $x=\pi^{\mu}s_2s_1$ satisfy $\mathbf{a}_x \subset C^0$ and $\mu_2 \geq 0$. Then \linebreak $\mu_1< \mu_2 < \mu_3$.  In addition, $\mu_1 < 0$ and $\mu_3 > 0$.

Now fix $\lambda=(\lambda_1, \lambda_2, \lambda_3) \in \mathcal{N}(G)_x$.  We then have \begin{equation}\label{E:132gen} \lambda_1 \leq -\mu_1-1 \quad \text{and} \quad \lambda_3 \geq -\mu_3+1.\end{equation}  In describing $(xI)_{\leq \lambda}$, we have
the following two subcases:
\begin{enumerate}
\item[($i$)] If $-\mu_3+1\leq \lambda_3\leq -\mu_2$, then \begin{equation} (xI)_{\leq \lambda} = \left\{ \begin{pmatrix} a& b& c\\d&e&f\\g&h&i \end{pmatrix} \in xI \biggm| a \in P^{-\lambda_1} \ \text{and}\ \begin{vmatrix} a & b\\ d& e\end{vmatrix} \in P^{\lambda_3}\right\}. \end{equation}
\item[($ii$)] If $-\mu_2 < \lambda_3$, then \begin{equation} (xI)_{\leq \lambda} = \left\{ \begin{pmatrix} a& b& c\\d&e&f\\g&h&i \end{pmatrix} \in xI \biggm| a \in P^{-\lambda_1}\ \text{and}\ \ov{d}b+\ov{g}c \in P^{\lambda_3}\right\}. \end{equation} Note that subcase ($ii$) only arises when $\mu_2 >0$, since $\lambda_3$ is always non-positive.
\end{enumerate}
\end{prop}

\begin{proof}
Unlike in Proposition \ref{T:123vals}, $e_1$ is not automatically a cyclic vector in this case.  By computing $e_2 \wedge \Phi(e_2) \wedge \Phi^2(e_2)$ as in the proof of Proposition \ref{T:Dneq0}, we see that $e_2$ is a cyclic vector for $(F^3, A\sigma)$ if and only if $D'= \displaystyle \ov{b}\begin{vmatrix} a& b\\g&h \end{vmatrix} -
\ov{h}\begin{vmatrix} b&c\\h&i\end{vmatrix} \neq 0$.  We have that $\val(D') = 2\mu_1+\mu_3$ in this case, and thus $e_2$ is a cyclic vector for $(F^3, A\sigma)$.

So that we can continue to use the expressions for $\alpha$ and $\beta$ from Section \ref{S:charpoly}, we replace $\Phi = A\sigma$ by $B\Phi B^{-1}$, where \begin{equation*} B:= \begin{pmatrix} 0&1&0\\1&0&0\\0&0&1\end{pmatrix}\end{equation*} is a change of basis.  Since $\sigma$ fixes $B$, we then see that $e_1$ is a cyclic vector for $(F^3, BAB^{-1}\sigma)$.  Thus, if we denote by $A' := BAB^{-1}$, we have that
\begin{equation*}A':= \begin{pmatrix} a&b&c\\d&e&f\\g&h&i \end{pmatrix} \in \begin{pmatrix}P^{\mu_2+1} & P^{\mu_2+1} &
P^{\mu_2}_{\times}\\ P^{\mu_1}_{\times} & P^{\mu_1+1} & P^{\mu_1} \\ P^{\mu_3} & P^{\mu_3}_{\times} & P^{\mu_3}
\end{pmatrix} = B(xI)B^{-1}.\end{equation*}
Since $\sigma(B)=B$, conjugation and $\sigma$-conjugation by $B$ coincide and thus induce the identity on $\mathcal{N}(G)_x$.  For the matrix $A'$ we have $\val(D) = 2\mu_1 + \mu_3$ and can thus use Equations \eqref{E:alpha} and \eqref{E:beta} to compute the conditions under which $A' \in (BxIB^{-1})_{\leq \lambda}$.  We then change back to our coordinates for $A$ instead of $A'$ to describe $(xI)_{\leq \lambda}$.

Lemma \ref{T:charpolypo} says that we need only compute the conditions under which $\alpha \in P^{-\lambda_1}$ and $\beta \in P^{\lambda_3}$.  First, observe that
$$-\ov{\ov{a}} - \frac{1}{D}\left( \ov{\ov{g}}\ov{f}\begin{vmatrix} d&e\\g&h\end{vmatrix} +
(\ov{\ov{d}}\ov{h}+\ov{\ov{g}}\ov{i})\begin{vmatrix}d&f\\g&i\end{vmatrix}\right) \in \mathcal{O}.$$  Hence we see that $\alpha \in P^{-\lambda_1} \iff \displaystyle \frac{\ov{\ov{d}}\ov{e}}{D}\begin{vmatrix} d & e\\g&h \end{vmatrix}\in P^{-\lambda_1} \iff \ov{e} \in P^{-\lambda_1}$, since $\displaystyle \frac{\ov{\ov{d}}}{D}\begin{vmatrix}d&e\\g&h \end{vmatrix} \in \mathcal{O}^{\times}$.  But then $ \ov{e} \in P^{-\lambda_1} \iff e \in P^{-\lambda_1}$, since $\sigma(P)=P$.

Similarly, by recalling Equation \eqref{E:beta} for $\beta$, we compute that
\begin{align*}
B_1 &\in  P^{\mu_1+\mu_2+2} & B_2 &\in \mathcal{O}\\
- \ov{\ov{d}}\ov{b}&\in P^{\mu_1+\mu_2+1} &  -\ov{\ov{g}}\ov{c} &\in \mathcal{O}\\
\frac{\ov{D}}{D}\begin{vmatrix}e&f\\h&i\end{vmatrix} &\in P^{\mu_1+\mu_3} & \phantom{}
\end{align*} Since $\alpha \in P^{\mu_1+1}$ and $\beta \in P^{\mu_1+\mu_2+1}$, Lemma \ref{T:charpolypo} implies Equation \eqref{E:132gen}.

Comparing the valuations of the summands of $\beta$, we once more see that we should consider two subcases:
\begin{enumerate}
\item[($i$)] $\mu_1+\mu_2+1 \leq \lambda_3\leq \mu_1+\mu_3$,
\item[($ii$)] $\mu_1+\mu_3 < \lambda_3$.
\end{enumerate}
We now consider each subcase individually.

\vskip 10 pt Subcase ($i$): $\mu_1+\mu_2+1 \leq \lambda_3\leq \mu_1+\mu_3$
\vskip 10 pt

In this subcase, $\beta \in P^{\lambda_3} \iff B_1 - \ov{\ov{d}}\ov{b} \in P^{\lambda_3}$.  Recalling Equation \eqref{E:betaarrange}, we have that \begin{equation*} B_1 - \ov{\ov{d}}\ov{b} = \frac{1}{D}\left(\ov{\ov{d}}\ov{\begin{vmatrix}a&b\\d&e\end{vmatrix}}\begin{vmatrix}d&e\\g&h\end{vmatrix} +
\ov{\ov{g}}\ov{a}\ov{f}\begin{vmatrix} d&e\\g&h\end{vmatrix}- \ov{\ov{d}}\ov{b}\ov{g}\begin{vmatrix}d&f\\g&i\end{vmatrix}\right).  \end{equation*} Again, the last two terms are automatically in $\mathcal{O}$, so
$\beta \in P^{\lambda_3} \iff \displaystyle
\frac{\ov{\ov{d}}}{D}\ov{\begin{vmatrix}a&b\\d&e\end{vmatrix}}\begin{vmatrix}d&e\\g&h\end{vmatrix} \in P^{\lambda_3} \iff \ov{\begin{vmatrix}a&b\\d&e\end{vmatrix}} \in P^{\lambda_3}$, since $\displaystyle \frac{\ov{\ov{d}}}{D}\begin{vmatrix}d&e\\g&h\end{vmatrix} \in
\mathcal{O}^{\times}$.  But again, $\ov{\begin{vmatrix}a&b\\d&e\end{vmatrix}} \in P^{\lambda_3} \iff \begin{vmatrix}a&b\\d&e\end{vmatrix} \in P^{\lambda_3}$.

\vskip 10 pt Subcase ($ii$): $\mu_1+\mu_3 < \lambda_3$
\vskip 10 pt

Note that $\beta \in P^{\lambda_3} \iff B_1 - \ov{\ov{d}}\ov{b} + \ds \frac{\ov{D}}{D}ei - \frac{\ov{D}}{D}fh
\in P^{\lambda_3}$.  However, we can employ a better estimate for $\val(e)$, since $e \in P^{-\lambda_1}$.
Using this observation, we compute that $B_1 \in
P^{\mu_2-\lambda_1+1}\subset P^{\mu_2 -\lambda_1}$.  Again, Lemma \ref{T:irrelterms} applies to demonstrate that $B_1 \in P^{\lambda_3}$. In addition, note that $\ds \frac{\ov{D}}{D} \in
\mathcal{O}^{\times}$ so that $\ds \frac{\ov{D}}{D}ei \in P^{\mu_3-\lambda_1}\subset P^{\mu_2 - \lambda_1}$ by our hypotheses on $\mu$.  Therefore, $\displaystyle \frac{\ov{D}}{D}ei \in P^{\lambda_3}$ by Lemma \ref{T:irrelterms} as well.

We have thus shown that $\beta \in P^{\lambda_3} \iff \ov{\ov{d}}\ov{b} +\displaystyle \frac{\ov{D}}{D}fh \in
P^{\lambda_3}$. Write \begin{align*} \ov{\ov{d}}\ov{b} + \frac{\ov{D}}{D}fh &= \frac{D}{D}\ov{\ov{d}}\ov{b} +
\frac{\ov{D}}{D}fh \notag\\ \phantom{\ov{\ov{d}}\ov{b} + \frac{\ov{D}}{D}fh} &= \frac{1}{D}\left(
\ov{\ov{d}}\ov{b}\ov{d} \begin{vmatrix} d&e \\ g&h \end{vmatrix} +
\ov{\ov{d}}\ov{b}\ov{g}\begin{vmatrix}d&f\\g&i \end{vmatrix} +fh\ov{\ov{d}} \ov{\begin{vmatrix} d&e \\ g&h
\end{vmatrix}} + fh\ov{\ov{g}}\ov{\begin{vmatrix}d&f\\g&i \end{vmatrix}}\right).  \end{align*}  First we argue that the second and last terms in this expression lie in $\mathcal{O}$.  Computing valuations, we
see that $\ds \frac{\ov{\ov{d}}\ov{b}\ov{g}}{D}\begin{vmatrix}d&f\\g&i\end{vmatrix} +
\frac{fh\ov{\ov{g}}}{D}\ov{\begin{vmatrix}d&f\\g&i
\end{vmatrix}} \in
P^{\mu_2+\mu_3+1} +P^{2\mu_3} \subset \mathcal{O}$, since $\mu_3 > \mu_2 \geq 0$ by hypothesis.  Thus, $\beta \in
P^{\lambda_3} \iff \ds \frac{1}{D}\left( \ov{\ov{d}}\ov{b}\ov{d} \begin{vmatrix} d&e \\ g&h \end{vmatrix}
+fh\ov{\ov{d}} \ov{\begin{vmatrix} d&e \\ g&h
\end{vmatrix}} \right) \in P^{\lambda_3}$.

Now we again use our estimate on $\val(e)$ to show that $\ds \frac{-1}{D} \left( \ov{\ov{d}}\ov{b}\ov{d}ge
+ fh\ov{\ov{d}}\ov{ge} \right)$ is automatically in $P^{\lambda_3}$.  Using that $e \in P^{-\lambda_1}$, we see
that $\ds \frac{\ov{\ov{d}}\ov{b}\ov{d}ge}{D} \in P^{\mu_2-\lambda_1+1}\subset P^{\lambda_3}$ by Lemma
\ref{T:irrelterms}.  Similarly, $\ds \frac{fh\ov{\ov{d}}\ov{ge}}{D} \in P^{\mu_3 - \lambda_1} \subset
P^{\mu_2-\lambda_1} \subseteq P^{\lambda_3}$.

We have thus demonstrated that $\ds \beta \in P^{\lambda_3} \iff \frac{\ov{\ov{d}}\ov{d}h}{D}(\ov{b}d +
\ov{h}f)\in P^{\lambda_3}$ in subcase ($ii$).  Again since $\ds \frac{\ov{\ov{d}}\ov{d}h}{D} \in \mathcal{O}^{\times}$, we conclude that $\beta \in P^{\lambda_3} \iff \ov{b}d + \ov{h}f \in P^{\lambda_3}$. Altogether, we have
proved the following:
\begin{enumerate}
\item[($i$)] If $-\mu_3+1\leq \lambda_3\leq -\mu_2$, then \begin{equation*} (BxIB^{-1})_{\leq \lambda} = \left\{ \begin{pmatrix} a& b& c\\d&e&f\\g&h&i \end{pmatrix} \in BxIB^{-1} \biggm| e \in P^{-\lambda_1} \ \text{and}\ \begin{vmatrix} a & b\\ d& e\end{vmatrix} \in P^{\lambda_3}\right\}. \end{equation*}
\item[($ii$)] If $-\mu_2 < \lambda_3$, then \begin{equation*} (BxIB^{-1})_{\leq \lambda} = \left\{ \begin{pmatrix} a& b& c\\d&e&f\\g&h&i \end{pmatrix} \in BxIB^{-1} \biggm| e \in P^{-\lambda_1}\ \text{and}\ \ov{b}d+\ov{h}f \in P^{\lambda_3}\right\}. \end{equation*}
\end{enumerate}
Conjugating by $B^{-1}$ to change back to the original coordinates, we obtain the desired expressions for $(xI)_{\leq \lambda}$.
\end{proof}

To discuss the remaining cases, we make two additional reduction steps.  In order to use the results from
Section \ref{S:charpoly}, our analysis will break into certain natural subcases.  During the course of several of these subcases, the following lemma will be
useful.

\begin{lemma}\label{T:split}
Any short exact sequence of isocrystals over $F$ splits.
\end{lemma}

For a proof of this lemma, see \cite{Dem}, and note that changing fields from $B(\ov{k})$ to $\ov{k}((\pi))$ does not alter the arguments.  In applying Lemma \ref{T:split}, the following result will also be necessary.

\begin{lemma}\label{T:GL_2}
Let $G = GL_2(F)$ and $(F^2, g\sigma)$ be an isocrystal, where $g:=\begin{pmatrix} a&b\\c&d\end{pmatrix} \in G$.
Let $\lambda =(\lambda_1,\lambda_2) \in \mathcal{N}(G)_x$ for $x \in \widetilde{W}$, the corresponding affine
Weyl group, in which case $\lambda_1 + \lambda_2=\val(\det x)$. Then $e_1$ is a cyclic vector for $(F^2,g\sigma)$ if and only if $c \neq 0$, and in this situation we have
\begin{equation*}(xI)_{\leq \lambda} = \left\{ \begin{pmatrix} a&b\\c&d\end{pmatrix}  \in xI \biggm| \ov{a} + \frac{\ov{c}}{c}d \in P^{-\lambda_1}
\right\}\end{equation*}
\end{lemma}

\begin{proof}Let $g\sigma = \Phi$.  To verify the condition under which $e_1$ is a cyclic vector, compute $e_1 \wedge \Phi(e_1)$.  If $e_1$ is a cyclic vector, then the characteristic polynomial for $(F^2,\Phi)$ is of the form $f = \sigma^2 + \alpha_1\sigma + \frac{\ov{c}}{c}\det g$.  Using linear algebra as in Section \ref{S:charpoly}, we compute that $\alpha_1 =  -(\ov{a} + \frac{\ov{c}}{c}d)$, as required.
\end{proof}

Our final reduction step in case A shows that, without loss of generality, we can assume that certain matrix entries are zero.

\begin{lemma}\label{T:13zeros}
Let $x = \pi^{\mu}s_1s_2s_1$, where $\mu = (\mu_1,\mu_2,\mu_3)$ satisfies $\mu_1 < \mu_2 < \mu_3$.  Define $$ J_1:= \begin{pmatrix} 1& 0& 0\\P^{\mu_2-\mu_1}&1&0\\P^{\mu_3-\mu_1}&P^{\mu_3-\mu_2}&1 \end{pmatrix}\quad \text{and}\quad  K_1:=  \left( xI \cap \left\{
\begin{pmatrix} a& b& c\\d&e&0\\g&0&0 \end{pmatrix} \right\} \right) .$$  Then the map \begin{align*}\kappa:  J_1 \times K_1 &\rightarrow xI \\ (j,k) & \mapsto j^{-1}k\sigma(j) \end{align*} is an isomorphism of schemes.
\end{lemma}

\begin{proof}
Consider $J_1, K_1,$ and $xI$ as schemes over $k = \mathbb{F}_q$.  We illustrate the argument by providing a bijection on sets of $\ov{k}$-points.  Let \begin{equation*} A := \begin{pmatrix} a & b & c \\ d & e & f \\ g & h & i \end{pmatrix} \in xI(\ov{k}) \quad \text{and} \quad j := \begin{pmatrix}1 & 0 & 0 \\d' & 1 & 0\\g' & h' & 1 \end{pmatrix} \in J_1(\ov{k}). \end{equation*} The reader may verify that, if the entries of $j$ are defined as follows: \begin{equation*} d':= -\frac{f}{c}, \quad h' := \frac{bi -ch}{ce - bf}, \quad g':= -\left(\frac{i+fh'}{c}\right), \end{equation*} then the product $jA\sigma(j)^{-1}$ lies in $K_1(\ov{k})$.  Observe that $\val(c) = \mu_1$ and $\val(ce-bf) = \mu_1+\mu_2$, so that the entries of $j$ are completely determined by the entries of $A$.  We therefore see that, given any $A \in xI(\ov{k})$, there exists a unique $j \in J_1(\ov{k})$ such that $jA\sigma(j)^{-1} \in K_1(\ov{k})$.  Consequently, we obtain mutually inverse injective morphisms $\iota: (xI)(\ov{k}) \rightarrow (J_1 \times K_1)(\ov{k})$ given by $A\mapsto (j,k)$, where $k:=jA\sigma(j)^{-1}$, and $\kappa: (J_1 \times K_1)(\ov{k}) \rightarrow xI(\ov{k})$ by $(j,k) \mapsto j^{-1}k\sigma(j)$. The argument on $S$-points, where $S$ is an arbitrary ring, proceeds in like manner.
\end{proof}

This isomorphism of schemes allows us to instead compute the codimension of \linebreak $J_1 \times \left( K_1 \cap (xI)_{\leq \lambda}\right)$ inside $J_1 \times K_1$, which will significantly simplify our calculations.

\begin{cor}\label{T:K_1cor}
Let $x = \pi^{\mu}s_1s_2s_1$, where $\mu = (\mu_1,\mu_2,\mu_3)$ satisfies $\mu_1 < \mu_2 < \mu_3$.  Denote
$(K_1)_{\leq \lambda}:= K_1 \cap (xI)_{\leq \lambda}$.  Then \begin{equation*}\label{E:K_1cor} \codim((K_1)_{\leq
\lambda} \subseteq K_1) = \codim((xI)_{\leq \lambda}\subseteq xI).\end{equation*}
\end{cor}

\begin{remark}\label{T:zerormk}
The reader should note that the argument in the proof of Lemma \ref{T:13zeros} generalizes to $x=\pi^{\mu}w$ in
$SL_n(F)$, where $w$ is the longest element of the Weyl group.  Furthermore, similar arguments can be employed
for any value of $w\in W$ to reduce to the case where certain matrix entries are zero.  We could have used analogous lemmas to make slight simplifications to the proofs of the previous two propositions.  However, such reductions do not reduce the complexity of the arguments as significantly as they do in the remaining
cases, and so the author finds it illuminating to carry out the computations for $w=s_1s_2$ and $w=s_2s_1$ without the
aid of any additional lemmas.
\end{remark}

The relative simplicity of the arguments in Propositions \ref{T:123vals} and \ref{T:132vals}
depends upon the existence of a convenient basis that guarantees that $e_1$ is a cyclic vector.  Since, by contrast, no such
obvious choice of basis exists in the remaining cases, we employ this further reduction step in each case to drastically
simplify the remaining arguments.  For the purpose of computing codimensions, we thus focus on providing a concrete description of $(K_1)_{\leq \lambda}$ inside $K_1$.  It is interesting to note, however, that the descriptions for $(K_1)_{\leq \lambda}$ actually agree with those for $(xI)_{\leq \lambda}$ in that these schemes are defined by precisely the same polynomial equations in the matrix entries.  This fact is somewhat tedious to verify and thus will not be proved.

\begin{prop}[\textbf{Case IIIA}]\label{T:13vals}Let $x=\pi^{\mu}s_1s_2s_1$ satisfy $\mathbf{a}_x \subset C^0$ and $\mu_2 \geq 0$.  Then $\mu_1 < \mu_2< \mu_3$.  In addition, $\mu_1 < 0$ and $\mu_3 >0$.

Now fix $\lambda=(\lambda_1,\lambda_2,\lambda_3) \in \mathcal{N}(G)_x$.  We then have \begin{equation}\label{E:13gen} \lambda_1 \leq -\mu_1-1 \quad \text{and} \quad \lambda_3 \geq -\mu_3+1.\end{equation} In describing $K_1 \cap (xI)_{\leq \lambda}$, we have the following two subcases:
\begin{enumerate}
\item[($i$)] If $\mu_2 + 1 < \mu_3$ and  $-\mu_3+1 \leq \lambda_3\leq -\mu_2$, or if $\mu_2+1=\mu_3$ and $d =0$, then we have \begin{equation}\label{E:13i} (K_1)_{\leq
\lambda} = \left\{ \begin{pmatrix} a& b& c\\d&e&0\\g&0&0 \end{pmatrix} \in K_1 \biggm| a \in P^{-\lambda_1}
\ \text{and}\ \begin{vmatrix} a & b\\ d& e\end{vmatrix} \in P^{\lambda_3}\right\}.
\end{equation}
\item[($ii$)] If $\mu_2+1<\mu_3$ and $-\mu_2 < \lambda_3$, or if $\mu_2+1=\mu_3$ and $d \neq 0$, then we have \begin{equation}\label{E:13ii}
(K_1)_{\leq \lambda} = \left\{ \begin{pmatrix} a& b& c\\d&e&0\\g&0&0 \end{pmatrix} \in K_1 \biggm| a \in
P^{-\lambda_1}\ \text{and}\ \ov{d}b+\ov{g}c \in P^{\lambda_3}\right\}. \end{equation}  Note that subcase ($ii$) only arises when $\mu_2 >0$, since $\lambda_3$ is always non-positive.  In the event that $\mu_2+1=\mu_3$ and $d\neq 0$, the value $\lambda_3=-\mu_3+1$ is fixed.  In addition, if $\mu_2+1<\mu_3$ and $-\mu_2 < \lambda_3$, then $d \neq 0$.
\end{enumerate}
\end{prop}

\begin{proof}
Consider an element \begin{equation*}\label{E:13coset}
A:= \begin{pmatrix} a& b& c\\d&e&0\\g&0&0 \end{pmatrix} \in \begin{pmatrix}P^{\mu_1+1} & P^{\mu_1+1} &
P^{\mu_1}_{\times}\\ P^{\mu_2+1} & P^{\mu_2}_{\times} & 0 \\ P^{\mu_3}_{\times} &0 & 0
\end{pmatrix} = K_1.
\end{equation*}  Here we compute that $D = -\ov{d}ge$, and so if $d \neq 0$, then $D\neq 0$.  We shall handle the two cases $d = 0$ and $d \neq 0$ separately.

\vskip 10 pt

First assume that $d\neq 0$.  Then $e_1$ is a cyclic vector for $(F^3,A\sigma)$, and Equations \eqref{E:alpha} and \eqref{E:beta} reduce to
\begin{equation*} \alpha = -\ov{\ov{a}} - \frac{\ov{\ov{d}}\ov{e}}{\ov{d}}, \quad \beta = \frac{\ov{\ov{d}}\ov{a}\ov{e}}{\ov{d}} - \ov{\ov{d}}\ov{b} - \ov{\ov{g}}\ov{c}. \end{equation*} Note that $\ov{\ov{d}}\ov{e}/\ov{d} \in \mathcal{O}$, since $\mu_2 \geq 0$.  Thus, $\alpha \in P^{-\lambda_1} \iff \ov{\ov{a}} \in P^{-\lambda_1} \iff a  \in P^{-\lambda_1}$.

Computing the valuations of the summands of $\beta$, we see that \begin{equation*} \frac{\ov{\ov{d}}\ov{a}\ov{e}}{\ov{d}} \in P^{\mu_1+\mu_2+1}, \quad \ov{\ov{d}}\ov{b} \in P^{\mu_1+\mu_2+2}, \quad \ov{\ov{g}}\ov{c} \in P^{\mu_1+\mu_3}_{\times}. \end{equation*} Observe that since $\alpha \in P^{\mu_1+1}$ and $\beta \in P^{\mu_1+\mu_2+1}$, Lemma \ref{T:charpolypo} implies Equation \eqref{E:13gen}.

Once more, we see that our analysis of $\beta$ naturally determines two subcases, in the event that $\mu_2+1<\mu_3$.  As before, we consider these subcases in turn, at first under the additional assumption that $\mu_2+1<\mu_3$.

\vskip 5 pt Subcase ($i$): $\mu_1+\mu_2+1 \leq \lambda_3\leq \mu_1 + \mu_3$, and $\mu_2+1<\mu_3$
\vskip 5 pt

In this subcase, we observe that $\ds \beta \in P^{\lambda_3} \iff \frac{\ov{\ov{d}}\ov{a}\ov{e}}{\ov{d}} - \ov{\ov{d}}\ov{b} \in P^{\lambda_3} \iff \ov{\ov{d}}\ov{a}\ov{e} - \ov{\ov{d}}\ov{d}\ov{b} \in P^{\lambda_3 + \val(d)} \iff \ov{\begin{vmatrix} a&b\\d&e \end{vmatrix}} \in P^{\lambda_3} \iff \begin{vmatrix} a&b\\d&e \end{vmatrix} \in P^{\lambda_3}$, as desired.

\vskip 5 pt Subcase ($ii$): $\mu_1+\mu_3 < \lambda_3$, and  $\mu_2+1<\mu_3$ \vskip 5 pt

We again use our estimate on $\val(a)$ to compute that $\ov{\ov{d}}\ov{ae}/\ov{d} \in P^{\mu_2 -
\lambda_1} \subseteq P^{\lambda_3}$ by Lemma \ref{T:irrelterms}. Thus, $\beta \in P^{\lambda_3} \iff
\ov{\ov{d}}\ov{b} + \ov{\ov{g}}\ov{c} \in P^{\lambda_3}  \iff
\ov{d}b + \ov{g}c \in P^{\lambda_3}$. By Lemma \ref{T:charpolypo}, the proposition is true in the case where $\mu_2+1<\mu_3$ and $d \neq 0$.

\vskip 5 pt Now assume that $\mu_2+1=\mu_3$. Computing the valuations of the summands of $\beta$ in this case, we see that \begin{equation*} \frac{\ov{\ov{d}}\ov{a}\ov{e}}{\ov{d}} \in P^{\mu_1+\mu_2+1}, \quad \ov{\ov{d}}\ov{b} \in P^{\mu_1+\mu_2+2}, \quad \ov{\ov{g}}\ov{c} \in P^{\mu_1+\mu_2+1}_{\times}. \end{equation*} Therefore, $\beta \in P^{\mu_1+\mu_2+1} = P^{-\mu_2}$ in this special case.  By Lemma \ref{T:charpolypo}, we thus see that $\nu_x = -(\mu_1+1,\mu_2,\mu_2)$.  Now, if $\lambda_3 > \mu_1+\mu_2+1$, then we must also have $\lambda_1 < -\mu_1-1$.  However, we see that  $\mu_2-\lambda_1 > \mu_1+\mu_2+1$, and so $\ov{\ov{d}}\ov{ae}/\ov{d} \in P^{\mu_2-\lambda_1}\subsetneq P^{\mu_1+\mu_2+1}$, since $a \in P^{-\lambda_1}$. But $\ov{\ov{g}}\ov{c} \in P^{\mu_1+\mu_2+1}_{\times} \iff \ov{g}c \in P^{\mu_1+\mu_2+1}_{\times}$, and so we must have $\beta \in P^{\mu_1+\mu_2+1}_{\times}$ and $\lambda_3=-\mu_3+1$ is fixed.  Since we also have that $\ov{\ov{d}}\ov{b} \in P^{\mu_1+\mu_2+2} \iff \ov{d}b \in P^{\mu_1+\mu_2+2} \subsetneq P^{\mu_1+\mu_2+1}$, Lemma \ref{T:charpolypo} yields Equation \eqref{E:13ii} in the case in which $\mu_2+1=\mu_3$ and $d \neq 0$.

\vskip 5 pt

Now assume that $d=0$.  In this case, we observe that $\left\{
\begin{pmatrix} y \\0\\z \end{pmatrix} \biggm| y,z \in F \right\} \cong F^2$ is a subspace fixed by $A\sigma$.
That is, $\left( F^2, \begin{pmatrix}a&c\\g&0\end{pmatrix}\sigma \right)$ is a sub-isocrystal of $(F^3, A\sigma)$, giving rise to the following short exact sequence:
\begin{equation*} 0 \longrightarrow \left( F^2, \begin{pmatrix}a&c\\g&0\end{pmatrix}\sigma \right) \longrightarrow (F^3, A\sigma) \longrightarrow (F, e\sigma) \rightarrow 0\ .
\end{equation*}  By Lemma \ref{T:split}, this short exact sequence splits, so that $(F^3, A\sigma)$ decomposes as the direct sum of the two sub-isocrystals.  The Newton slope sequence for the direct sum is obtained by ordering the Newton slopes of the two sub-isocrystals \cite{Kat}, so it suffices to understand the Newton strata for these two isocrystals.

Note that $\val(e)=\mu_2$ so that the only Newton slope sequence occurring for $(F,e\sigma)$ is $-\mu_2$.  Therefore, the Newton polygon for $(F^3, A\sigma)$ satisfies either $\lambda_3 = -\mu_2$ or $\lambda_2 = -\mu_2$, depending on whether or not $-\lambda_1+\mu_2\geq -\mu_2$.  In either case, note that $\lambda_3 \leq -\mu_2$, so that if $\mu_2+1<\mu_3$ we are necessarily in subcase ($i$).

Applying Lemma \ref{T:GL_2} to $(F^2, \begin{pmatrix} a&c\\g&0\end{pmatrix}\sigma)$, we see that $e_1$ is a cyclic vector,  since $\val(g) = \mu_3$.  We now define some notation to distinguish our restriction of $SL_3(F)$ to the copy of $GL_2(F)$ corresponding to this sub-isocrystal. Denote by $x' := \pi^{(\mu_1,\mu_3)}w'$, where $w'=\begin{pmatrix} 0&1\\1&0\end{pmatrix}$ is the restriction of $w$ to the affine Weyl group for $GL_2(F)$.  Let $\eta \in \mathcal{N}(G)_{x'}$.  Using this notation, applying Lemma \ref{T:GL_2} says that \begin{equation*} (x'I)_{\leq \eta} = \left\{ y \in x'I \biggm| \ov{a} \in P^{-\eta_1} \right\}.\end{equation*}  Altogether, Lemmas \ref{T:split} and \ref{T:GL_2}, together with the fact that $\ov{a} \in P^{-\lambda_1} \iff a \in P^{-\lambda_1}$, say that if $d=0$, then \begin{equation}\label{E:13d=0} (K_1)_{\leq
\lambda} = \left\{ \begin{pmatrix} a& b& c\\d&e&0\\g&0&0 \end{pmatrix} \in K_1 \biggm| a \in P^{-\lambda_1}\right\}. \end{equation} We thus see that the first slope $\lambda_1$ determines only $\val(a)$, and in this case, $a \in P^{\mu_1+1}$.  In addition, either $\lambda_3 = -\mu_2 \geq -\mu_3 +1$, or $\lambda_3=-\lambda_1-\lambda_2=-\lambda_1+\mu_2 \geq \mu_1+\mu_2+1 = -\mu_3+1$.  Therefore, $\lambda_3 \geq -\mu_3 +1$, whether or not $\lambda_2 = -\mu_2$ or $\lambda_3 = -\mu_2$.  Lemma \ref{T:charpolypo} then indicates that Equation \eqref{E:13gen} still applies.

It remains to show that Equations \eqref{E:13i} and \eqref{E:13d=0} are equivalent under the hypothesis $d=0$.  Let us consider the case in which $\lambda_3 = -\mu_2$.  Then we see that $\lambda_2 \geq \lambda_3 \iff -\lambda_1-\lambda_3 \geq \lambda_3 \iff-\lambda_1 + \mu_2 \geq \lambda_3$.  Therefore, $ae \in P^{-\lambda_1 + \mu_2} \subseteq P^{\lambda_3}$.  In the situation in which $\lambda_2 = -\mu_2$, we have $-\lambda_1-\lambda_2 = \lambda_3 \iff -\lambda_1+\mu_2 = \lambda_3$, so that we again automatically have $ae \in P^{\lambda_3}$.  Therefore, Equations \eqref{E:13i} and \eqref{E:13d=0} are equivalent in the case where $d=0$, as desired.
\end{proof}

For the remaining three cases, we employ arguments similar to those used in the proof of Proposition
\ref{T:13vals}.  We begin by stating the analog of Lemma \ref{T:13zeros} for case IVA.  The proof proceeds in
manner similar to Lemma \ref{T:13zeros} and is thus omitted.

\begin{lemma}\label{T:12zeros}
Let $x = \pi^{\mu}s_1$, where $\mu = (\mu_1,\mu_2,\mu_3)$ satisfies $\mu_1 < \mu_2 \leq \mu_3$.  Define $$ J_2:=
\begin{pmatrix} 1& 0& 0\\P^{\mu_2-\mu_1}&1&0\\P^{\mu_3-\mu_1+1}&0&1 \end{pmatrix}\quad
\text{and}\quad  K_2:= \begin{pmatrix} P^{\mu_1+1}&P^{\mu_1}_{\times} & P^{\mu_1}
\\  P^{\mu_2}_{\times}&0 & P^{\mu_2}\\ P^{\mu_3+1}&0 & P^{\mu_3}_{\times}
\end{pmatrix} .$$  Then the map $\kappa: J_2 \times K_2  \rightarrow xI$ given by $(j,k) \mapsto j^{-1}k\sigma(j)$ is an isomorphism of schemes.
\end{lemma}

\begin{cor}\label{T:K_2cor}
Let $x = \pi^{\mu}s_1$, where $\mu = (\mu_1,\mu_2,\mu_3)$ satisfies $\mu_1 < \mu_2 \leq \mu_3$.  Denote
$(K_2)_{\leq \lambda}:= K_2 \cap (xI)_{\leq \lambda}$.  Then \begin{equation*}\label{E:K_2cor} \codim((K_2)_{\leq
\lambda} \subseteq K_2) = \codim((xI)_{\leq \lambda}\subseteq xI).\end{equation*}
\end{cor}

\begin{prop}[\textbf{Case IVA}]\label{T:12vals}Let $x=\pi^{\mu}s_1$ satisfy $\mathbf{a}_x \subset C^0$ and $\mu_2 \geq 0$.  Then \linebreak $\mu_1 < \mu_2\leq \mu_3$.  In addition, $\mu_1 < 0$ and $\mu_3 >0$.

Now fix $\lambda=(\lambda_1,\lambda_2,\lambda_3) \in \mathcal{N}(G)_x$.  We then have
\begin{equation}\label{E:12gen} \lambda_1 \leq -\mu_1-\frac{1}{2} \quad \text{and} \quad \lambda_3 =
-\mu_3.\end{equation}  Further, the only possibility in which $\lambda_1 = -\mu_1-\frac{1}{2}$ is for
$\mu=(-1,0,1)$.  Otherwise, $\lambda_1 \leq -\mu_1 - 1$.  Finally, we have the following description of
$(K_2)_{\leq \lambda} := K_2 \cap (xI)_{\leq \lambda}$: \begin{equation}\label{E:12} (K_2)_{\leq \lambda} =
\left\{ \begin{pmatrix} a& b& c\\d&0&f\\g&0&i \end{pmatrix} \in K_2 \biggm| a \in
P^{-\lambda_1}\right\}. \end{equation}
\end{prop}

\begin{proof}
Consider an element
\begin{equation*}\label{E:12vals} A =
\begin{pmatrix}a&b&c\\d&0&f\\g&0&i\end{pmatrix} \in \begin{pmatrix} P^{\mu_1+1}&P^{\mu_1}_{\times} & P^{\mu_1}
\\  P^{\mu_2}_{\times}&0 & P^{\mu_2}\\ P^{\mu_3+1}&0 & P^{\mu_3}_{\times}
\end{pmatrix}=K_2.\end{equation*}  Here we see that $D = \ov{g}\begin{vmatrix}d&f\\g&i \end{vmatrix}$, so that $D\neq
0$ if $g \neq 0$.

Assume first that $g \neq 0$ so that $e_1$ is a cyclic vector for $(F^3,A\sigma)$. Using that $e=h=0$, Equations
\eqref{E:alpha} and \eqref{E:beta} reduce to \begin{equation*} \alpha = -\ov{\ov{a}} -
\frac{\ov{\ov{g}}\ov{i}}{\ov{g}}, \quad \beta = \frac{\ov{\ov{g}}\ov{ai}}{\ov{g}} - \ov{\ov{d}}\ov{b} -
\ov{\ov{g}}\ov{c}. \end{equation*}  First, observe that $\ov{\ov{g}}\ov{i}/ \ov{g} \in \mathcal{O},$ since $\mu_3
\geq 0$.  In the special case $\mu = (-1,0,1)$, we have $\alpha \in \mathcal{O}$, so that $\ov{\ov{a}} \in
P^{-\lambda_1}$ automatically. For all other $\mu$ satisfying the hypotheses of this proposition, we see that
$\alpha \in P^{-\lambda_1} \iff \ov{\ov{a}} \in P^{-\lambda_1} \iff a \in P^{-\lambda_1}$.  In addition, if $\mu \neq (-1,0,1)$ note that
$\lambda_1 \leq -\mu_1-1$, since $a \in P^{\mu_1+1}$.

Now, comparing valuations of the summands of $\beta$, we have \begin{equation*}
\frac{\ov{\ov{g}}\ov{ai}}{\ov{g}} \in P^{\mu_1+\mu_3+1}, \quad \ov{\ov{d}}\ov{b} \in P^{\mu_1+\mu_2}_{\times},
\quad \ov{\ov{g}}\ov{c} \in P^{\mu_1+\mu_3+1}. \end{equation*}  Since $\mu_2 \leq \mu_3$, we observe that
$\val(\beta) = \mu_1+\mu_2 = -\mu_3$.  In the special case $\mu=(-1,0,1)$, we see that the only
possibility is that $\lambda = (\frac{1}{2},\frac{1}{2},-1).$  Equation \eqref{E:12gen} consequently holds.  We
have now shown that $\lambda$ determines only $\val(a)$, and Lemma \ref{T:charpolypo} implies the result if
$g\neq 0$.

\vskip 5 pt

If $g=0$, we have the following split exact sequence: \begin{equation*} 0 \longrightarrow \left( F^2,
\begin{pmatrix}a&b\\d&0\end{pmatrix}\sigma \right) \longrightarrow (F^3, A\sigma) \longrightarrow (F, i\sigma)
\rightarrow 0\ .\end{equation*}  Observe that $\val(i) = \mu_3$, so we need only understand the conditions
on the two-dimensional sub-isocrystal.  Since $d \neq 0$, we again use Lemma \ref{T:GL_2} to see that
$\lambda_1$ determines only $\val(a)$, as desired.  Finally, since $a \in P^{\mu_1+1}$, Lemma
\ref{T:charpolypo} indicates that the inequalities in \eqref{E:12gen} still hold.
\end{proof}

\begin{lemma}\label{T:23zeros}
Let $x = \pi^{\mu}s_2$, where $\mu = (\mu_1,\mu_2,\mu_3)$ satisfies $\mu_1 \leq \mu_2 < \mu_3$.  Define $$ J_3:=
\begin{pmatrix} 1& 0& 0\\ 0&1&0\\P^{\mu_3-\mu_1+1}&P^{\mu_3-\mu_2}&1 \end{pmatrix}\quad
\text{and}\quad  K_3:=  \begin{pmatrix} P^{\mu_1}_{\times}&P^{\mu_1} & P^{\mu_1} \\  P^{\mu_2+1}&P^{\mu_2+1} &
P^{\mu_2}_{\times}\\ 0&P^{\mu_3}_{\times} & 0 \end{pmatrix} .$$  Then the map $\kappa: J_3 \times K_3
\rightarrow xI$ given by $(j,k) \mapsto j^{-1}k\sigma(j)$ is an isomorphism of schemes.
\end{lemma}

\begin{cor}\label{T:K_3cor}
Let $x = \pi^{\mu}s_2$, where $\mu = (\mu_1,\mu_2,\mu_3)$ satisfies $\mu_1 \leq \mu_2 < \mu_3$.  Denote
$(K_3)_{\leq \lambda}:= K_3 \cap (xI)_{\leq \lambda}$.  Then \begin{equation*}\label{E:K_3cor} \codim((K_3)_{\leq
\lambda} \subseteq K_3) = \codim((xI)_{\leq \lambda}\subseteq xI).\end{equation*}
\end{cor}

\begin{prop}[\textbf{Case VA}]\label{T:23vals}Let $x=\pi^{\mu}s_2$ satisfy $\mathbf{a}_x \subset C^0$ and $\mu_2 \geq 0$. Then  $\mu_1 < \mu_2< \mu_3$.  In addition, $\mu_1 < 0$ and $\mu_3 >0$.

Now fix $\lambda=(\lambda_1,\lambda_2,\lambda_3) \in \mathcal{N}(G)_x$. We then have \begin{equation}\label{E:23gen} \lambda_1 = -\mu_1 \quad \text{and} \quad \lambda_3 \geq -\mu_3 + \frac{1}{2}.\end{equation} Further, the only possibility in which $\lambda_3 = -\mu_3 +\frac{1}{2}$ is for $\mu_2+1=\mu_3$.  Otherwise, $\lambda_3 \geq -\mu_3 +1$. In describing $(K_3)_{\leq \lambda} := K_3 \cap (xI)_{\leq \lambda}$, we have the following two subcases:
\begin{enumerate}
\item[($i$)] If $\mu_2+1 = \mu_3$, then $\lambda = (-\mu_1, \frac{\mu_1}{2}, \frac{\mu_1}{2})$ and $(K_3)_{\leq \lambda} = K_3.$
\item[($ii$)] If $\mu_2+1 < \mu_3$, then \begin{equation}\label{E:23} (K_3)_{\leq \lambda} = \left\{ \begin{pmatrix} a& b& c\\d&e&f\\0&h&0 \end{pmatrix} \in K_3 \biggm| \begin{vmatrix} a & b\\ d& e\end{vmatrix} \in P^{\lambda_3}\right\}. \end{equation}
\end{enumerate}
\end{prop}

\begin{proof}
Let $A \in K_3$.  Then we may write \begin{equation*}\label{E:23vals} A =
\begin{pmatrix}a&b&c\\d&e&f\\0&h&0\end{pmatrix} \in \begin{pmatrix} P^{\mu_1}_{\times}&P^{\mu_1} & P^{\mu_1} \\
P^{\mu_2+1}&P^{\mu_2+1} & P^{\mu_2}_{\times}\\ 0&P^{\mu_3}_{\times} & 0 \end{pmatrix},\end{equation*} and in this
context, $D = \ov{d}dh$.  Therefore, $e_1$ is a cyclic vector if $d \neq 0$.

First assume that $d\neq 0$ so that Equations \eqref{E:alpha} and \eqref{E:beta} hold.  Using that $g = i = 0$,
these equations reduce to
\begin{equation*} \alpha = -\ov{\ov{a}} - \frac{\ov{\ov{d}}\ov{e}}{\ov{d}}, \quad \beta =
\frac{\ov{\ov{d}}\ov{ae}}{\ov{d}} - \ov{\ov{d}}\ov{b}
-\frac{\ov{D}}{D}fh.
\end{equation*} First, observe that $\ov{\ov{d}}\ov{e}/ \ov{d} \in \mathcal{O},$ since $\mu_2 \geq 0$.  Thus,
$\alpha \in P^{-\lambda_1} \iff \ov{\ov{a}} \in P^{-\lambda_1} \iff a \in P^{-\lambda_1}$.  However, $a \in P^{\mu_1}_{\times}$, so the
only possibility is that $\lambda_1=-\mu_1.$

Similarly, we can compute that
\begin{equation*}
\frac{\ov{\ov{d}}\ov{ae}}{\ov{d}} \in  P^{\mu_1+\mu_2+1}, \quad -\ov{\ov{d}}\ov{b}\in P^{\mu_1+\mu_2+1},
\quad -\frac{\ov{D}}{D}fh \in P^{\mu_2+\mu_3}_{\times}\subset \mathcal{O}. \end{equation*}
We now consider the two subcases defined in the statement of the proposition.

\vskip 10 pt Subcase ($i$): $\mu_2+1 = \mu_3$ \vskip 10 pt

If $\mu_2+1 = \mu_3$, note that $\mu_1+\mu_2 + 1 > \frac{\mu_1}{2}$.  Consequently, $\lambda_3 = \frac{\mu_1}{2}$ is fixed.  But because we also know that $\lambda_1 = -\mu_1$, we thus see that $\lambda = (-\mu_1, \frac{\mu_1}{2},\frac{\mu_1}{2})$ is the only possible value for $\lambda$ in this subcase. That is,  $\mathcal{N}(G)_x$ consists of a single element, and so $(K_3)_{\leq \lambda} = K_3$.

\vskip 10 pt Subcase ($ii$): $\mu_2+1 < \mu_3$ \vskip 10 pt

We first observe that $\beta \in P^{\mu_1+\mu_2+1}$ so that if $\mu_2+1 < \mu_3$, then $\lambda_3 \geq -\mu_3+1.$  In particular, we have verified Equation \eqref{E:23gen}.  Comparing the valuations of the summands of $\beta$, it appears that we should consider two further subcases:
\begin{enumerate}
\item[($a$)] $\mu_1+\mu_2+1 \leq \lambda_3\leq \mu_1+\mu_3$,
\item[($b$)] $\mu_1+\mu_3 < \lambda_3$.
\end{enumerate}
First, we argue that subcase ($b$) does not actually arise.  By Lemma \ref{T:irrelterms}, we know that $\mu_2-\lambda_1\geq \lambda_3$ for this range of $\lambda_3$.  In addition, recall that by our analysis of $\val(\alpha)$, we always have that $-\lambda_1 = \mu_1$.  Thus, $\mu_2-\lambda_1 = \mu_1 + \mu_2 < \mu_1 + \mu_3 < \lambda_3$, which is a contradiction.

If $\mu_1+\mu_2+1 \leq \lambda_3\leq \mu_1+\mu_3$, we see that $\ds \beta \in P^{\lambda_3} \iff
\frac{\ov{\ov{d}}\ov{a}\ov{e}}{\ov{d}} - \ov{\ov{d}}\ov{b} \in P^{\lambda_3} \iff \ov{\ov{d}}\ov{a}\ov{e} -
\ov{\ov{d}}\ov{d}\ov{b} \in P^{\lambda_3 + \val(d)} \iff \ov{\begin{vmatrix} a&b\\d&e \end{vmatrix}} \in
P^{\lambda_3} \iff \begin{vmatrix} a&b\\d&e \end{vmatrix} \in
P^{\lambda_3}$.  Equation \eqref{E:23} then follows by Lemma \ref{T:charpolypo}.

\vskip 5 pt

If, on the other hand $d=0$, we have the following split exact sequence: \begin{equation*} 0 \longrightarrow
 (F, a\sigma) \longrightarrow (F^3, A\sigma) \longrightarrow \left( F^2,
\begin{pmatrix}e&f\\h&0\end{pmatrix}\sigma \right)
\rightarrow 0\ .\end{equation*}  Observe that $\val(a) = \mu_1$, so we need only understand the conditions
on the two-dimensional sub-isocrystal.  Since $h \neq 0$, we again use Lemma \ref{T:GL_2} to see that
$\lambda_3$ determines only $\val(e)$.  In particular, we must have $e \in P^{\lambda_3-\mu_1}
\iff ae \in P^{\lambda_3}$, since $\val(a) = \mu_1$.  Therefore, when $d=0$, Equation \eqref{E:23}
holds. Finally, since $ae \in P^{-\mu_3+1}$, Lemma \ref{T:charpolypo} indicates that the inequalities in
\eqref{E:23gen} are still true.
\end{proof}

The same methods used in the previous propositions in this subsection can be employed in the case where $x$ is a pure translation to show that $\val(\alpha) = \mu_1$ and $\val(\beta) = -\mu_3$; \textit{i.e.}, the set $\mathcal{N}(G)_x$ consists of a single Newton slope sequence.  Alternatively, we can cite the following more general result of \cite{GHKRadlvs}.

\begin{theorem}\label{T:translation}
Let $\mu \in X_*(T)$.  Then any element of $I\pi^{\mu}I$ is $\sigma$-conjugate under $I$ to $\pi^{\mu}$.
\end{theorem}

Using either argument, we obtain the following proposition.

\begin{prop}[\textbf{Case VIA}]\label{T:1vals}If $x=\pi^{\mu}$, then $ \mathcal{N}(G)_x = \{ -\mu\}.$
\end{prop}

The hypotheses on $\mu$ in the statements of the propositions in this section omit $\mu=0$.  Although only the
base alcove $\mathbf{a}_1$ satisfies $\mu=0$ and $\mathbf{a}_1 \subset C^0$, we discuss the case $\mu=0$ for the
sake of completeness.   Recall the automorphism $\varphi$ from Lemma \ref{T:reduction}, which represents rotation
by 120 degrees about the center of the base alcove and induces the identity on $\mathcal{N}(G)$.  If
$x=\pi^{(-1,0,1)}s_1s_2$, for example, then $\varphi(x) = \pi^0s_1s_2$ so that our analysis in Proposition
\ref{T:123vals} yields that $\mathcal{N}(G)_x = \{0\}$.
Similarly, if $\mu=0$, then one can verify that $\mathcal{N}(G)_x = \{(0,0,0)\}$ for any $w \in W$, completing our analysis of case A.

\subsection{Conditions on valuations determining the Newton strata: Case B}\label{S:valpolysB}
We now argue that computing the explicit form of the variety $(xI)_{\leq \lambda}$ in case B proceeds in exactly the same manner as in case A.  The idea is to use a change of basis as in the proof of Proposition \ref{T:132vals} to reduce the computations required for case B to ones essentially the same as the calculations performed in Section \ref{S:valpolysA}.

More specifically, let $x = \pi^{\mu}w$ satisfy $\mathbf{a}_x \subset s_1(C^0)$, where $\mu_1 \geq 0$ and $\mu \neq (\mu_1,\mu_2,\mu_1)$.  Consider \begin{equation}\label{E:x'I'}s_1^{-1}\pi^{\mu}wIs_1 = \pi^{{\mu}_{\text{dom}^*}}(s_1^{-1}ws_1)(s_1^{-1}Is_1)=:x'I'.\end{equation}  By
$\mu_{\text{dom}^*}$ we mean the unique antidominant element in the Weyl orbit of $\mu$.  Here, $I'=s_1^{-1}Is_1$ is the
non-standard Iwahori subgroup
\begin{equation}\label{E:I'} I' = \begin{pmatrix} \mathcal{O}^{\times} &P &\mathcal{O}\\ \mathcal{O} & \mathcal{O}^{\times}
&\mathcal{O} \\ P &P & \mathcal{O}^{\times} \end{pmatrix}. \end{equation}  The varieties $(x'I)_{\leq \lambda'}$ for $\lambda' \in \mathcal{N}(G)_{x'}$ are precisely the ones described in the previous section.  As we demonstrate below, by replacing $I$ by $I'$ in the propositions from Section \ref{S:valpolysA}, we obtain a complete description of $(x'I')_{\leq \lambda'}$ for all possible values of $x'$.  To illustrate this phenomenon, we briefly present the calculation for $w=s_2s_1$ in case B, which mirrors case IA, since $s_1^{-1}(s_2s_1)s_1=s_1s_2$.  The other five arguments proceed in a similar fashion.

\begin{prop}[\textbf{Case IIB}]\label{T:132Bvals} Let $x=\pi^{\mu}s_2s_1 \in \widetilde{W}$ satisfy $\mathbf{a}_x \subset s_1(C^0)$, where $\mu_1 \geq 0$ and $\mu \neq (\mu_1,\mu_2,\mu_1)$.  Then $\mu_2 < \mu_1 < \mu_3$.  In addition, $\mu_2 < 0$ and $\mu_3 >0$.

Consider $x' = \pi^{(\mu_2,\mu_1,\mu_3)}s_1s_2$, and fix $\lambda'=(\lambda_1,\lambda_2,\lambda_3) \in \mathcal{N}(G)_{x'}$.  We then have \begin{equation*}\label{E:132Bgen}\lambda_1 \leq -\mu_2 - 1 \quad \text{and} \quad \lambda_3 \geq -\mu_3 + 1.\end{equation*}  Recall the non-standard Iwahori subgroup $I'$ defined in Equation \eqref{E:I'}.  To describe $(x'I')_{\leq \lambda'}$, we have the following two subcases:
\begin{enumerate}
\item[($i$)] If $-\mu_3+1\leq \lambda_3\leq -\mu_1$, then \begin{equation*} (x'I')_{\leq \lambda'} = \left\{ \begin{pmatrix} a& b& c\\d&e&f\\g&h&i \end{pmatrix} \in x'I' \biggm| a \in P^{-\lambda_1} \ \text{and}\ \begin{vmatrix} a & b\\ d& e\end{vmatrix} \in P^{\lambda_3}\right\}. \end{equation*}
\item[($ii$)] If $-\mu_1 < \lambda_3$, then \begin{equation*} (x'I')_{\leq \lambda'} = \left\{ \begin{pmatrix} a& b& c\\d&e&f\\g&h&i \end{pmatrix} \in x'I' \biggm| a \in P^{-\lambda_1}\ \text{and}\ \ov{d}b+\ov{g}c \in P^{\lambda_3}\right\}. \end{equation*} Note that subcase ($ii$) only arises when $\mu_1 >0$, since $\lambda_3$ is always non-positive.
\end{enumerate}
\end{prop}

\begin{proof}
If $A \in x'I'$, then we have \begin{equation*}\label{E:132Bcoset}
A:= \begin{pmatrix} a& b& c\\d&e&f\\g&h&i \end{pmatrix} \in \begin{pmatrix}P^{\mu_2+1} & P^{\mu_2+1} &
P^{\mu_2}_{\times}\\ P^{\mu_1}_{\times} & P^{\mu_1+1} & P^{\mu_1} \\ P^{\mu_3} & P^{\mu_3}_{\times} & P^{\mu_3}
\end{pmatrix} = x'I',
\end{equation*} so that $e_1$ is a cyclic vector for $(F^3,A\sigma)$.  Now, observe that
$$\frac{1}{D}\left( (\ov{\ov{d}}\ov{e}+\ov{\ov{g}}\ov{f})\begin{vmatrix} d&e\\g&h\end{vmatrix} +
(\ov{\ov{d}}\ov{h}+\ov{\ov{g}}\ov{i})\begin{vmatrix}d&f\\g&i\end{vmatrix}\right) \in \mathcal{O},$$ since $\mu_1 \geq 0$.  Therefore, $\alpha \in P^{-\lambda_1} \iff \ov{\ov{a}} \in P^{-\lambda_1} \iff a \in P^{-\lambda_1}$. In addition, since in this case $a \in P^{\mu_2+1}$, we see that $\lambda_1\leq -\mu_2-1$.

Now we consider the condition $\beta \in P^{\lambda_3}$.  Compute that
\begin{align*}
B_1 &\in  P^{\mu_1+\mu_2+2} & B_2 &\in P^{\mu_2+\mu_3+1}\\
- \ov{\ov{d}}\ov{b}&\in P^{\mu_1+\mu_2+1} &  -\ov{\ov{g}}\ov{c} &\in P^{\mu_2+\mu_3}\\
\frac{\ov{D}}{D}\begin{vmatrix}e&f\\h&i\end{vmatrix} &\in P^{\mu_1+\mu_3} \subset \mathcal{O}. & \phantom{}
\end{align*}  In particular, $\beta \in P^{\mu_1+\mu_2+1}$ so that $\lambda_3 \geq -\mu_3 + 1.$  Comparing the valuations of the summands of $\beta$, we see that we again have two subcases:
\begin{enumerate}
\item[($i$)] $\mu_1+\mu_2+1 \leq \lambda_3\leq \mu_2+\mu_3$,
\item[($ii$)] $\mu_2+\mu_3 < \lambda_3$.
\end{enumerate}
The analysis of these two subcases proceeds in the same manner as in the proof of Proposition \ref{T:123vals}, which the reader may easily verify.
\end{proof}

Observe that the varieties $(xI)_{\leq \lambda}$ and $(x'I')_{\leq \lambda'}$ differ only by a change of basis, so that the description for case IIB provided in Proposition \ref{T:132Bvals} suffices.  Furthermore, the polynomials that describe $(x'I')_{\leq \lambda'}$ in case IIB are exactly the same as the ones appearing in case IA.  In fact, after performing the change of basis, the only difference between the analysis of these two cases is that the ranges for $\lambda_3$ that determine subcases ($i$) and ($ii$) are slightly different.  We generalize these observations for any $x$ satisfying the hypotheses of case B in the following remark.

\begin{remark}\label{T:valpolysB}
Let $x=\pi^{\mu}w$ satisfy $\mathbf{a}_x \subset s_1(C^0)$ with $\mu_1 \geq 0$ and $\mu \neq (\mu_1,\mu_2,\mu_1)$.  Conjugation by $s_1^{-1}$ transforms $x$ to $x'$, $I$ to $I'$, and $\lambda \in \mathcal{N}(G)_x$ to $\lambda' \in \mathcal{N}(G)_{x'}$, where $x'$ and $I'$ are defined by Equation \eqref{E:x'I'}.  Note, however, that the posets $\mathcal{N}(G)_x$ and $\mathcal{N}(G)_{x'}$ may be different and not even in bijection with one another.  In fact, it is not clear from our methods what the map between these two posets is. \begin{question}\label{T:posetmap} Can we explicitly describe the map $\mathcal{N}(G)_x \rightarrow \mathcal{N}(G)_{x'}$ induced by the map $xI \rightarrow x'I'$ given by conjugation by an element of the finite Weyl group? \end{question} \noindent This question is related to Question \ref{T:nuxform} about finding a closed formula for $\nu_x$.  Together with Conjecture \ref{T:domconj}, Question \ref{T:posetmap} would provide a means by which we could answer Question \ref{T:nonemptyQ} at least in the case of $SL_n(F)$, which asks for a complete description of $\mathcal{N}(G)_x$.

Independent of understanding the map between posets, however, we know that \begin{equation*} \codim((xI)_{\leq
\lambda} \subseteq xI) = \codim((x'I')_{\leq \lambda'} \subseteq x'I'). \end{equation*} Therefore, for the
purpose of the theorem, it suffices to compute $(x'I')_{\leq \lambda'}$.  As Proposition \ref{T:132Bvals}
illustrates, the only possible difference between $(x'I')_{\leq \lambda'}$ and the corresponding variety
$(x'I)_{\leq \lambda}$ from Section \ref{S:valpolysA}, is the range for $\lambda_3$ that defines the subcases.
We therefore leave the remainder of the verification to the reader, both here and in the proof of Theorem
\ref{T:main}.
\end{remark}

\end{section}

\begin{section}{The Poset of Newton Slope Sequences $\mathcal{N}(G)_x$}\label{S:slopes}
As a direct consequence of our calculations in Section \ref{S:valpolysA}, we can list the Newton slope sequences that arise for a particular $IxI$. This calculation is necessary in the proof of Theorem \ref{T:main} in that it provides a concrete number for the length of the segment $[\lambda, \nu_x]$ in the poset $\mathcal{N}(G)_x$.  The reader will recall from Section \ref{S:adlv} that one additional application of such a calculation is to determine for which $b \in G$ the affine Deligne-Lusztig variety $X_x(b) \neq \emptyset$.  We begin by explicitly describing the poset $\mathcal{N}(G)_x$ for $x$ such that $\mathbf{a}_x$ lies in the antidominant Weyl chamber.  We then explain the algorithm for computing $\mathcal{N}(G)_x$ for any $x \in \widetilde{W}$, although we omit the details.  These results are recorded in Tables \ref{Ta:s_12} through \ref{Ta:s_2} at the end of the paper.  An obvious consequence of understanding $\mathcal{N}(G)_x$ is that we obtain a list of the generic Newton slope sequences $\nu_x$.  We conclude this section by discussing some patterns for these generic Newton slope sequences.

\subsection{The set of Newton slopes in the antidominant Weyl chamber}\label{S:slopesC0}
In Section \ref{S:valpolysA}, for each $x \in \widetilde{W}$ satisfying $\mathbf{a}_x \subset C^0$ and $\mu_2 \geq 0$, we explicitly computed either $(xI)_{\leq \lambda}$ or $K_i \cap (xI)_{\leq \lambda}$ for $i \in \{1, 2, 3\}$.  We now use those calculations to describe the set of Newton slopes $\mathcal{N}(G)_x$ for any $x\in \widetilde{W}$ such that $\mathbf{a}_x \subset C^0$.

\pagebreak

\begin{prop} We study the cases $w = s_1s_2$ and $w=s_2s_1$ together.
\begin{enumerate}
\item Let $x=\pi^{\mu}s_1s_2$ and $\mathbf{a}_x \subset C^0$.  Then $\mu_1< \mu_2 \leq \mu_3$, and \begin{equation}\label{E:N(G)123}\mathcal{N}(G)_x =
\begin{cases}
 \{ \lambda \in \mathcal{N}(G) \mid \lambda \leq (-\mu_1-1, \frac{\mu_1+1}{2},\frac{\mu_1+1}{2}) \}, & \text{if $\mu_2=\mu_3$;}\\
 \{ \lambda \in \mathcal{N}(G) \mid \lambda \leq -\mu - (1,0,-1)\}, & \text{otherwise}.
\end{cases}
\end{equation}
\item Let $x=\pi^{\mu}s_2s_1$ and $\mathbf{a}_x \subset C^0$.  Then $\mu_1\leq \mu_2 < \mu_3$, and \begin{equation}\label{E:N(G)132}\mathcal{N}(G)_x =
\begin{cases}
 \{ \lambda \in \mathcal{N}(G) \mid \lambda \leq (\frac{\mu_3-1}{2}, \frac{\mu_3-1}{2},-\mu_3+1) \}, & \text{if $\mu_1=\mu_2$;}\\
 \{ \lambda \in \mathcal{N}(G) \mid \lambda \leq -\mu - (1,0,-1)\}, & \text{otherwise}.
\end{cases}
\end{equation}
\end{enumerate}
\end{prop}

\begin{proof}
We first examine $\mu$ satisfying the additional hypothesis $\mu_2\geq 0$, since this was a crucial hypothesis in Section \ref{S:valpolysA}.  We then apply the reflection $\psi$ from Lemma \ref{T:reduction}, which interchanges the two simple roots.  Recall that $\psi(x) = \pi^{(-\mu_3,-\mu_2, -\mu_1)}w'$, where the reduced expression for $w'$ is that for $w$ with all of the subscripts reversed.  In particular, $\psi(s_1s_2) = s_2s_1$, which motivates our studying these two cases together.

Let $x=\pi^{\mu}s_1s_2$, where $\mu_1< \mu_2 \leq \mu_3$ and $\mu_2 \geq 0$.  Recall from Proposition \ref{T:123vals} that $\lambda_1 \leq -\mu_1-1$ and  $\lambda_3 \geq -\mu_3+1$, unless $\mu_2=\mu_3$.  Therefore, by Lemma \ref{T:charpolypo}, if $\mu_2\neq \mu_3$, we have  $\lambda \leq -\mu - (1,0,-1)$.  In the special case $\mu_2=\mu_3$, Equation \eqref{E:123gen} indicates that $\lambda \leq  (-\mu_1-1, \frac{\mu_1+1}{2},\frac{\mu_1+1}{2})$.

Now consider $x=\pi^{\mu}s_2s_1$, where $\mu_1< \mu_2 < \mu_3$ and $\mu_2 \geq 0$.  Recall Equation \eqref{E:132gen} that implies that $\lambda \leq -\mu - (1,0,-1)$ holds without exception.

To demonstrate the reverse containment, it suffices to show that for any $\lambda$ in the designated ranges, we have $(xI)_{\lambda} \neq \emptyset$.  Let $x=\pi^{\mu}s_1s_2$, where $\mu_1< \mu_2 \leq \mu_3$ and $\mu_2 \geq 0$, and assume that $\lambda_3 \geq -\mu_3+1$.  Routine calculations demonstrate that \begin{equation*} \begin{pmatrix} \pi^{-\lambda_1} & \pi^{\lambda_3 - \mu_2} & \pi^{\mu_1} \\ \pi^{\mu_2} & 0 & 0 \\ 0 & \pi^{\mu_3} & 0 \end{pmatrix} \in (xI)_{\lambda}, \end{equation*}  where the reader will recall that $\pi^{\ell}:= \pi^{\lceil \ell \rceil}$, for $\ell \in \Q$. In the special case $\lambda =\linebreak (-\mu_1-1,\frac{\mu_1+1}{2},\frac{\mu_1+1}{2})$, we have that \begin{equation*} \begin{pmatrix} \pi^{\mu_1+1} & 0 & \pi^{\mu_1} \\ \pi^{\mu_2} & 0 & 0 \\ 0 & \pi^{\mu_3} & 0 \end{pmatrix} \in (xI)_{\lambda}. \end{equation*}

Now let $x=\pi^{\mu}s_2s_1$, where $\mu_1< \mu_2 < \mu_3$ and $\mu_2 \geq 0$.  To prove the opposite containment in Equation \eqref{E:N(G)132}, one can check that  \begin{equation*} \begin{pmatrix} \pi^{-\lambda_1} & \pi^{\mu_1} & 0 \\ \pi^{\lambda_3-\mu_1} & 0 & \pi^{\mu_2} \\ \pi^{\mu_3}&0 & 0 \end{pmatrix} \in (xI)_{\lambda}. \end{equation*}

Finally, observe that $\psi(\pi^{(\mu_1,\mu_2,\mu_3)}s_1s_2) = \pi^{-(\mu_3,\mu_2,\mu_1)}s_2s_1$ and vice versa.  By applying $\psi$ to the values of $x$ we have discussed, we see that Equations \eqref{E:N(G)123} and \eqref{E:N(G)132} are satisfied.
\end{proof}

Recall from Section \ref{S:valpolysA} that in Cases IIIA, IVA, and VA, we explicitly computed $(K_i)_{\leq \lambda} = \linebreak K_i \cap (xI)_{\leq \lambda}$, where $i =1,2,3$, respectively.  Note, however, that the obvious map $\ov{\nu}:(K_i)_{\leq \lambda} \rightarrow \mathcal{N}(G)$ has the same image as the map $\ov{\nu}: (xI)_{\leq \lambda} \rightarrow \mathcal{N}(G)$.  To compute $\mathcal{N}(G)_x$ in these three cases, it therefore suffices to use the results from Propositions \ref{T:13vals}, \ref{T:12vals}, and \ref{T:23vals}, in which we only consider elements in the corresponding subschemes $(K_i)_{\leq \lambda}$.

\begin{prop}\label{T:N(G)13}
Let $x=\pi^{\mu}s_1s_2s_1$ and $\mathbf{a}_x \subset C^0$.  Then $\mu_1 < \mu_2< \mu_3$, and \begin{equation*}\mathcal{N}(G)_x =
\begin{cases}
\{ \lambda \in \mathcal{N}(G) \mid (\frac{\mu_3-1}{2}, \frac{\mu_3-1}{2}, -\mu_3+1) \leq \lambda \leq -\mu-(1,0,-1)\}, & \text{if $\mu_2+1 = \mu_3$},\\
\{ \lambda \in \mathcal{N}(G) \mid (-\mu_1-1, \frac{\mu_1+1}{2}, \frac{\mu_1+1}{2}) \leq \lambda \leq -\mu-(1,0,-1)\}, & \text{if $\mu_1+1 = \mu_2$},\\
 \{ \lambda \in \mathcal{N}(G) \mid \lambda \leq -\mu - (2,0,-2)\}\cup \{-\mu-(1,0,-1)\}, & \text{otherwise}.
\end{cases}
\end{equation*}
\end{prop}

\begin{proof}
First assume that $\mu_2 \geq 0$.  Recall the inequalities from Equation \eqref{E:13gen} that demonstrate that $\mathcal{N}(G)_x \subseteq  \{ \lambda \in \mathcal{N}(G) \mid \lambda \leq -\mu - (1,0,-1)\}$. In the special case $\mu_2+1=\mu_3$, recall from Proposition \ref{T:13vals} that $\lambda_3 = -\mu_3+1$ is fixed.  Let $x = \pi^{\mu}s_1s_2s_1$ satisfy $\mathbf{a}_x \subset C^0$ and $\mu_2 \geq 0$. Note in this case that $\min\{\val(ae)\}=\mu_1+\mu_2+1$ and $\min\{\val(bd)\}=\mu_1+\mu_2+2$.  This observation implies that if $-\mu_3+1\leq \lambda_3 \leq -\mu_2$, and both $\lambda<-\mu-(1,0,-1)$ and $\lambda \nleq - \mu-(2,0,-2)$, then $(K_1)_{\lambda} = \emptyset$.

For the opposite containments in the case $\mu_2 \geq 0$, we again provide elements in $(K_1)_{\lambda}$ for all possible $\lambda$.  Here, our examples must be handled according to several cases.  First, note that \begin{equation*}\begin{pmatrix} \pi^{\mu_1+1} & 0 & \pi^{\mu_1} \\ 0 & \pi^{\mu_2} & 0 \\ \pi^{\mu_3} & 0 & 0\end{pmatrix} \in (K_1)_{\lambda}, \end{equation*} for $\lambda_3 = -\mu_3+1$. If $\lambda \leq -\mu - (2,0,-2)$, we may assume that $\mu_2+2\leq \mu_3$.  Consider the two subcases in the statement of Proposition \ref{T:13vals} and compute that \begin{align*} (i) \quad \begin{pmatrix} \pi^{-\lambda_1} & \pi^{-\lambda_1-1}+\pi^{\lambda_3-\mu_2-1} & \pi^{\mu_1} \\ \pi^{\mu_2+1} & \pi^{\mu_2} & 0 \\ \pi^{\mu_3} & 0 & 0\end{pmatrix} &\in (K_1)_{\lambda},\\ (ii) \quad \begin{pmatrix} \pi^{-\lambda_1} & -\pi^{\mu_1+1} & \pi^{\mu_1} \\ \pi^{\mu_3-1}+\pi^{\lambda_3-\mu_1-1} & \pi^{\mu_2} & 0 \\ \pi^{\mu_3} & 0 & 0\end{pmatrix} &\in (K_1)_{\lambda}. \end{align*}

Finally, since $\psi(s_1s_2s_1) = s_2s_1s_2= s_1s_2s_1$, the result for all $\mu$ follows.
\end{proof}

\begin{prop} We study the remaining cases $w=s_1$ and $w=s_2$ together.
\begin{enumerate}
\item Let $x=\pi^{\mu}s_1$ and $\mathbf{a}_x \subset C^0$.  Then $\mu_1 < \mu_2\leq \mu_3$, and \begin{equation*}\mathcal{N}(G)_x =
\begin{cases}
 \{ (\frac{\mu_3}{2}, \frac{\mu_3}{2},-\mu_3) \}, & \text{if $\mu_1+1=\mu_2$;}\\
 \{ \lambda \in \mathcal{N}(G) \mid  (\frac{\mu_3}{2}, \frac{\mu_3}{2},-\mu_3) \leq \lambda \leq -\mu - (1,-1,0)\}, & \text{otherwise}.
\end{cases}
\end{equation*}
\item Let $x=\pi^{\mu}s_2$ and $\mathbf{a}_x \subset C^0$.  Then $\mu_1 \leq \mu_2< \mu_3$, and \begin{equation*}\mathcal{N}(G)_x =
\begin{cases}
 \{ (-\mu_1, \frac{\mu_1}{2}, \frac{\mu_1}{2}) \}, & \text{if $\mu_2+1=\mu_3$;}\\
 \{ \lambda \in \mathcal{N}(G) \mid (-\mu_1, \frac{\mu_1}{2}, \frac{\mu_1}{2}) \leq \lambda \leq -\mu - (0,1,-1)\}, & \text{otherwise}.
\end{cases}
\end{equation*}
\end{enumerate}
\end{prop}

\begin{proof}
Let $x=\pi^{\mu}s_1$, where $\mu_1 < \mu_2\leq \mu_3$, and further assume that $\mu_2 \geq 0$. First consider $\mu \neq (-1,0,1)$, and recall from Proposition \ref{T:12vals} that $\lambda_3= -\mu_3$ is fixed and $\lambda_1 \leq -\mu_1-1$ holds.  Combining these observations yields $(\frac{\mu_3}{2}, \frac{\mu_3}{2},-\mu_3)\leq \lambda \leq -\mu-(1,-1,0)$.  In the special case $\mu = (-1,0,1)$, recall from Equation \eqref{E:12gen} that the only possibility is that $\lambda =(\frac{1}{2}, \frac{1}{2},-1)$.

 Conversely,  if $\mu \neq (-1,0,1)$
consider \begin{equation*} \begin{pmatrix} \pi^{-\lambda_1}& \pi^{\mu_1} & 0\\ \pi^{\mu_2} & 0&0 \\ \pi^{\mu_3+1}&0&
\pi^{\mu_3} \end{pmatrix} \in (K_2)_{\lambda}.\end{equation*}  In the case in which $\mu=(-1,0,1)$, consider \begin{equation*} \begin{pmatrix} 0 & \pi^{-1} & 0\\ 1 & 0&0 \\ \pi^{\mu_3+1}&0&
\pi \end{pmatrix} \in (K_2)_{\lambda}.\end{equation*}

Now let $x=\pi^{\mu}s_2$, where  $\mu_1 \leq \mu_2< \mu_3$ and $\mu_2 \geq 0$. Proposition \ref{T:23vals} says that if $\mu_2+1=\mu_3$, then $\mathcal{N}(G)_x$ consists of the single slope sequence $(-\mu_1, \frac{\mu_1}{2}, \frac{\mu_1}{2})$.  If $\mu_2+1<\mu_3$, then we showed in Proposition \ref{T:23vals} that $\lambda_1= -\mu_1$ is fixed and $\lambda_3 \geq -\mu_3+1$. Combining these observations yields $ (-\mu_1, \frac{\mu_1}{2}, \frac{\mu_1}{2}) \leq \lambda \leq -\mu - (0,1,-1)$, for $\mu_2+1 < \mu_3$.

To demonstrate the reverse containment, we must provide two classes of examples.  In the case in which $\mu_2+1<\mu_3$, consider the following element: \begin{equation*}
\begin{pmatrix} \pi^{\mu_1}& \pi^{\lambda_3 - \mu_2-1} & 0\\ \pi^{\mu_2+1} & 0&\pi^{\mu_2} \\ 0& \pi^{\mu_3}&0
\end{pmatrix} \in (K_3)_{\lambda}.\end{equation*} If, on the other hand, $\mu_2+1=\mu_3$, we have that  \begin{equation*}
\begin{pmatrix} \pi^{\mu_1}& 0 & 0\\  \pi^{\mu_2+1} & 0&\pi^{\mu_2} \\ 0& \pi^{\mu_3}&0
\end{pmatrix} \in (K_3)_{\lambda}.\end{equation*}

Finally, observe that $\psi(s_1) = s_2$ and vice versa to complete our description of $\mathcal{N}(G)_x$ for all $\mu$.
\end{proof}

\subsection{Generic Newton slopes in the antidominant Weyl chamber}\label{S:genslopesC0}

Except in some boundary cases in which the associated alcove is adjacent the wall of the antidominant Weyl chamber, our calculations in the previous section demonstrate that we have the following generic Newton slope sequences:
\begin{equation}\label{E:roughgens} \nu_x = \begin{cases}
 -\mu - (1,0,-1), &\text{if $w = s_1s_2,\ s_2s_1,$ or $s_1s_2s_1$;}\\
 -\mu - (1,-1,0), &\text{if $w = s_1$;}\\
 -\mu - (0,1,-1), &\text{if $w = s_2$;}\\
 -\mu, &\text{if $w = 1$},
\end{cases}
\end{equation}
for $x$ such that $\mathbf{a}_x \subset C^0$.  The alcoves in the antidominant Weyl chamber for which $\nu_x$ satisfies one of the above equalities correspond to alcoves lying in the \emph{shrunken} Weyl chamber, as defined by Reuman in \cite{Reu}.  In fact, in $SL_3(F)$ in general, the values of $x \in \widetilde{W}$ for which $\nu_x$ is half-integral always correspond to alcoves that lie outside these shrunken Weyl chambers, but not conversely, as we have seen.

Except when $x$ is a pure translation, we have seen that $\nu_x \neq -\mu_{\text{dom}}$ if $\mathbf{a}_x \subset
C^0$.  For every other value of $w \in \widetilde{W}$, there is a correction term.  Note in addition that each
correction term is a positive coroot.  These observations illustrate one difference between calculating the
codimensions of the Newton strata using the Cartan decomposition $G(F) = KTK$, where $K= G(\mathcal{O})$, and
the affine Bruhat decomposition $I\widetilde{W}I$.  The generic Newton slope sequence associated to a particular
double coset $K\pi^{\mu}K$ is always given by $-\mu_{\text{dom}}$, using the conventions of this paper, since
Mazur's inequality says that $\nu_x \leq -\mu_{\text{dom}}$, but we have $\pi^{\mu} \in K\pi^{\mu}K$. When using
the affine Bruhat decomposition, our calculations demonstrate that this initial guess for the generic Newton
slope sequence is not always correct.

\subsection{Poset of Newton slopes $\mathcal{N}(G)_x$ for $SL_3(F)$}\label{S:N(G)_x}

Similar calculations can be performed to determine the generic Newton slope sequences and the posets $\mathcal{N}(G)_x$ for $\mathbf{a}_x$ lying in the remaining Weyl
chambers.  The results depend very much on the Weyl chamber in consideration.

To compute the posets $\mathcal{N}(G)_x$ for all $x\in \widetilde{W}$, we proceed in several steps as outlined in the reduction arguments from Section \ref{S:reduction}.  First, compute $(xI)_{\leq \lambda}$ for each $x \in \widetilde{W}$ such that $\mathbf{a}_x \subset s_1(C^0)$ and $\mu_1 \geq 0$ as indicated in Section \ref{S:valpolysB}.  One immediate corollary of this calculation will be the generation of a list of generic Newton slope sequences associated to these alcoves.  Now, as in Section \ref{S:slopesC0}, we can explicitly find matrices that lie in $(xI)_{\lambda}$ for all $\lambda \leq \nu_x$.  Once we have the description for $\mathcal{N}(G)_x$, where $x$ satisfies $\mathbf{a}_x \subset s_1(C^0)$ and $\mu_1 \geq 0$, apply the reflection $\psi$ from Lemma \ref{T:reduction} that interchanges the two simple roots.  The result will be a description of $\mathcal{N}(G)_x$ for $x$ such that $\mathbf{a}_x \subset s_2(C^0)$ and $\mu_3 \leq 0$.  Now, if we apply to these alcoves the rotation by 120 degrees about the center of the base alcove, given by $\varphi$ in Lemma \ref{T:reduction}, this completes our description of $\mathcal{N}(G)_x$ for $x$ such that $\mathbf{a}_x \subset s_1(C^0)$.  Finally, using that $\varphi$ induces the identity on $\mathcal{N}(G)$, we can extend the results for the two adjacent Weyl chambers $C^0$ and $s_1(C^0)$ to the remaining four Weyl chambers.

We omit the details of these calculations, but we include the resulting descriptions for both $\nu_x$ and $\mathcal{N}(G)_x$ in Tables \ref{Ta:s_12} through \ref{Ta:s_2} at the conclusion of the paper.  Let $x = \pi^{\mu}w$, where $\mu=(\mu_1,\mu_2,\mu_3)$ and $w \in W$.  Recall from Theorem \ref{T:translation} that if $x = \pi^{\mu}$, then $\nu_x = -\mu_{\text{dom}}$ and $\mathcal{N}(G)_x = \{\nu_x\}$.  Since this poset is relatively uninteresting, we do not create a separate table for $w=1$.  For the other five cases, we organize the results according to the finite Weyl part $w$.  We first list the generic Newton slope sequences $\nu_x$. The posets $\mathcal{N}(G)_x$ then consist of all elements $\lambda \in \mathcal{N}(G)$ that satisfy the indicated properties, where we write $\nu_x = (\nu_1, \nu_2,\nu_3)$.  In addition, when expressing elements of $\widetilde{W}$ as products of the generators, we write $s_{121}$ for $s_1s_2s_1$, etc.

Inside the shrunken Weyl chambers, the
correction term for a given alcove, if any, is a coroot of the three forms appearing in Equation \eqref{E:roughgens}.  As we have seen, in the antidominant Weyl chamber, every value of $x$ gives rise to
a correction term, except the pure translation.  In all other Weyl chambers, this is not the case.  There are
increasingly fewer correction terms as the Weyl chambers get farther from the antidominant chamber.  In fact,
there are no necessary correction terms for the generic Newton slope sequences associated to alcoves in the
dominant Weyl chamber $s_{121}(C^0)$.  In the dominant Weyl chamber, the initial estimate $\nu_x = -\mu_{\text{dom}}$ is always correct.  We conjecture that this phenomenon is a general pattern (see Conjecture \ref{T:domconj}).  Other patterns among the descriptions for $\nu_x$ and $\mathcal{N}(G)_x$ for $SL_3(F)$ exist, but we cannot yet
formulate a precise conjecture of this sort.
 
\end{section}

\begin{section}{Proofs of Theorem \ref{T:main} and Corollary \ref{T:maincor}}\label{S:codimcalc}

\subsection{Proof of Theorem \ref{T:main} and Corollary \ref{T:maincor}, case A}\label{S:thmproof}
Fix $x=\pi^{\mu}w \in \widetilde{W}$, where $\mu=(\mu_1,\mu_2,\mu_3)$.  We choose an integer $N$ such that $(xI)_{\leq \lambda}=\rho_N^{-1}\rho_N(xI)_{\leq \lambda}$ for any  $\lambda \in
\mathcal{N}(G)_x$, where we recall from Section \ref{S:admis} that $\rho_N:xI \twoheadrightarrow xI/P^N$ is the map that truncates the power series entries of $xI$ at level $P^N$.  We explicitly describe the geometric structure of the closed subscheme $\rho_N(xI)_{\leq \lambda}$ in $xI/P^N$, and our description will yield a dimension formula.

\vskip 5 pt \textbf{Case A}: $\mathbf{a}_x \subset C^0$ and $\mu_2 \geq 0$.
\vskip 5 pt

We illustrate the argument by considering the specific example of case IA.  In this case, we argue that $\rho_N(xI)_{\leq \lambda}$ is actually a fiber bundle over an irreducible affine scheme, having irreducible fibers.  Let $x = \pi^{\mu}s_1s_2$, where $\mathbf{a}_x \subset C^0$ and $\mu_2 \geq 0$.  Let us begin by analyzing subcase ($i$).  If we further assume that $\mu_2 \neq \mu_3$, then $-\mu_3+1 \leq \lambda_3 \leq -\mu_2+1$, and recall from \eqref{E:123i} that we have the following description of $(xI)_{\leq \lambda}$: \begin{equation}\label{E:123calc} (xI)_{\leq \lambda} = \left\{ \begin{pmatrix} a& b& c\\d&e&f\\g&h&i \end{pmatrix} \in xI \biggm| a \in P^{-\lambda_1} \ \text{and}\ \begin{vmatrix} a & b\\ d& e\end{vmatrix} \in P^{\lambda_3}\right\}. \end{equation} Furthermore, by Equation \eqref{E:123gen}, we see that $\nu_x = (-\mu_1-1,-\mu_2,-\mu_3+1)$ in this case.  In general we denote $\nu_x=(\nu_1,\nu_2,\nu_3)$.

Let us roughly outline the structure of the argument.  The codimension of $(xI)_{\leq \lambda}$ in $xI$ will be
given by the codimension in \begin{equation*}X:= P^{\mu_1+1} \times P^{\mu_1+1} \times P^{\mu_2}_{\times} \times
P^{\mu_2}\end{equation*} of  \begin{equation*}X_{\leq \lambda}:= \{ (a,b,d,e) \in X \mid a \in P^{-\lambda_1},\ ae-bd
\in P^{\lambda_3}\}.\end{equation*}  Consider the projection map \begin{align*} p:X & \longrightarrow Y:=
P^{\mu_1+1}\times P^{\mu_2}_{\times} \times P^{\mu_2} \\ (a,b,d,e) & \mapsto (a,d,e) \phantom{ P^{\mu_1+1}\times
P^{\mu_2}_{\times} \times P^{\mu_2}}, \end{align*} and consider the restriction of $p$ to $X_{\leq \lambda}$
\begin{equation*}p_{\lambda}: X_{\leq \lambda}\rightarrow Y_{\leq \lambda} := P^{-\lambda_1}\times P^{\mu_2}_{\times}
\times P^{\mu_2}.\end{equation*}  Observe that $Y_{\leq \lambda}$ has codimension $-\lambda_1-(\mu_1+1)$ in $Y$.
Further, the fiber of $p_{\lambda}$ over the point $(a,d,e) \in Y_{\leq \lambda}$ is a coset of the form $aed^{-1} +
P^{\lambda_3-\mu_2}$, so that the fibers of $p_{\lambda}:X_{\leq \lambda} \rightarrow Y_{\leq \lambda}$ are cosets of
$P^{\lambda_3-\mu_2}$ in the fiber $P^{\mu_1+1}$ of the map $p:X \rightarrow Y$.  Therefore, $(xI)_{\leq
\lambda}$ is irreducible, having codimension given by  \begin{align*}\codim((xI)_{\leq \lambda} \subseteq xI) &=
\codim(X_{\leq \lambda} \subseteq X)\notag \\ \phantom{\codim((xI)_{\lambda} \subseteq xI)}&=  \lceil
-\lambda_1-(\mu_1+1)\rceil + \lceil(\lambda_3 - \mu_2)-(\mu_1+1)\rceil \notag\\
\phantom{\codim(X_{\leq \lambda} \subseteq X)} &= \lceil \nu_1-\lambda_1 \rceil + \lceil -\nu_3+\lambda_3\rceil
\notag \\ \phantom{\codim(X_{\leq \lambda} \subseteq X)}&= \sum_{i=1}^2 \lceil \langle \omega_i,\nu_x-
\lambda\rangle\rceil \\ \phantom{\codim(X_{\leq \lambda} \subseteq X)}&=
\length_{\mathcal{N}(G)_x}[\lambda,\nu_x].\end{align*} To obtain the last equality, we recall from Equation \eqref{E:N(G)123} that there is indeed a chain of the specified length in $\mathcal{N}(G)_x$.  

Further, since $(xI)_{\lambda}$ is open in $(xI)_{\leq \lambda}$, and since $(xI)_{\leq \lambda}$ is irreducible in $xI$, we see that the closure of the Newton stratum $(xI)_{\lambda}$ in $xI$ is precisely $(xI)_{\leq \lambda}$.  Therefore, if we have two Newton polygons $\lambda_1 < \lambda_2$ which are adjacent in the poset $\mathcal{N}(G)_x$, then we can compute that the codimension of the smaller Newton stratum in the closure of the larger is given by  \begin{align*}\codim((xI)_{\lambda_1} \subseteq (xI)_{\leq \lambda_2}) &=
\codim((xI)_{\leq \lambda_1} \subseteq (xI)_{\leq \lambda_2})\notag \\ \phantom{\codim((xI)_{\lambda_1} \subseteq (xI)_{\leq \lambda_2})}&= \codim((xI)_{\leq \lambda_1} \subseteq xI) - \codim((xI)_{\leq \lambda_2} \subseteq xI) \notag\\
\phantom{\codim((xI)_{\lambda_1} \subseteq (xI)_{\leq \lambda_2})} &=
\length_{\mathcal{N}(G)_x}[\lambda_1,\nu_x] - \length_{\mathcal{N}(G)_x}[\lambda_2,\nu_x] \notag\\
\phantom{\codim((xI)_{\lambda_1} \subseteq (xI)_{\leq \lambda_2})} &= 1.\end{align*} To make all of these arguments rigorous, we of course need to use the admissibility of $(xI)_{\leq
\lambda}$.  By choosing $N \geq \mu_3+1$ and replacing $P^{j}$ everywhere by $P^j/P^N$, the result will follow
for case IA, subcase ($i$) with $\mu_2 \neq \mu_3$.

Subcase ($ii$) and the case $\mu_2 = \mu_3$ for case IA are handled similarly, as are all other cases.  We mention that in the special case $\mu_2=\mu_3$, we have $\nu_3=-\mu_3+\frac{1}{2}$, and the value $\lceil \lambda_3-\nu_3\rceil$ yields an overestimate for the codimension.  In particular, for $\lambda_3=\nu_3 + \frac{1}{2}$, we have $\lceil \lambda_3-\nu_3\rceil = 1$, even though $\codim((xI)_{\leq \lambda}\subseteq xI) = \lceil \nu_1-\lambda_1 \rceil$ in this case.  Therefore, in case IA, subcase ($i$) with $\mu_2=\mu_3$, we have $\codim ((xI)_{\leq \lambda} \subseteq xI)= \left( \sum^2_{i=1}\lceil \langle \omega_i, \nu_x-\lambda \rangle\rceil \right)-1$.  We omit the details for the remaining arguments since they all proceed in a similar fashion, but we highlight some subtle differences among them.

\begin{lemma}\label{T:mult}
Consider the multiplication map given by \begin{align*} m: \mathcal{O}/P^n \times \mathcal{O}/ P^n & \rightarrow \mathcal{O}/ P^n \notag \\ (x,y) & \mapsto xy.\end{align*}  Suppose $z \in P^k_{\times}$ with $0\leq k<n$.  Then the fiber of $m$ over $z$ is of the form \begin{equation*}m^{-1}(z+P^n) = \coprod\limits_{0\leq j \leq k} S_j, \end{equation*} where each $S_j$ is a fiber bundle over $P^j_{\times}/P^n$, having fibers isomorphic to $P^{n-j}/P^n$.  If, on the other hand, $z \in P^n$, then the fiber of $m$ over $z$ is of the form \begin{equation*}m^{-1}(z+P^n) = \bigcup\limits_{0\leq j \leq n}\left( P^j/P^n \times P^{n-j}/P^n \right). \end{equation*} In either case, the fiber of $m$ over a point $z \in \mathcal{O}/P^n$ is a union of irreducible spaces of codimension $n$ in $\mathcal{O}/P^n \times \mathcal{O}/P^n$.
\end{lemma}

\begin{proof} Clear. \end{proof}

We now examine case IIIA, which proceeds in a slightly different manner than the other cases.  Let
$x=\pi^{\mu}s_1s_2s_1$ with $\mathbf{a}_x \subset C^0$ and $\mu_2 \geq 0$.  We begin with subcase ($i$), in which we recall that if $-\mu_3+1 \leq
\lambda_3 \leq -\mu_2$, then Equation \eqref{E:123calc} still describes $(xI)_{\leq \lambda}$.  Here, the main
difference is that we do not know the valuation of either $b$ or $d$.  Let  \begin{equation*}X:= P^{\mu_1+1}
\times P^{\mu_1+1} \times P^{\mu_2+1} \times P^{\mu_2}_{\times},\end{equation*} and let $X_{\leq \lambda}$ be defined
as in the previous argument.  Consider the projection map \begin{align*} p:X & \longrightarrow Z:=
P^{\mu_1+1}\times P^{\mu_2}_{\times} \\ (a,b,d,e) & \mapsto (a,e),\end{align*}  and consider the restriction of
$p$ to $X_{\leq \lambda}$ \begin{equation*}p_{\lambda}: X_{\leq \lambda}\rightarrow Z_{\leq \lambda} := P^{-\lambda_1}\times
P^{\mu_2}_{\times},\end{equation*}  Observe as before that $Z_{\leq \lambda}$ has codimension $-\lambda_1-(\mu_1+1)$
in $Z$. In this case, however, the fiber of $p_{\lambda}$ over the point $(a,e) \in Z_{\leq \lambda}$ is isomorphic to the fiber of the multiplication map $m$ from Lemma \ref{T:mult} over the point $ae + P^{\lambda_3}$, after appropriate scaling.   In particular, for any $(a,e) \in Z_{\leq \lambda}$ the fiber of $p_{\lambda}$ is reducible, and each irreducible component has codimension $\lambda_3+\mu_3-2$ in the fiber $P^{\mu_1+1}\times P^{\mu_2+1}$ of $p$ over $(a,e)\in Z$.  Now we define \begin{equation*}X_{\lambda}:= \{ (a,b,d,e) \in X \mid a \in P^{-\lambda_1}_{\times},\ ae-bd
\in P^{\lambda_3}_{\times}\}\subset X_{\leq \lambda},\end{equation*} where the reader will observe that while $X_{\lambda}$ consists of elements having Newton polygon $\lambda$, it is only a subspace of the Newton stratum corresponding to $\lambda$.  Nevertheless, to compute the closure of the Newton stratum, it suffices to work with the subspace $X_{\lambda}$, which justifies the notation.  If we define locally closed subsets of $X_{\lambda}$ and $X_{\leq \lambda}$ as follows: \begin{equation*}X^j_{\lambda}:= \{ (a,b,d,e) \in X \mid a \in P^{-\lambda_1}_{\times},\ ae-bd
\in P^{\lambda_3}_{\times},\ b \in P^j_{\times}\},\ \text{and}\end{equation*} \begin{equation*}X^j_{\leq \lambda}:= \{ (a,b,d,e) \in X \mid a \in P^{-\lambda_1},\ ae-bd
\in P^{\lambda_3},\ b\in P^j_{\times}\},\end{equation*} then we can write $X_{\leq \lambda}$ as a disjoint union of locally closed subsets of $X$: \begin{equation}\label{E:union}X_{\leq \lambda} = \coprod\limits_{j=\mu_1+1}^{-\lambda_1-1} X^j_{\leq \lambda} \quad \text{or} \quad X_{\leq \lambda} = \coprod\limits_{j=\mu_1+1}^{\lambda_3-\mu_2-1} X^j_{\leq \lambda}.\end{equation} Here, the two cases correspond to the two possibilities in Lemma \ref{T:mult}, where we must consider whether $\val(ae) < \lambda_3$ or $\val(ae)\geq \lambda_3$.  Therefore, we see that there are either $\lceil -\lambda_1-\mu_1-1\rceil$ or $\lceil \lambda_3+\mu_3-1\rceil$ components, depending on which of these integers is smaller.  Furthermore, each of the components $X^j_{\leq \lambda}$ are irreducible in $X$ by our discussion above, and are therefore precisely the closures of the $X^j_{\lambda}$, respectively.  Since the union in \eqref{E:union} is finite in either case, we see that the closure of $X_{\lambda}$ is $X_{\leq \lambda}$, and therefore the closure of $(xI)_{\lambda}$ is $(xI)_{\leq \lambda}$.  The codimension calculations proceed as before, as does the means by which we can make this argument rigorous.  The reader should also note that case IIIA, subcase ($ii$) should be handled in the same manner as subcase ($i$).  In particular, in case IIIA, the scheme $(xI)_{\leq \lambda}$ is reducible for all $\lambda \leq -\mu-(2,0,-2)$.

We point out one further difference in case IIIA.  Consider the
point $(a_{\mu_1+1},e_{\mu_2}) \in Z_{\leq \lambda}$, where $a_i$ and $e_i$ denote the coefficients of $\pi^i$ in $a$ and $e$, respectively.  Since in case IIIA we have
$\min\{\val(ae)\}=\mu_1+\mu_2+1$ and $\min\{\val(bd)\}=\mu_1+\mu_2+2$, we see that the fiber over
this point is empty.  In fact, as we saw in the proof of Proposition \ref{T:N(G)13}, this phenomenon occurs whenever $\lambda<\nu_x=-\mu-(1,0,-1)$ and $\lambda \nleq
- \mu-(2,0,-2)$. 

\subsection{Proof of Theorem \ref{T:main} and Corollary \ref{T:maincor}, case B}\label{S:corproof}

Recall from Remark \ref{T:valpolysB} that the codimensions in case B agree with the codimensions of the
Newton strata in $x'I'$, where $x'=\pi^{{\mu}_{\text{dom}^*}}(s_1^{-1}ws_1)$ and $I'=s_1^{-1}Is_1$.  Note that
the theorem once more holds trivially in case VI, since for $w=1$, we have $w':=s_1^{-1}ws_1=1$, and so $\mathcal{N}(G)_x =
\{-\mu_{\text{dom}}\}$.  A routine calculation computes $x'I'$ for each of the five remaining values of $w' \in W$:
\begin{align*}
(\text{I})\ \ &x'I' = \begin{pmatrix} P^{\mu_2}&P^{\mu_2}_{\times}&P^{\mu_2}\\P^{\mu_1+1}&P^{\mu_1+1}&P^{\mu_1}_{\times} \\P^{\mu_3}_{\times}&P^{\mu_3+1}&P^{\mu_3}\end{pmatrix} & (\text{IV})\ \ &x'I' = \begin{pmatrix} P^{\mu_2}&P^{\mu_2}_{\times}&P^{\mu_2}\\P^{\mu_1}_{\times}&P^{\mu_1+1}&P^{\mu_1} \\P^{\mu_3+1}&P^{\mu_3+1}&P^{\mu_3}_{\times}\end{pmatrix}\\
(\text{II})\ \ &x'I' = \begin{pmatrix} P^{\mu_2+1}&P^{\mu_2+1}&P^{\mu_2}_{\times}\\P^{\mu_1}_{\times}&P^{\mu_1+1}&P^{\mu_1} \\P^{\mu_3}&P^{\mu_3}_{\times}&P^{\mu_3}\end{pmatrix} & (\text{V})\ \ &x'I' = \begin{pmatrix} P^{\mu_2+1}&P^{\mu_2+1}&P^{\mu_2}_{\times}\\P^{\mu_1}&P^{\mu_1}_{\times}&P^{\mu_1} \\P^{\mu_3}_{\times}&P^{\mu_3+1}&P^{\mu_3}\end{pmatrix}.\\
(\text{III})\ \ &x'I' = \begin{pmatrix} P^{\mu_2}_{\times}&P^{\mu_2+1}&P^{\mu_2}\\P^{\mu_1+1}&P^{\mu_1+1}&P^{\mu_1}_{\times} \\P^{\mu_3}&P^{\mu_3}_{\times}&P^{\mu_3}\end{pmatrix} & \phantom{}
\end{align*}
Arguments similar to those used in case A apply to show that $(x'I')_{\leq
\lambda'}$ is admissible and has the structure of a fiber bundle over an irreducible base space, having non-empty fibers over every point.  Note that case VB should be handled in the same way as IIIA, in which the fibers are reducible and look like fibers of the requisite analog of the multiplication map from Lemma \ref{T:mult}.

Using the arguments outlined above, the reader can check that for any $x \in \widetilde{W}$ satisfying the
conditions of cases A or B, we have \begin{equation*} \codim\left( (xI)_{\leq \lambda} \subseteq  xI\right)  =
\length_{\mathcal{N}(G)_x}[\lambda,\nu_x].\end{equation*}  Lemmas \ref{T:reduction} and \ref{T:coset} then imply
that for all $x \in \widetilde{W}$, we have \begin{equation*} \codim\left( (IxI)_{\leq \lambda} \subseteq  IxI\right)  =
\length_{\mathcal{N}(G)_x}[\lambda,\nu_x]. \end{equation*}  We have also shown that for any $x \in \widetilde{W}$ satisfying the conditions of cases A or B, the closure of $(xI)_{\lambda}$ in $xI$ is precisely $(xI)_{\leq \lambda}$, in which case the codimensions between adjacent strata always equals 1.  Finally, observe that the proofs of Lemmas \ref{T:reduction} and \ref{T:coset} enable us to extend these observations about the closures of the Newton strata $(xI)_{\lambda}$ to $(IxI)_{\lambda}$ and finally to all $x\in \widetilde{W}$.  Therefore, Theorem \ref{T:main} holds for any $x \in \widetilde{W}$.

The reader will verify the remaining root-theoretic versions of the codimension formula provided in Corollary \ref{T:maincor} during the course of the proof of Theorem \ref{T:main}.  Extend the analysis for cases A and B to the rest of $C^0$ and $s_1(C^0)$ by applying the reflection $\psi$ that interchanges the two simple roots, discussed in Lemma \ref{T:reduction}.  We may then extend our calculations to the remaining Weyl chambers by applying the rotations $\varphi$ and $\varphi^2$ to the Weyl chambers $C^0$ and $s_1(C^0)$, where $\varphi$ is the rotation by 120 degrees about the center of the base alcove defined in Lemma \ref{T:reduction}.
\end{section}

\newpage

\begin{table}[h]
  \begin{center}\renewcommand{\arraystretch}{1.25}
     \begin{tabular}{| c | l | l |} \hline
        Weyl chamber & $\nu_x$ & $\mathcal{N}(G)_x$ \\ \hline\hline
          $C^0$   & $-(\mu+(1, -\frac{1}{2},-\frac{1}{2}))_{\text{dom}},\ $ if $\mu_2 = \mu_3$ & $\{\lambda \leq \nu_x\}$ \\ \cline{2-3}
                  & $-(\mu+(1,0,-1))_{\text{dom}},\ $ otherwise & $\{ \lambda \leq \nu_x\}$ \\ \hline\hline
          $s_1(C^0)$ & $-(\mu + (\frac{1}{2},0, -\frac{1}{2}))_{\text{dom}},\ $ if $\mu_1+1=\mu_3$ &  $\{\lambda \leq \nu_x\}$ \\ \cline{2-3}
               & $-(\mu +(1,0,-1))_{\text{dom}},\ $ otherwise  & $\{ \lambda \leq \nu_x\}$ \\ \hline\hline
          $s_2(C^0)$ & $-(\mu+(1,-1,0))_{\text{dom}}$ & $\{ \lambda \leq \nu_x\}$\\ \hline\hline
          $s_{12}(C^0)$ & $-\mu_{\text{dom}}$ &  $\{ \lambda \leq \nu_x\}$\\ \hline\hline
          $s_{21}(C^0)$   & $-(\mu+(\frac{1}{2},-\frac{1}{2},0))_{\text{dom}},\ $ if $\mu_1+1=\mu_2$ &$\{ \lambda \leq \nu_x\}$ \\ \cline{2-3}
                  & $-(\mu+(1,-1,0))_{\text{dom}},\ $ otherwise & $\{ \lambda \leq \nu_x\}$ \\ \hline\hline
          $s_{121}(C^0)$ & $-\mu_{\text{dom}}$ & $\{ \lambda \leq \nu_x\}$ \\ \hline
     \end{tabular}
    \caption{$w = s_{12}$}\label{Ta:s_12}
  \end{center}
\end{table}

\begin{table}[h]
  \begin{center}\renewcommand{\arraystretch}{1.25}
     \begin{tabular}{| c | l | l |} \hline
        Weyl chamber & $\nu_x$ & $\mathcal{N}(G)_x$ \\ \hline\hline
          $C^0$   & $-(\mu+(\frac{1}{2},\frac{1}{2},-1))_{\text{dom}},\ $ if $\mu_1 = \mu_2$ & $\{\lambda \leq \nu_x\}$ \\ \cline{2-3}
                  & $-(\mu+(1,0,-1))_{\text{dom}},\ $ otherwise & $\{ \lambda \leq \nu_x\}$ \\ \hline\hline
           $s_1(C^0)$ & $-(\mu+(0,1,-1))_{\text{dom}}$ & $\{ \lambda \leq \nu_x\}$\\ \hline\hline
          $s_2(C^0)$ & $-(\mu + (\frac{1}{2},0, -\frac{1}{2}))_{\text{dom}},\ $ if $\mu_1+1=\mu_3$ &  $\{\lambda \leq \nu_x\}$ \\ \cline{2-3}
               & $-(\mu +(1,0,-1))_{\text{dom}},\ $ otherwise  & $\{ \lambda \leq \nu_x\}$ \\ \hline\hline
           $s_{12}(C^0)$   & $-(\mu+(0,\frac{1}{2},-\frac{1}{2}))_{\text{dom}},\ $ if $\mu_2+1=\mu_3$ &$\{ \lambda \leq \nu_x\}$ \\ \cline{2-3}
                  & $-(\mu+(0,1,-1))_{\text{dom}},\ $ otherwise & $\{ \lambda \leq \nu_x\}$ \\ \hline\hline
          $s_{21}(C^0)$ & $-\mu_{\text{dom}}$ &  $\{ \lambda \leq \nu_x\}$\\ \hline\hline
          $s_{121}(C^0)$ & $-\mu_{\text{dom}}$ & $\{ \lambda \leq \nu_x\}$ \\ \hline
     \end{tabular}
    \caption{$w = s_{21}$}\label{Ta:s_21}
  \end{center}
\end{table}

\clearpage

\begin{table}[h]
  \begin{center}\renewcommand{\arraystretch}{1.25}
     \begin{tabular}{| c | l | l |} \hline
        Weyl chamber & $\nu_x$ & $\mathcal{N}(G)_x$ \\ \hline\hline
          $C^0$   & $-(\mu+(1,0,-1))_{\text{dom}}$ & $\{(\nu_1, -\frac{\nu_1}{2}, -\frac{\nu_1}{2}) \leq \lambda \leq \nu_x\},\ $ if $\mu_1+1=\mu_2$ \\ \cline{3-3}
          &   &    $\{(-\frac{\nu_3}{2}, -\frac{\nu_3}{2}, \nu_3) \leq \lambda \leq \nu_x\},\ $ if $\mu_2+1=\mu_3$  \\ \cline{3-3}
          &    &  $\{\nu_x\} \cup \{\lambda \leq \nu_x-(1,0,-1)\},\ $ otherwise   \\ \hline\hline
          $s_1(C^0)$ & $-(\mu+(\frac{1}{2}, 0, -\frac{1}{2}))_{\text{dom}},\ $ if $\mu_1+1=\mu_3$ & $\{\nu_x\}$ \\ \cline{2-3}
                 &  $-(\mu+(1,0,-1))_{\text{dom}},\ $ otherwise & $\{(\nu_1,-\frac{\nu_1}{2}, -\frac{\nu_1}{2}) \leq \lambda \leq \nu_x\}$  \\ \hline\hline
          $s_2(C^0)$ &  $-(\mu+(\frac{1}{2}, 0, -\frac{1}{2}))_{\text{dom}},\ $ if $\mu_1+1=\mu_3$ &$\{\nu_x\}$ \\ \cline{2-3}
                     &  $-(\mu+(1,0,-1))_{\text{dom}},\ $ otherwise  &  $\{(-\frac{\nu_3}{2}, -\frac{\nu_3}{2},\nu_3) \leq \lambda \leq \nu_x\} $ \\ \hline\hline
          $s_{12}(C^0)$ & $-\mu_{\text{dom}}$ &  $\{(\nu_1, -\frac{\nu_1}{2}, -\frac{\nu_1}{2}) \leq \lambda \leq \nu_x\}$ \\ \hline\hline
          $s_{21}(C^0)$   &$-\mu_{\text{dom}}$ & $\{(-\frac{\nu_3}{2},-\frac{\nu_3}{2},\nu_3)\leq \lambda \leq \nu_x\}$ \\ \hline\hline                      $s_{121}(C^0)$ & $-\mu_{\text{dom}}$ & $\{\lambda \leq \nu_x\}$ \\ \hline
     \end{tabular}
    \caption{$w = s_{121}$}\label{Ta:s_121}
  \end{center}
\end{table}

\begin{table}[h]
  \begin{center}\renewcommand{\arraystretch}{1.25}
     \begin{tabular}{| c | l | l |} \hline
        Weyl chamber & $\nu_x$ & $\mathcal{N}(G)_x$ \\ \hline\hline
          $C^0$   & $-(\mu+(\frac{1}{2}, -\frac{1}{2}, 0))_{\text{dom}},\ $ if $\mu_1+1 = \mu_2$ & $\{\nu_x\}$ \\ \cline{2-3}
                  & $-(\mu+(1,-1,0))_{\text{dom}},\ $ otherwise & $\{(-\frac{\nu_3}{2}, -\frac{\nu_3}{2}, \nu_3) \leq \lambda \leq \nu_x\}$  \\ \hline\hline
          $s_1(C^0)$ & $-\mu_{\text{dom}}$ &  $\{(-\frac{\nu_3}{2}, -\frac{\nu_3}{2}, \nu_3) \leq \lambda \leq \nu_x\}$ \\ \hline\hline
          $s_2(C^0)$ & $-(\mu+(1,-1,0))_{\text{dom}}$ & $\{(\nu_1, -\frac{\nu_1}{2}, -\frac{\nu_1}{2}) \leq \lambda \leq \nu_x\},\ $ if $\mu_1=\mu_3$ \\ \cline{3-3}
                     &   & $\{\lambda \leq \nu_x\},\ $ otherwise \\ \hline\hline
          $s_{12}(C^0)$ & $-\mu_{\text{dom}}$ &  $\{(\nu_1, -\frac{\nu_1}{2}, -\frac{\nu_1}{2}) \leq \lambda \leq \nu_x\},\ $ if $\mu_2=\mu_3$ \\ \cline{3-3}
                        &   & $\{\lambda \leq \nu_x\},\ $ otherwise  \\ \hline\hline
          $s_{21}(C^0)$   & $-(\mu+(\frac{1}{2},-\frac{1}{2},0))_{\text{dom}},\ $ if $\mu_1+1=\mu_2$ & $\{\nu_x\}$ \\ \cline{2-3}
                  & $-(\mu+(1,-1,0))_{\text{dom}},\ $ otherwise & $\{(\nu_1, -\frac{\nu_1}{2}, -\frac{\nu_1}{2}) \leq \lambda \leq \nu_x\}$ \\ \hline\hline
          $s_{121}(C^0)$ & $-\mu_{\text{dom}}$ & $\{(\nu_1, -\frac{\nu_1}{2}, -\frac{\nu_1}{2}) \leq \lambda \leq \nu_x\}$ \\ \hline
     \end{tabular}
    \caption{$w = s_1$}\label{Ta:s_1}
  \end{center}
\end{table}

\clearpage

\begin{table}[h]
  \begin{center}\renewcommand{\arraystretch}{1.25}
     \begin{tabular}{| c | l | l |} \hline
        Weyl chamber & $\nu_x$ & $\mathcal{N}(G)_x$ \\ \hline\hline
          $C^0$   & $-(\mu+(0,\frac{1}{2}, -\frac{1}{2}))_{\text{dom}},\ $ if $\mu_2+1 = \mu_3$ & $\{\nu_x\}$ \\ \cline{2-3}
                  & $-(\mu+(0,1,-1))_{\text{dom}},\ $ otherwise & $\{(\nu_1, -\frac{\nu_1}{2}, -\frac{\nu_1}{2}) \leq \lambda \leq \nu_x\}$  \\ \hline\hline
          $s_1(C^0)$ & $-(\mu+(0,1,-1))_{\text{dom}}$ &  $\{(-\frac{\nu_3}{2}, -\frac{\nu_3}{2}, \nu_3) \leq \lambda \leq \nu_x\},\ $ if $\mu_1=\mu_3$ \\ \cline{3-3}
                     &     &  $\{\lambda \leq \nu_x\},\ $ otherwise    \\ \hline\hline
          $s_2(C^0)$ & $-\mu_{\text{dom}}$ & $\{(\nu_1, -\frac{\nu_1}{2}, -\frac{\nu_1}{2}) \leq \lambda \leq \nu_x\}$\\  \hline\hline
          $s_{12}(C^0)$ & $-(\mu+(0,\frac{1}{2},-\frac{1}{2}))_{\text{dom}},\ $ if $\mu_2+1=\mu_3$ &  $\{\nu_x\}$ \\ \cline{2-3}
                        & $-(\mu+(0,1,-1))_{\text{dom}},\ $ otherwise  & $\{(-\frac{\nu_3}{2},-\frac{\nu_3}{2},\nu_3) \leq \lambda \leq \nu_x\}$  \\ \hline\hline
          $s_{21}(C^0)$   &  $-\mu_{\text{dom}}$ & $\{(-\frac{\nu_3}{2},-\frac{\nu_3}{2},\nu_3) \leq \lambda \leq \nu_x\},\ $ if $\mu_1=\mu_2$ \\ \cline{3-3}
                         &   &  $\{ \lambda \leq \nu_x\},\ $ otherwise   \\ \hline\hline
          $s_{121}(C^0)$ & $-\mu_{\text{dom}}$ & $\{(-\frac{\nu_3}{2}, -\frac{\nu_3}{2}, \nu_3) \leq \lambda \leq \nu_x\}$ \\ \hline
     \end{tabular}
    \caption{$w = s_2$}\label{Ta:s_2}
  \end{center}
\end{table}

\bibliographystyle{amsplain}
\bibliography{references}

\providecommand{\bysame}{\leavevmode\hbox to3em{\hrulefill}\thinspace}
\providecommand{\MR}{\relax\ifhmode\unskip\space\fi MR }
\providecommand{\MRhref}[2]{%
  \href{http://www.ams.org/mathscinet-getitem?mr=#1}{#2}
}
\providecommand{\href}[2]{#2}
\begin{thebibliography}{10}

\bibitem{AG}
F.~Adreatta and E.~Goren, \emph{Hilbert modular varieties of low dimension},
  Geometric aspects of Dwork theory, vol.~I, Walter de Gruyter GmbH \& Co.,
  Berlin, 2004, pp.~113--175.

\bibitem{Bar}
I.~Barsotti, \emph{Analytical methods for abelian varieties in positive
  characteristic}, Colloq. Th{\'{e}}orie des Groupes Alg{\'{e}}briques
  (Bruxelles), 1962, pp.~77--85.

\bibitem{BF}
R.~Blache and R.~F{\'{e}}rard, \emph{Newton stratification for polynomials: the
  open stratum}, J. Number Theory \textbf{123} (2007), no.~2, 456--472.

\bibitem{BW}
O.~B{\"{u}}ltel and T.~Wedhorn, \emph{Congruence relations for {Shimura}
  varieties associated to some unitary groups}, J. Inst. Math. Jussieu
  \textbf{5} (2006), no.~2, 229--261.

\bibitem{Ch}
C.-L. Chai, \emph{Newton polygons as lattice points}, Amer. J. Math.
  \textbf{122} (2000), no.~5, 967--990.

\bibitem{dJO}
A.~de~Jong and F.~Oort, \emph{Purity of the stratification by {Newton}
  polygons}, J. Amer. Math. Soc. \textbf{13} (2000), no.~1, 209--241.

\bibitem{Dem}
M.~Demazure, \emph{Lectures on $p$-divisible groups}, Lecture notes in
  mathematics, vol. 302, Springer-Verlag, Berlin-New York, 1972.

\bibitem{GO}
E.~Goren and F.~Oort, \emph{Stratifications of {Hilbert} modular varieties}, J.
  Algebraic Geom. \textbf{9} (2000), no.~1, 111--154.

\bibitem{GKM}
M.~Goresky, R.~Kottwitz, and R.~MacPherson, \emph{Codimensions of root
  valuation strata}, math.RT/0601197, 2006.

\bibitem{GHKR}
U.~G{\"{o}}rtz, T.~Haines, R.~Kottwitz, and D.~Reuman, \emph{Dimensions of some
  affine {Deligne}-{Lusztig} varieties}, Ann. Sci. {\'{E}}cole Norm. Sup. (4)
  \textbf{39} (2006), no.~3, 467--511.

\bibitem{GHKRadlvs}
\bysame, \emph{Affine {Deligne}-{Lusztig} varieties in affine flag varieties},
  math.AG/0805.0045v1, 2008.

\bibitem{Gro}
A.~Grothendieck, \emph{Groupes de barsotti-tate et cristaux de
  {Dieudonn{\'{e}}}}, S{\'{e}}minaire de Math{\'{e}}matiques Sup{\'{e}}rieures
  {\'{E}}t{\'{e}} 1970 (Montr{\'{e}}al, Que.), no.~45, Les Presses de
  l'Universit{\'{e}} de Montr{\'{e}}al, 1974.

\bibitem{Hai}
T.~Haines, \emph{Introduction to {Shimura} varieties with bad reduction of
  parahoric type}, Harmonic analysis, the trace formula, and {Shimura}
  varieties (Providence, RI), Clay Math. Proc. (4), Amer. Math. Soc., 2005,
  pp.~583--642.

\bibitem{Har}
S.~Harashita, \emph{Ekedahl-{Oort} strata and the first {Newton} slope strata},
  J. Algebraic Geom. \textbf{16} (2007), no.~1, 171--199.

\bibitem{HT}
M.~Harris and R.~Taylor, \emph{The geometry and cohomology of some simple
  {Shimura} varieties}, Annals of Mathematics Studies, vol. 151, Princeton
  University Press, Princeton, NJ, 2001.

\bibitem{Kat}
N.~M. Katz, \emph{Slope filtration of {$F$}-crystals}, Journ{\'{e}}es de
  G{\'{e}}om{\'{e}}trie Alg{\'{e}}brique de Rennes (Rennes, 1978) Vol. I,
  Ast{\'{e}}risque, No. 63, Soci{\'{e}}t{\'{e}} Math{\'{e}}matique de France,
  Paris, 1979, pp.~113--163.

\bibitem{Kob}
N.~Koblitz, \emph{$p$-adic variation of the zeta-function over families of
  varieties defined over finite fields}, Compositio Math. \textbf{31} (1975),
  no.~2, 119--218.

\bibitem{KotIsoI}
R.~Kottwitz, \emph{Isocrystals with additional structure}, Compositio Math.
  \textbf{56} (1985), no.~2, 201--220.

\bibitem{KotIsoII}
\bysame, \emph{Isocrystals with additional structure {II}}, Compositio Math.
  \textbf{109} (1997), no.~3, 255--339.

\bibitem{KotNewtStrata}
\bysame, \emph{Dimensions of {Newton} strata in the adjoint quotient of
  reductive groups}, Pure Appl. Math. Q. \textbf{2} (2006), no.~3, 817--836.

\bibitem{LiO}
K.-Z. Li and F.~Oort, \emph{Moduli of supersingular abelian varieties}, Lecture
  notes in mathematics, vol. 1680, Springer-Verlag, Berlin, 1998.

\bibitem{Man}
J.~I. Manin, \emph{Theory of commutative formal groups over fields of finite
  characteristic}, Uspehi Mat. Nauk \textbf{18} (1963), no.~6 (114), 3--90.

\bibitem{NO}
P.~Norman and F.~Oort, \emph{Moduli of abelian varities}, Ann. of Math. (2)
  \textbf{112} (1980), no.~3, 413--439.

\bibitem{Omoduli&NP}
F.~Oort, \emph{Moduli of abelian varities and {Newton} polygons}, C. R. Acad.
  Sci. Paris S{\'{e}}r. {I} Math. \textbf{312} (1991), no.~5, 385--389.

\bibitem{Omoduliposchar}
\bysame, \emph{Moduli of abelian varieties in positive characteristic},
  Barsotti Symposium in Algebraic Geometry (Abano Terme, 1991), Perspect.
  Math., vol.~15, Academic Press, San Diego, CA, 1994, pp.~253--276.

\bibitem{Onp&formalgps}
\bysame, \emph{Newton polygons and formal groups: conjectures by {Manin} and
  {Grothendieck}}, Ann. of Math. (2) \textbf{152} (2000), no.~1, 183--206.

\bibitem{Onpinmoduli}
\bysame, \emph{Newton polygon strata in the moduli space of abelian varieties},
  Moduli of abelian varieties (Texel Island, 1999), Progr. Math., vol. 195,
  Birkh{\"{a}}user, Basel, 2001, pp.~417--440.

\bibitem{Ofoliations}
\bysame, \emph{Foliations in moduli spaces of abelian varieties}, J. Amer.
  Math. Soc. \textbf{17} (2004), no.~2, 267--296.

\bibitem{Onp&pdiv}
\bysame, \emph{Newton polygons and p-divisible groups: a conjecture by
  {Grothendieck}}, Automorphic forms {I}, Ast{\'{e}}risque, No. 298,
  Soci{\'{e}}t{\'{e}} Math{\'{e}}matique de France, Paris, 2005, pp.~255--269.

\bibitem{RapNewton}
M.~Rapoport, \emph{On the {Newton} stratification}, S{\'{e}}minaire Bourbaki,
  Vol. 2001/2002, Ast{\'{e}}risque, No. 290, Soci{\'{e}}t{\'{e}}
  Math{\'{e}}matique de France, Paris, 2003, pp.~207--224.

\bibitem{RapShimura}
\bysame, \emph{A guide to the reduction modulo $p$ of {Shimura} varieties},
  Automorphic forms {I}, Ast{\'{e}}risque, No. 298, Soci{\'{e}}t{\'{e}}
  Math{\'{e}}matique de France, Paris, 2005, pp.~271--318.

\bibitem{RR}
M.~Rapoport and M.~Richartz, \emph{On the classification and specialization of
  {$F$}-isocrystals with additional structure}, Compositio Math. \textbf{103}
  (1996), no.~2, 153--181.

\bibitem{Reu}
D.~Reuman, \emph{Formulas for the dimensions of some affine {Deligne}-{Lusztig}
  varieties}, Michigan Math. J. \textbf{52} (2004), no.~2, 435--451.

\bibitem{Ta}
J.~Tate, \emph{Classes d'isog{\'{e}}nie de vari{\'{e}}t{\'{e}}s
  ab{\'{e}}liennes sur un corps fini (d'apr{\`{e}}s {T}. {Honda})}, Lecture
  notes in mathematics, S{\'{e}}minaire Bourbaki, Vol. 1968/69: Exp. 347-363,
  vol. 179, Springer-Verlag, Berlin-New York, 1971.

\bibitem{vdGO}
G.~van~der Geer and F.~Oort, \emph{Moduli of abelian varieties: a short
  introduction and survey}, Moduli of curves and abelian varieties, Aspects
  Math. E33, Vieweg, Braunschweig, 1999, pp.~1--21.

\bibitem{Vas}
A.~Vasiu, \emph{Crystalline boundedness principle}, Ann. Sci. {\'{E}}cole Norm.
  Sup. (4) \textbf{39} (2006), no.~2, 245--300.

\bibitem{VieDim}
E.~Viehmann, \emph{The dimensions of some affine {Deligne}-{Lusztig}
  varieties}, Ann. Sci. {\'{E}}cole Norm. Sup. (4) \textbf{39} (2006), no.~3,
  513--526.

\bibitem{VieConncpt}
\bysame, \emph{Connected components of closed affine {Deligne}-{Lusztig}
  varieties}, Math. Ann. \textbf{340} (2007), no.~2, 315--333.

\bibitem{VieGlobal}
\bysame, \emph{The global structure of moduli spaces of polarized p-divisible
  groups}, math.AG/0703841v1, 2007.

\bibitem{VieModuli}
\bysame, \emph{Moduli spaces of $p$-divisible groups}, J. Algebraic Geom.
  \textbf{17} (2008), no.~2, 341--374.

\bibitem{WedOrdinariness}
T.~Wedhorn, \emph{Ordinariness in good reductions of {Shimura} varieties of
  {P}{E}{L}-type}, Ann. Sci. {\'{E}}cole Norm. Sup. (4) \textbf{32} (1999),
  no.~5, 575--618.

\bibitem{WedDim}
\bysame, \emph{The dimension of {Oort} strata of {Shimura} varieties of
  {P}{E}{L}-type}, Moduli of abelian varieties (Texel Island, 1999), Progr.
  Math., vol. 195, Birkh{\"{a}}user, Basel, 2001, pp.~441--471.

\bibitem{WedCong}
\bysame, \emph{Congruence relations in the {Siegel} case}, Functional analysis,
  {V}{I}{I} ({Dubrovnik}, 2001), Various Publ. Ser. (Aarhus), vol.~46, Univ.
  Aarhus, Aarhus, 2002, pp.~193--217.

\bibitem{Weil}
A.~Weil, \emph{Vari{\'{e}}t{\'{e}}s ab{\'{e}}liennes et courbes
  alg{\'{e}}briques}, Actualit{\'{e}}s Sci. Ind., no. 1064, Publ. Inst. Math.
  Univ. Strasbourg 8 (1946), Hermann, Paris, 1948.

\bibitem{Yu}
C.-F. Yu, \emph{On the slope stratification of certain {Shimura} varieties},
  Math. Zeit. \textbf{251} (2005), no.~4, 859--873.

\end{thebibliography}

\end{document}